\documentclass[11pt]{article} 
\usepackage{algorithmic}
\usepackage{algorithm}
\usepackage{times} 
\usepackage{amsfonts}
\usepackage{amssymb}
\usepackage{amsmath} 
\usepackage{amsthm} 
\usepackage{latexsym} 
\usepackage{epsfig}
\usepackage{latexsym} 
\usepackage{verbatim}
\usepackage{makeidx}
\usepackage{color}

\makeindex
\begin{document}
 
%%%%%%%%%%%%%%%%%%%%%%%%%%%%%%%%% 
%  Enviroments 
\newtheorem{theorem}{Theorem}
\newtheorem{corollary}[theorem]{Corollary}
\newtheorem{prop}[theorem]{Proposition} 
\newtheorem{problem}[theorem]{Problem}
\newtheorem{lemma}[theorem]{Lemma} 
\newtheorem{remark}[theorem]{Remark}
\newtheorem{observation}[theorem]{Observation}
\newtheorem{defin}{Definition} 
\newtheorem{example}[theorem]{Example}
\newtheorem{conj}{Conjecture} 
\newcommand{\PR}{\noindent {\bf Proof:\ }} % beginning of a
% proof
\def\EPR{\hfill $\Box$\linebreak\vskip.5mm} % end of a proof
 
%%%%%%%%%%%%%%%%%%%%%%%%%%%%%%%%%%% 
%  Operators 
\def\Pol{{\sf Pol}} 
\def\mPol{{\sf MPol}} 
\def\Polo{{\sf Pol}_1} 
\def\PPol{{\sf pPol\;}} 
\def\Inv{{\sf Inv}}
\def\mInv{{\sf MInv}} 
\def\Clo{{\sf Clo}\;} 
\def\Con{{\sf Con}} 
\def\concom{{\sf Concom}\;} 
\def\End{{\sf End}\;}
\def\Sub{{\sf Sub}\;} 
\def\Im{{\sf Im}} 
\def\Ker{{\sf Ker}\;} 
\def\Bl{{\sf Block}}
\def\H{{\sf H}}
\def\S{{\sf S}} 
\def\D{{\sf P}} 
\def\I{{\sf I}} 
\def\Var{{\sf var}} 
\def\PVar{{\sf pvar}} 
\def\fin#1{{#1}_{\rm fin}}
\def\P{{\sf P}} 
\def\Pfin{{\sf P_{\rm fin}}} 
\def\Id{{\sf Id}}
\def\R{{\rm R}} 
\def\F{{\rm F}} 
\def\Term{{\sf Term}}
\def\var#1{{\sf var}(#1)} 
\def\Sg#1{{\sf Sg}(#1)} 
\def\Sgg#1#2{{\sf Sg}_{#1}(#2)} 
\def\Cg#1{{\sf Cg}(#1)}
\def\Cgg#1#2{{\sf Cg}_{#1}(#2)} 
\def\Cen{{\sf Cen}}
\def\tol{{\sf tol}} 
\def\lnk{{\sf lk}} 
\def\rbcomp#1{{\sf rbcomp}(#1)}
  
%%%%%%%%%%%%%%%%%%%%%%%%%%%%%%%%%% 
%  Operations  
\let\cd=\cdot 
\let\eq=\equiv 
\let\op=\oplus 
\let\omn=\ominus
\let\meet=\wedge 
\let\join=\vee 
\let\tm=\times
\def\ldiv{\mathbin{\backslash}} 
\def\rdiv{\mathbin/}
  
%%%%%%%%%%%%%%%%%%%%%%%%%%%%%%%%%% 
%  Tame congruence  
\def\typ{{\sf typ}} 
\def\zz{{\un 0}} 
\def\zo{{\un 1}}
\def\one{{\bf1}} 
\def\two{{\bf2}} 
\def\three{{\bf3}}
\def\four{{\bf4}} 
\def\five{{\bf5}}
\def\pq#1{(\zz_{#1},\mu_{#1})}
  
%%%%%%%%%%%%%%%%%%%%%%%%%%%%%%%%%%% 
%  Accents and so  
\let\wh=\widehat 
\def\ox{\ov x} 
\def\oy{\ov y} 
\def\oz{\ov z}
\def\of{\ov f} 
\def\oa{\ov a} 
\def\ob{\ov b} 
\def\oc{\ov c}
\def\od{\ov d} 
\def\oob{\ov{\ov b}} 
\def\rx{{\rm x}}
\def\rf{{\rm f}} 
\def\rrm{{\rm m}} 
\let\un=\underline
\let\ov=\overline 
\let\cc=\circ 
\let\rb=\diamond 
\def\ta{{\tilde a}} 
\def\tz{{\tilde z}}
\let\td=\tilde
\let\dg=\dagger
\let\ddg=\ddagger
  
%%%%%%%%%%%%%%%%%%%%%%%%%%%%%%%% 
% Abbreviations for algebras and clones
  
\def\zZ{{\mathbb Z}} 
\def\B{{\mathcal B}} 
\def\P{{\mathcal P}}
\def\zL{{\mathbb L}} 
\def\zD{{\mathbb D}}
 \def\zE{{\mathbb E}}
\def\zG{{\mathbb G}} 
\def\zA{{\mathbb A}} 
\def\zB{{\mathbb B}}
\def\zC{{\mathbb C}} 
\def\zM{{\mathbb M}} 
\def\zR{{\mathbb R}}
\def\zS{{\mathbb S}} 
\def\zT{{\mathbb T}} 
\def\zN{{\mathbb N}}
\def\zQ{{\mathbb Q}} 
\def\zW{{\mathbb W}} 
\def\bK{{\bf K}}
\def\C{{\bf C}} 
\def\M{{\bf M}} 
\def\E{{\bf E}} 
\def\N{{\bf N}}
\def\O{{\bf O}} 
\def\bN{{\bf N}} 
\def\bX{{\bf X}} 
\def\GF{{\rm GF}} 
\def\cC{{\mathcal C}} 
\def\cA{{\mathcal A}}
\def\cB{{\mathcal B}} 
\def\cD{{\mathcal D}} 
\def\cE{{\mathcal E}} 
\def\cF{{\mathcal F}} 
\def\cG{{\mathcal G}} 
\def\cH{{\mathcal H}}
\def\cI{{\mathcal I}} 
\def\cK{{\mathcal K}} 
\def\cL{{\mathcal L}} 
\def\cP{{\mathcal P}} 
\def\cR{{\mathcal R}} 
\def\cRY{{\mathcal RY}}
\def\cS{{\mathcal S}} 
\def\cT{{\mathcal T}} 
\def\cU{{\mathcal U}} 
\def\cV{{\mathcal V}} 
\def\cW{{\mathcal W}} 
\def\oB{{\ov B}}
\def\oC{{\ov C}} 
\def\ooB{{\ov{\ov B}}} 
\def\ozB{{\ov{\zB}}}
\def\ozD{{\ov{\zD}}} 
\def\ozG{{\ov{\zG}}}
\def\tcA{{\widetilde\cA}} 
\def\tcC{{\widetilde\cC}}
\def\tcF{{\widetilde\cF}} 
\def\tcI{{\widetilde\cI}}
\def\tB{{\widetilde B}} 
\def\tC{{\widetilde C}}
\def\tD{{\widetilde D}} 
\def\ttB{{\widetilde{\widetilde B}}}
\def\ttC{{\widetilde{\widetilde C}}}
\def\tba{{\tilde\ba}} 
\def\ttba{{\tilde{\tilde\ba}}}
\def\tbb{{\tilde\bb}} 
\def\ttbb{{\tilde{\tilde\bb}}}
\def\tbc{{\tilde\bc}} 
\def\tbd{{\tilde\bd}}
\def\tbe{{\tilde\be}} 
\def\tbt{{\tilde\bt}}
\def\tbu{{\tilde\bu}} 
\def\tbv{{\tilde\bv}}
\def\tbw{{\tilde\bw}} 
\def\tdl{{\tilde\dl}} 
\def\ocP{{\ov\cP}}
\def\tzA{{\widetilde\zA}} 
\def\tzC{{\widetilde\zC}}
\def\new{{\mbox{\footnotesize new}}}
\def\old{{\mbox{\footnotesize old}}}
\def\prev{{\mbox{\footnotesize prev}}}
\def\oo{{\mbox{\sf\footnotesize o}}}
\def\pp{{\mbox{\sf\footnotesize p}}}
\def\nn{{\mbox{\sf\footnotesize n}}} 
\def\oR{{\ov R}}
  
%%%%%%%%%%%%%%%%%%%%%%%%%%%%%% 
% Abbreviations for varieties
  
\def\gA{{\mathfrak A}} 
\def\gV{{\mathfrak V}} 
\def\gS{{\mathfrak S}} 
\def\gK{{\mathfrak K}} 
\def\gH{{\mathfrak H}}
  
%%%%%%%%%%%%%%%%%%%%%%%%%%%%%%%% 
%  Vectors  
\def\ba{{\bf a}} 
\def\bb{{\bf b}} 
\def\bc{{\bf c}} 
\def\bd{{\bf d}} 
\def\be{{\bf e}} 
\def\bbf{{\bf f}} 
\def\bg{{\bf g}}
\def\bh{{\bf h}}
\def\bi{{\bf i}} 
\def\bm{{\bf m}} 
\def\bo{{\bf o}} 
\def\bp{{\bf p}} 
\def\bs{{\bf s}} 
\def\bu{{\bf u}} 
\def\bt{{\bf t}} 
\def\bv{{\bf v}} 
\def\bx{{\bf x}}
\def\by{{\bf y}} 
\def\bw{{\bf w}} 
\def\bz{{\bf z}}
\def\ga{{\mathfrak a}} 
\def\oal{{\ov\al}} 
\def\obeta{{\ov\beta}}
\def\ogm{{\ov\gm}} 
\def\oep{{\ov\varepsilon}}
\def\oeta{{\ov\eta}} 
\def\oth{{\ov\th}} 
\def\ovm{{\ov\mu}}
\def\ozero{{\ov0}}
\def\bB{{\bf B}} 
\def\bA{{\bf A}}

%%%%%%%%%%%%%%%%%%%%%%%%%%%%%%%%% 
% Constraint satisfaction Problem
  
\def\CCSP{\hbox{\rm c-CSP}} 
\def\CSP{{\rm CSP}} 
\def\NCSP{{\rm \#CSP}} 
\def\mCSP{{\rm MCSP}} 
\def\FP{{\rm FP}} 
\def\PTIME{{\bf PTIME}} 
\def\GS{\hbox{($*$)}} 
\def\ry{\hbox{\rm r+y}}
\def\rb{\hbox{\rm r+b}} 
\def\Gr#1{{\mathrm{Gr}(#1)}}
\def\Grp#1{{\mathrm{Gr'}(#1)}} 
\def\Grpr#1{{\mathrm{Gr''}(#1)}}
\def\Scc#1{{\mathrm{Scc}(#1)}} 
\def\rel{R} 
\def\relo{Q}
\def\rela{S} 
\def\reli{T} 
\def\relp{P} 
\def\relov{R'}
\def\dep{\mathsf{dep}}
\def\Filt{\mathrm{Ft}}
\def\Filts{\mathrm{Fts}} 
\def\Agr{$\mathbb{A}$}
\def\Al{\mathrm{Alg}}
\def\Sig{\mathrm{Sig}}
\def\strat{\mathsf{strat}}
\def\relmax{\mathsf{relmax}}
\def\srelmax{\mathsf{srelmax}}
\def\Meet{\mathsf{Meet}}
\def\mmax{\mathsf{max}}
\def\amax{\mathsf{amax}}
\def\umax{\mathsf{umax}}
\def\emin{\mathsf{Z}}
\def\as{\mathsf{as}}
\def\asm{\mathsf{asm}}
\def\se#1{\mathsf{s}(#1)}
\def\see#1#2{\mathsf{s}_{#1}(#2)}
\def\star{\hbox{$(*)$}}
\def\bmal{{\mathbf m}}
\def\Af{\mathsf{Af}}
\let\sqq=\sqsubseteq
\def\maj{\mathsf{maj}}
\def\razm{\mathsf{size}}
\def\Razm{\mathsf{MAX}}
\def\Centr{\mathsf{Center}}
\def\centr{\mathsf{center}}

%%%%%%%%%%%%%%%%%%%%%%%%%%%%%%%%%% 
% Mathematical abbreviations
  
\let\sse=\subseteq 
\def\ang#1{\langle #1 \rangle}
\def\angg#1{\left\langle #1 \right\rangle}
\def\dang#1{\ang{\ang{#1}}} 
\def\vc#1#2{#1 _1\zd #1 _{#2}}
\def\tms{\tm\dots\tm}
\def\zd{,\ldots,} 
\let\bks=\backslash 
\def\red#1{\vrule height7pt depth3pt width.4pt
\lower3pt\hbox{$\scriptstyle #1$}}
\def\fac#1{/\lower2pt\hbox{$\scriptstyle #1$}}
\def\me{\stackrel{\mu}{\eq}} 
\def\nme{\stackrel{\mu}{\not\eq}}
\def\eqc#1{\stackrel{#1}{\eq}} 
\def\cl#1#2{\arraycolsep0pt
\left(\begin{array}{c} #1\\ #2 \end{array}\right)}
\def\cll#1#2#3{\arraycolsep0pt \left(\begin{array}{c} #1\\ #2\\
#3 \end{array}\right)} 
\def\clll#1#2#3#4{\arraycolsep0pt
\left(\begin{array}{c} #1\\ #2\\ #3\\ #4 \end{array}\right)}
\def\cllll#1#2#3#4#5#6{ \left(\begin{array}{c} #1\\ #2\\ #3\\
#4\\ #5\\ #6 \end{array}\right)} 
\def\pr{{\rm pr}}
\let\upr=\uparrow 
\def\ua#1{\hskip-1.7mm\uparrow^{#1}}
\def\sua#1{\hskip-0.2mm\scriptsize\uparrow^{#1}} 
\def\lcm{{\rm lcm}} 
\def\perm#1#2#3{\left(\begin{array}{ccc} 1&2&3\\ #1&#2&#3
\end{array}\right)} 
\def\w{$\wedge$} 
\let\ex=\exists
\def\NS{{\sc (No-G-Set)}} 
\def\lev{{\sf lev}}
\let\rle=\sqsubseteq 
\def\ryle{\le_{ry}} 
\def\ryprec{\le_{ry}}
\def\os{\mbox{[}} 
\def\zs{\mbox{]}}
\def\link{{\sf link}}
\def\solv{\stackrel{s}{\sim}} 
\def\mal{\mathbf{m}}
\def\precs{\prec_{as}}

%%%%%%%%%%%%%%%%%%%%%%%%%%%%%%%%%%% 
% Other abbreviations
  
\def\lb{$\linebreak$}  
  
%%%%%%%%%%%%%%%%%%%%%%%%%%%%%%%%%%% 
%  Functions  
\def\ar{\hbox{ar}} 
\def\Im{{\sf Im}\;} 
\def\deg{{\sf deg}}
\def\id{{\rm id}}
  
%%%%%%%%%%%%%%%%%%%%%%%%%%%%%%%%% 
%  Greek symbols  
\let\al=\alpha 
\let\gm=\gamma 
\let\dl=\delta 
\let\ve=\varepsilon
\let\ld=\lambda 
\let\om=\omega 
\let\vf=\varphi 
\let\vr=\varrho
\let\th=\theta 
\let\sg=\sigma 
\let\Gm=\Gamma 
\let\Dl=\Delta
\let\kp=\kappa
  
%%%%%%%%%%%%%%%%%%%%%%%%%%%%%%%%%% 
% Fonts and special symbols
  
\font\tengoth=eufm10 scaled 1200 
\font\sixgoth=eufm6
\def\goth{\fam12} 
\textfont12=\tengoth 
\scriptfont12=\sixgoth
\scriptscriptfont12=\sixgoth 
\font\tenbur=msbm10
\font\eightbur=msbm8 
\def\bur{\fam13} 
\textfont11=\tenbur
\scriptfont11=\eightbur 
\scriptscriptfont11=\eightbur
\font\twelvebur=msbm10 scaled 1200 
\textfont13=\twelvebur
\scriptfont13=\tenbur 
\scriptscriptfont13=\eightbur
\mathchardef\nat="0B4E 
\mathchardef\eps="0D3F

%%%%%%%%%%%%%%%%%%%%%%%%%%%%%%%%%
%%%%%%%%%%%%%%%%%%%%%%%%%%%%%%%%% 
\title{Separation of congruence intervals and implications}
\author{Andrei A.\ Bulatov\\ 
%% School of Computing Science, 
%% Simon Fraser University, Burnaby, Canada\\ 
%% \it e-mail: abulatov@cs.sfu.ca} 
} 
\date{} 
\maketitle

\begin{abstract}
The Constraint Satisfaction Problem (CSP) has been intensively 
studied in many areas of computer science and mathematics. 
The approach to the CSP based on tools from universal algebra 
turned out to be the most successful one to study the complexity
and algorithms for this problem. Several techniques have been 
developed over two decades. One of them is through associating
edge-colored graphs with algebras and studying how the properties
of algebras are related with the structure of the associated graphs.
This approach has been introduced in our previous two papers
(A.\ Bulatov, Graphs of finite algebras: edges, and connectivity. Algebra Universalis, 85(4):45, and Graphs of finite algebras: maximality, rectangularity, and decomposition. Algebra Universalis, 85(4):43, 2024). 
In this paper we further advance it by introducing new 
structural properties of finite idempotent algebras omitting type
\one\ and belonging to varieties generated by algebras satisfying an 
additional technical condition of being smooth. These properties include 
separation of congruences, collapsing polynomials,
and their implications for the structure of subdirect products of
finite algebras. It is known that any finite algebra $\zA$ which generates 
a variety omitting type~\one\ has an idempotent reduct $\zA'$ such that 
the variety generated by $\zA'$ omits type~\one\ and has a smooth generating 
set. In particular, the variety generated by any minimal Taylor 
algebra has a smooth generating set. This paper also provides the algebraic 
background for our proof of Feder-Vardi Dichotomy Conjecture
(A. Bulatov, A Dichotomy Theorem for Nonuniform CSPs. 
FOCS 2017: 319-330).
\end{abstract}

%%%%%%%%%%%%%%%%%%%%%%%%%%%%%%%%%%%
%%%%%%%%%%%%%%%%%%%%%%%%%%%%%%%%%%%
\section{Introduction}

Over the last two decades methods from universal algebra found 
strong applications in computer science, specifically in the study of
the Constraint Satisfaction Problem (CSP) and related combinatorial
problems. The original research problem where the algebraic
approach was used is the complexity of so-called nonuniform CSPs,
and more specifically the Dichotomy Conjecture posed by Feder
and Vardi in \cite{Feder93:monotone,Feder98:monotone} and 
refined in \cite{Bulatov05:classifying}. The Dichotomy Conjecture 
states that every nonuniform CSP is either solvable in polynomial
time or is NP-complete, and also delineates the precise borderline
between the two cases. Every nonuniform CSP can be associated with 
a finite algebra, and the complexity of the CSP is completely
determined by this algebra \cite{Jeavons97:closure,Bulatov05:classifying}. 
The Dichotomy Conjecture was confirmed independently by the author 
\cite{Bulatov17:dichotomy,Bulatov17:dichotomy-corr} and by Zhuk
\cite{Zhuk17:proof,Zhuk17:proof-corr}, and the algebraic approach 
played a key role in both proofs.  

The specific version of the algebraic approach used in 
\cite{Bulatov17:dichotomy,Bulatov17:dichotomy-corr} was developed 
in \cite{Bulatov04:graph,Bulatov08:recent,Bulatov16:connectivity,%
Bulatov16:restricted,Bulatov19:semilattice,Bulatov24:maximal}.
In this paper we further advance this approach preparing the 
ground for a proof of the Dichotomy Conjecture. We will introduce 
two structural features of finite algebras and demonstrate how 
they influence the structure of subdirect products of finite 
idempotent algebras omitting type \one. 

First we introduce the notion of separability of prime intervals
in the congruence lattice by a unary polynomial. More precisely, 
we say that a prime interval $\al\prec\beta$
in the congruence lattice of an algebra $\zA$ can be separated from 
interval $\gm\prec\dl$ if there is a unary polynomial $f$ of $\zA$ such 
that $f(\dl)\sse\gm$, but $f(\beta)\not\sse\al$. This concept
can be extended to subdirect products of algebras, say, 
$\rel\sse\zA_1\tm\zA_2$, in which case intervals $\al\prec\beta$ and 
$\gm\prec\dl$ may be in the congruence lattices of different factors,
say $\al,\beta\in\Con(\zA_1),\gm,\dl\in\Con(\zA_2)$, and $f$ is 
a polynomial of $\rel$. The relation `cannot be separated from' 
on the set of prime intervals is clearly reflexive and transitive.
Our first result, Theorem~\ref{the:relative-symmetry},
shows that it is also to some extent symmetric. 

The property proved in Theorem~\ref{the:relative-symmetry} is used 
to prove
the existence of the second structural feature of subdirect products,
collapsing polynomials, see Theorem~\ref{the:collapsing}. A unary 
polynomial of a subdirect product $\rel\sse\zA_1\tms\zA_n$ for a 
prime interval $\al\prec\beta$ in $\Con(\zA_i)$ for some $i$ is 
collapsing if for any $j$ and any prime interval $\gm\prec\dl$ in
$\Con(\zA_j)$ it holds that $f(\dl)\not\sse\gm$ if and only if $\al\prec\beta$ 
cannot be separated from $\gm\prec\dl$. Collapsing polynomials
are one of the main tools in the proof of the Dichotomy Conjecture
\cite{Bulatov17:dichotomy,Bulatov17:dichotomy-corr}, as they are 
very useful in the study of the structure of subdirect products. 
One example of such results is the Congruence 
Lemma~\ref{lem:affine-link}, which provides much information 
about the fine structure of a subdirect product of algebras when 
one of its factors is restricted on its congruence block. The 
Congruence Lemma is another important tool in the proof of 
the Dichotomy Conjecture. 

Besides congruence separation and collapsing polynomials we 
also introduce an alternative definition of the centralizer and
use it to derive certain properties of subdirect products. In addition, we
introduce two more technical properties of subdirect products, 
chaining and polynomial closure, and study their properties that
again are instrumental in the proof of the Dichotomy Conjecture.

%%%%%%%%%%%%%%%%%%%%%%%%%%%%%%%%%%%%
%%%%%%%%%%%%%%%%%%%%%%%%%%%%%%%%%%%%
\section{Preliminaries}\label{sec:preliminaries}

Here we introduce all the notation and terminology used in this paper.
It mainly follows the standard books \cite{Burris81:universal,Mckenzie87:algebras}.

%%%%%%%%%%%%%%%%%%%%%%%%%%%%%%%%%%%%
\subsection{Notation and agreements}

By $[n]$ we denote the set $\{1\zd n\}$. For sets $\vc An$ tuples 
from $A_1\tms A_n$ are denoted in boldface, say, $\ba$; the $i$th component of 
$\ba$ is referred to as $\ba[i]$. An $n$-ary relation 
$\rel$ over sets $\vc An$ is any subset of $A_1\tms A_n$. For 
$I=\{\vc ik\}\sse[n]$ by $\pr_I\ba,\pr_I\rel$ we denote the \emph{projections} 
$\pr_I\ba=(\ba[i_1]\zd\ba[i_k])$, 
$\pr_I\rel=\{\pr_I\ba\mid\ba\in\rel\}$ of tuple
$\ba$ and relation $\rel$. If $\pr_i\rel=A_i$ for each $i\in[n]$, relation $\rel$ is 
said to be a \emph{subdirect product} of 
$A_1\tms A_n$. It will be convenient to use $\ov A$ for 
$A_1\tms A_n$ if the sets $A_i$ are clear from the context. For $I\sse[n]$ 
we will use $\ov A_I$, for $\prod_{i\in I}A_i$, or if $I$ 
is clear from the context just $\ov A$.

Algebras will be denoted by $\zA,\zB$ etc.; we often do not distinguish between 
subuniverses and subalgebras. For $B\sse\zA$ the subalgebra generated by
$B$ is denoted $\Sgg\zA B$ or just $\Sg B$. For $C\sse\zA^2$ the congruence 
generated by $C$ is denoted $\Cgg\zA C$ or just $\Cg C$. 
The equality relation and 
the full congruence of algebra $\zA$ are denoted $\zz_\zA$ and 
$\zo_\zA$, respectively. Often when we need to 
use one of these trivial congruences of an algebra indexed in some way, say, 
$\zA_i$, we write $\zz_i,\zo_i$ for $\zz_{\zA_i}, \zo_{\zA_i}$.
The set of all polynomials (unary, binary polynomials) of $\zA$ 
is denoted by $\Pol(\zA)$ and $\Polo(\zA)$, $\Pol_2(\zA)$, respectively. 
A unary polynomial $f$ is \emph{idempotent} if $f\circ f=f$. We 
frequently use operations (acting component-wise) on subalgebras of direct products of algebras, say, 
$\rel\sse\zA_1\tms \zA_n$. If $f$ is such an operation (say, $k$-ary) then we 
denote its component-wise action 
also by $f$, e.g.\ $f(\vc ak)$ for $\vc ak\in\zA_i$. In the same way we denote the
action of $f$ on projections of $\rel$, e.g.\ $f(\vc\ba k)$ for $I\sse[n]$ and 
$\vc\ba k\in\pr_I\rel$. What we mean will always be clear from the context.
We use similar agreements for collections of congruences. If $\al_i\in\Con(\zA_i)$, $i\in[n]$,
then $\ov\al$ denotes the congruence $\al_1\tms\al_n$ of 
$\rel$. If $I\sse[n]$ we use $\ov\al_I$ to denote 
$\prod_{i\in I}\al_i$. If it does not lead to a confusion we write 
$\ov\al$ for $\ov\al_I$. Sometimes $\al_i$ are specified for $i$ from a certain 
set $I\sse[n]$, then by $\ov\al$ we mean the congruence $\prod_{i\in[n]}\al'_i$
where $\al'_i=\al_i$ if $i\in I$ and $\al'_i$ is the equality relation otherwise.
For example, if $\al\in\Con(\zA_1)$ then $\rel\fac{\ov\al}$ means the factor of $\rel$
modulo $\al\tm\zz_2\tm\dots\tm\zz_n$. For $\al\in\Con(\zA)$ and a polynomial $f$ of $\zA$, we will often abuse notation and denote the action of $f$ on $\zA\fac\al$ by the same symbol $f$. For $\al,\beta\in\Con(\zA)$ we write
$\al\prec\beta$ if $\al<\beta$ and $\al\le\gm\le\beta$ in $\Con(\zA)$ implies 
$\gm=\al$ or $\gm=\beta$.
In this paper all algebras are finite, idempotent and omit type \one, except the definition of edges and Theorem~\ref{the:connectedness} in the beginning of Section~\ref{sec:colored-graphs}.

%%%%%%%%%%%%%%%%%%%%%%%%%%%%%%%%%%%
\subsection{Minimal sets and polynomials}

We will use the following basic facts from tame congruence theory
\cite{Hobby88:structure}, often without further notice. 

%minimal sets, collapsing
Let $\zA$ be a finite algebra and $\al,\beta\in\Con(\zA)$ with $\al\prec\beta$. 
An \emph{$(\al,\beta)$-minimal set} is a set minimal
with respect to inclusion among the sets of the form $f(\zA)$, where 
$f\in\Polo(\zA)$ is such that $f(\beta)\not\sse\al$. Sets $B,C\sse\zA$ are said
to be \emph{polynomially isomorphic}
in $\zA$ if there are $f,g\in\Polo(\zA)$ such that $f(B)=C$, $g(C)=B$, and
$f\circ g, g\circ f$ are identity mappings on $C$ and $B$, respectively.

\begin{lemma}[Theorem 2.8, \cite{Hobby88:structure}]%
\label{lem:minimal-sets}
Let $\al,\beta\in\Con(\zA)$, $\al\prec\beta$. Then the following hold.
\begin{itemize}
\item[(1)] 
Any $(\al,\beta)$-minimal sets $U,V$ are polynomially isomorphic.
\item[(2)] 
For any $(\al,\beta)$-minimal set $U$ and any $f\in\Polo(\zA)$, if
$f(\beta\red U)\not\sse\al$ then $f(U)$ is an $(\al,\beta)$-minimal set, $U$ 
and $f(U)$ are polynomially isomorphic,  and $f$ witnesses this fact.
\item[(3)] 
For any $(\al,\beta)$-minimal set $U$ there is $f\in\Polo(\zA)$ such that
$f(\zA)=U$, $f(\beta)\not\sse\al$, and $f$ is idempotent, in particular, 
$f$ is the identity mapping on $U$.
\item[(4)] 
For any $(a,b)\in\beta-\al$ and an $(\al,\beta)$-minimal set $U$ there is
$f\in\Polo(\zA)$ such that $f(\zA)=U$ and $(f(a),f(b))\in\beta\red U-\al\red U$.
\item[(5)] For any $(\al,\beta)$-minimal set $U$, $\beta$ is the transitive closure of
$$
\al\cup\{(f(a),f(b))\mid (a,b)\in\beta\red U, f\in\Polo(\zA)\}.
$$
In fact, as $\al\prec\beta$, this claim can be strengthened as follows. 
For any $(a,b)\in\beta-\al$, $\beta$ is the symmetric and transitive closure of 
$$
\al\cup\{(f(a),f(b))\mid f\in\Polo(\zA)\}.
$$
\item[(6)] 
For any $f\in\Polo(\zA)$ such that $f(\beta)\not\sse\al$ there is an
$(\al,\beta)$-minimal set $U$ such that $f$ witnesses that $U$ and $f(U)$
are polynomially isomorphic.
\end{itemize}
\end{lemma}

%traces, subtraces
For an $(\al,\beta)$-minimal set $U$ and a $\beta$-block $B$ such that 
$\beta\red{U\cap B}\ne\al\red{U\cap B}$, the set $U\cap B$ is said
to be an \emph{$(\al,\beta)$-trace}. A 2-element set 
$\{a,b\}\sse U\cap B$ such that $(a,b)\in\beta-\al$, is called an 
\emph{$(\al,\beta)$-subtrace}. 
Depending on the 
structure of its minimal sets the interval $(\al,\beta)$ can be of one of the 
five types, \one--\five. Since we assume the tractability conditions of the
Dichotomy Conjecture, type~\one\ does not occur in algebras we deal with.

\begin{lemma}[Chapter~4 of \cite{Hobby88:structure}]\label{lem:traces}
Let $\al,\beta\in\Con(\zA)$ and $\al\prec\beta$. Then the following hold.
\begin{itemize}
\item[(1)] 
If $\typ(\al,\beta)=\two$ then every $(\al,\beta)$-trace (considered as an algebra whose operations are the restrictions of polynomials of $\zA$ on the trace) factored modulo $\al$ is polynomially equivalent to a 1-dimensional vector space.
\item[(2)] If $\typ(\al,\beta)\in\{\three,\four,\five\}$ then every $(\al,\beta)$-minimal 
set $U$ contains exactly one trace $T$, and if $\typ(\al,\beta)\in\{\three,\four\}$,
$T$ contains only 2 elements. Also, $T\fac\al$ is polynomially equivalent
to a Boolean algebra, 2-element lattice, or 2-element semilattice, respectively.
\end{itemize}
\end{lemma}

%projectivity and minimal sets
Intervals $(\al,\beta),(\gm,\dl)$, $\al,\beta,\gm,\dl\in\Con(\zA)$ and 
$\al\prec\beta,\gm\prec\dl$ are said to be 
\emph{perspective} if $\beta=\al\join\dl,
\gm=\al\meet\dl$, or $\dl=\beta\join\gm,\al=\beta\meet\gm$.

\begin{lemma}[Lemma~6.2, \cite{Hobby88:structure}]%
\label{lem:perspective-intervals}
Let $\al,\beta,\gm,\dl\in\Con(\zA)$ be such that $\al\prec\beta,\gm\prec\dl$
and intervals $(\al,\beta),(\gm,\dl)$ are perspective. Then
$\typ(\al,\beta)=\typ(\gm,\dl)$ and a set $U$ is $(\al,\beta)$-minimal
if and only if it is $(\gm,\dl)$-minimal.
\end{lemma}

%pseudo-operations
We will also use polynomials that behave on a minimal set in a particular way.

\begin{lemma}[Lemmas 4.15, 4.17, \cite{Hobby88:structure}]%
\label{lem:pseudo-meet}
Let $\al,\beta\in\Con(\zA)$, $\al\prec\beta$, and
$\typ(\al,\beta)\in\{\three,\four,\five\}$. Let $U$ be an $(\al,\beta)$-minimal
set and $T$ its only trace. Then there is an element $1\in T$ and a binary polynomial 
$p$ of $\zA$ such that
\begin{itemize}
\item[(1)] 
$(1,a)\not\in\al$ for any $a\in U-\{1\}$;
\item[(2)] 
for all $a\in U-\{1\}$, the algebra $(\{a,1\},p)$ is a semilattice with 
neutral element~$1$, that is, $p(1,1)=1$ and $p(1,a)=p(a,1)=p(a,a)=a$.
\end{itemize}
Polynomial $p$ is said to be a \emph{pseudo-meet} operation on $U$.

If $\typ(\al,\beta)\in\{\three,\four\}$ then $|T|=2$, say, $T=\{0,1\}$, 
and there is a binary polynomial $q$ of $\zA$, a \emph{pseudo-join} operation,
that satisfies the conditions of item (2) except $q(1,0)=q(0,1)=1$.
\end{lemma}

%%%%%%%%%%%%%%%%%%%%%%%%%%%%%%%%%%%
\subsection{Coloured graphs}\label{sec:colored-graphs}

% definitions  from connectivity
In \cite{Bulatov04:graph,Bulatov08:recent,Bulatov16:connectivity,%
Bulatov24:graph,Bulatov24:maximal,Bulatov20:restricted} we introduced a 
local approach to the
structure of finite algebras. As we use this approach throughout the paper, 
we present it here in some detail, see also 
\cite{Bulatov24:graph,Bulatov24:maximal}. 
For the sake of the definitions below we slightly abuse terminology 
and by a module mean an algebra term equivalent to the full idempotent reduct of a module.

For an algebra $\zA$ the (undirected) graph $\cG(\zA)$ is defined as follows. 
The vertex set is the universe $A$ of $\zA$. An unordered pair $ab$ of vertices is an 
\emph{edge} if and only if there exists a proper congruence $\th$ of 
$\Sg{a,b}$, and either $\Sg{a,b}\fac\th$ is 
a set (that is an algebra all of whose term operations are projections),
or there is a term operation $f$ of $\zA$ 
such that either $\Sg{a,b}\fac\th$ is a module and $f$ is an affine 
operation $x-y+z$ on it (`affine operation' will always refer to $x-y+z$), or $f$ is a semilattice operation on
$\{a\fac\th,b\fac\th\}$, or $f$ is a majority operation on
$\{a\fac\th,b\fac\th\}$. (Note that we use the same operation symbol in 
these cases.) 

If there are a proper congruence $\th$ and a term operation $f$ of $\zA$ such that $f$ 
is a semilattice operation on $\{a\fac\th,b\fac\th\}$ then $ab$ is said to have the
\emph{semilattice type}. An edge $ab$ is of 
\emph{majority type} if there are 
a proper congruence $\th$ and a term operation $f$ such that $f$ is a majority
operation on $\{a\fac\th,b\fac\th\}$ and there is no semilattice 
term operation on $\{a\fac\th,b\fac\th\}$. Also, $ab$ 
has the \emph{affine type} if there are proper $\th$ and $f$ 
such that $f$ is an affine operation on $\Sg{a,b}\fac\th$ and 
$\Sg{a,b}\fac\th$ is a module; in particular it implies that 
there is no semilattice or majority operation on $\{a\fac\th,b\fac\th\}$.  
Finally, if $\{a\fac\th,b\fac\th\}$ is a set, $ab$ is said to have the 
\emph{unary type}. In all cases we say that congruence $\th$ 
\emph{witnesses} the type of edge $ab$. For an edge $ab$ the set 
$\{a\fac\th,b\fac\th\}$ is said to be a \emph{thick edge}. Observe that 
a pair $ab$ can be an edge of more 
than one type as witnessed by different congruences, although this has no 
consequences in this paper.

Omitting type \one\ is characterized in \cite[Theorem~3.5]{Bulatov24:graph}. The second part of the next statement easily follows from \cite[Theorem~3.5]{Bulatov24:graph}.

\begin{theorem}[Theorem~3.5, \cite{Bulatov24:graph}]%
\label{the:connectedness}
An idempotent algebra $\zA$ omits type \one\ (that is, the variety generated
by $\zA$ omits type \one) if and only if $\cG(\zA)$ contains no edges of the 
unary type.

Moreover, a finite class $\cK$ of similar idempotent algebras closed under 
subalgebras and quotient algebras omits type~\one\ if and only if $\cG(\zA)$ 
contains no edges of the unary type for any $\zA\in\cK$.
\end{theorem}

For the sake of the Dichotomy Conjecture, it suffices to consider 
\emph{reducts} of an algebra $\zA$ omitting type \one, 
that is, algebras with the same universe but 
reduced set of term operations, as long as reducts also omit type \one. 
In particular, we are interested in reducts of $\zA$, in which semilattice and 
majority edges are subalgebras. An algebra $\zA$ such that 
$a\fac\th\cup b\fac\th$ is a subuniverse of $\zA$ 
for every semilattice or majority edge $ab$ of $\zA$ is called 
\emph{smooth}. 
It is easy to see that every subalgebra and every quotient of a smooth algebra is smooth.

\begin{theorem}[Theorem~3.12, \cite{Bulatov24:graph}]%
If $\zA$ is a finite idempotent algebra and the variety it generates omits type~\one\ (or, equivalently, $\cG(\zA)$ contains no unary edges), then there exists a reduct $\zA'$ of $\zA$ such that (i) the variety generated by $\zA'$ omits type~\one, and (ii) $\zA'$ is smooth. 
\end{theorem}

From this point on all algebras occurring in the paper are idempotent and omit type~\one, unless stated otherwise.

Many concepts and results in the paper involve a class of algebras rather than a 
single algebra. Such a class, usually denoted by $\cK$, is finite, consists of smooth algebras, and is closed under taking subalgebras and quotient algebras. A class of similar algebras satisfying these conditions will be called a \emph{smooth} class. 
For a smooth class $\cK$ let $\cV$ be the class of finite 
algebras from the variety it generates, that is, the pseudovariety generated 
by $\cK$. We will slightly abuse the terminology and call $\cV$ the variety 
generated by $\cK$.  

Observe that as the following example shows even though $\cK$ consists of smooth algebras, algebras in $\cV$ are not necessarily smooth.

\begin{example}\label{exa:non-smooth}
Let $\zA$ be an algebra with universe $A=\{a,b,c\}$ and two basic operations $f$ and $g$. The operation $f$ is majority on $A$, and $g$ is minority on $\{a,b\}$ and $\{a,c\}$, and a 2/3-minority on $\{b,c\}$, that is, $g(x,y,y)=g(x,y,x)=g(y,y,x)=x$ on $\{b,c\}$. If $\{x,y,z\}=\{a,b,c\}$ then $f(x,y,z)=g(x,y,z)=x$. As is easily seen, all three 2-element subsets of $A$ are subuniverses and there is no term operation of $\zA$ that is semilattice on any of the three 2-element subsets. Therefore, each of the pairs $ab$, $bc$, $ac$ is a majority edge as witnessed by the equality relation, and $\zA$ is smooth. Consider $\zA^2$ and $(a,b),(b,c)\in\zA^2$. Since $\zA^2$ does not have a binary term operation acting as a semilattice operation on these two pairs, $(a,b)(b,c)$ is a majority edge in $\zA^2$ witnessed by the equality relation. However, $\{(a,b),(b,c)\}$ is not a subalgebra, because
\[
g\left(\cl ab,\cl bc,\cl ab\right)=\cl{g(a,b,a)}{g(b,c,b)}=\cl bb.
\]
\end{example}

The next statement uniformizes the operations witnessing the type of edges
in smooth algebras.

\begin{theorem}[Theorem~3.21 and Corollary~3.22, \cite{Bulatov24:graph}]%
\label{the:uniform}
Let $\cK$ be a smooth class and $\cV$ the variety it generates. There are term operations $f,g,h$ of 
$\cV$ such that for any $\zA\in\cK$, any edge $ab$ of $\zA$, and $E=\{a\fac\th,b\fac\th\}$, where $\th$ is a congruence witnessing that $ab$ is an edge,
\begin{itemize}
\item[(i)]
$f\red E$ is a semilattice operation if $ab$ is a semilattice edge; 
it is the first projection if $ab$ is a majority or affine edge;
\item[(ii)]
$g\red E$ is a majority operation if $ab$ is a majority edge;
it is the first projection if $ab$ is an affine edge, and
$g\red E(x,y,z)=f\red E(x,f\red E(y,z))$ if $ab$ is
semilattice;
\item[(iii)]
$h\red{\Sg{a,b}\fac\th}$ is an affine operation if $ab$ is an affine edge; 
it is the first projection if $ab$ is a majority edge, and
$h\red E(x,y,z)=f\red E(x,f\red E(y,z))$ if $ab$ is semilattice.
\end{itemize}
Moreover, operations $g,h$ can be chosen such that for any $\zA\in\cV$ and any affine edge $ab$ of $\zA$, where $\th$ is a congruence witnessing that $ab$ is an affine edge, $g$ is the first projection  and $h$ is an affine operation on $\Sg{a,b}\fac\th$.
%% such that for any $\zA\in\cK$ and any $a,b\in\zA$, the operation $f$ is a 
%% semilattice operation on $\{a\fac\th,b\fac\th\}$ if $ab$ is a semilattice edge;
%% $g$ is a majority operation on $\{a\fac\th,b\fac\th\}$ if $ab$ is a majority edge;
%% $h$ is the affine operation $x-y+z$ on $\Sg{a,b}\fac\th$ if $ab$ is an affine 
%% edge, where $\th$ witnesses the type of the edge.
\end{theorem}

\begin{proof}
Theorem~\ref{the:uniform} is proved in \cite{Bulatov24:graph} except for the part concerning the behavior of $g,h$ on algebras from $\cV$. A proof of this part is almost verbatim a part of the proof of Theorem~3.21 from \cite{Bulatov24:graph}.

%% Let $g^\dg,h^\dg$ be operations satisfying items (ii) and (iii) of the theorem. 
Let $\vc\zD\ell$ be a list of all 2-generated algebras from $\cV$ (up to an isomorphism) that are term equivalent to a module. There are finitely many such algebras, because $\cV$ is finitely generated. It suffices to find $g,h$ that act correctly on the $\zD_i$'s, because $\Sg{a,b}\fac\th$ for any affine edge $ab$ of an algebra from $\cV$ is in this list. Let $\vc h\ell$ be a list of terms of $\cV$ such that $h_i$ is the affine operation of $\zD_i$. Note that every binary term operation of $\zD_i$ is of the form $\al x+(1-\al)y$.

First we show that for any $i\in[\ell]$ there is $h^i$ such that $h^i$ is the affine operation of $\zD_j$, $j\le i$. The base case of induction is given by $h^1=h_1$. If $h^{i-1}$ is constructed, then if $h^{i-1}$ is the affine operation of $\zD_i$ then set $h^i=h^{i-1}$. Otherwise, $h^{i-1}=\al x+\beta y+\gm z$ on $\zD_i$ with $\al+\beta+\gm=1$. Set $h'(x,y)=h^{i-1}(x,y,y)$, $h''(x,y)=h^{i-1}(y,y,x)$, and observe that $h'(x,y)=\al x+(1-\al)y$ on $\zD_i$, $h''(x,y)=\gm x+(1-\gm)y$ on $\zD_i$, and $h'(x,y)=h''(x,y)=x$ on $\zD_j$, $j<i$. Let 
\[
h''_1(x,y,z)=h'(x,h_i(x,y,z)),\qquad h''_3(x,y,z)=h''(x,h_i(z,y,x)).
\]
%% Then 
%% \begin{align*}
%% h''_1(x,y,z) & =  x-(1-\al)y+(1-\al)z,\quad \text{on $\zD_i$ and}\\
%% h''_1(x,y,z) &= x, \qquad \text{on $\zD_j$, $j<i$}.
%% \end{align*}
%% Similarly, we can obtain $h''_3(x,y,z)$ with the property that $h''_3(x,y,z)=x-(1-\gm)y+(1-\gm)z$ on $\zD_i$ and $h''_3(x,y,z)=x$ on $\zD_j$, $j<i$. 
Furthermore, set $h''_2(x,y,z)=h''_1(h^{i-1}(x,y,z),z,x)$ and 
%% As is easily seen, for this operation we have
%% \begin{align*}
%% h''_2(x,y,z) &= \al x+\beta y+\gm z-(1-\al)z+(1-\al)x =
%% x+\beta y-\beta z,\quad \text{on $\zD_i$ and}\\
%% h''_2(x,y,z) &= x-y+z, \qquad \text{on $\zD_j$, $j<i$}.
%% \end{align*}
%% Finally, we set
\[
h^i(x,y,z)=h''_1(h''_2(h''_3(x,y,z),y,z),y,z).
\]
%% Again, we have 
%% \begin{align*}
%% h^i(x,y,z)\red{D'_i} &= x-(1-\gm)y+(1-\gm)z+\beta y-\beta z
%% -(1-\al)y+(1-\al)z\\
%% & =  x-y+z,\quad \text{on $\zD_i$ and}\\
%% h^i(x,y,z)\red{D'_j} &= x-y+z, \qquad \text{on $\zD_j$, $j<i$}.
%% \end{align*}
It is straightforward to verify that $h^i$ satisfies the desired conditions. We now set $h^\dg=h^\ell$.

Next, we prove that $g$ can be chosen such that $g(x,y,z)=x$ on $\zD_i$ for all $i\in[\ell]$. Let $g^\dg$ be an operation satisfying item (ii) of the theorem. Let $\vc Ck$ be a list of all thick majority edges of algebras from $\cK$. Since the algebras from $\cK$ are smooth, every $C_j$ is a 2-element algebra. We start with showing that for every $j\in[k]$ there is a term $p_j$ such that $p_j(x,y)=x$ on $C_j$ and $p_j(x,y)=y$ on $\zD_i$, $i\in[\ell]$. As $C_j$ is a majority but not a semilattice edge, every binary term operation is a projection on $C_j$, and the term $h^\dg$ can only be one of the following operations on $C_j$: a projection, a majority operation, a 2/3-minority operation (i.e.\ one satisfying the identities $h^\dg(x,x,y)=h^\dg(x,y,x)=y, h^\dg(y,x,x)=x$ or similar), or a minority operation. If $h^\dg$ on $C_j$ satisfies the identity $h^\dg(x,y,y)=y$ or $h^\dg(y,y,x)=y$ then set $p_j(x,y)=h^\dg(y,x,x)$ or $p_j(x,y)=h^\dg(x,x,y)$, respectively. If $h^\dg(x,y,x)=x$ on $C_j$, then set $p_j(x,y)=h^\dg(x,h^\dg(x,y,x),x)$. Finally, suppose $h^\dg$ on $C_j$ is the minority operation and $g^\dg(x,y,z)=\al_ix+\beta_iy+\gm_iz$ on $\zD_i$, $i\in[\ell]$. Then set $s_1(x,y)=h^\dg(g^\dg(x,y,y),y,g^\dg(y,y,x))$ and $s_2(x,y)=g^\dg(x,y,x)$. As is easily seen, $s_1(x,y)=y,s_2(x,y)=x$ on $C_j$ and $s_1(x,y)=s_2(x,y)=(1-\beta_i)x+\beta_i y$ on $\zD_i$. Then set $p_j(x,y)=h^\dg(s_1(x,y),s_2(x,y),y)$. We have $p_j(x,y)=h^\dg(y,x,y)=x$ on $C_j$ and $p_j(x,y)=y$ on each $\zD_i$, $i\in[\ell]$, as required. Note also, that for any $j'\in[k]$ the operation $p_j$ is a projection on $C_{j'}$.

Now, we prove by induction that for every $j\in[k]$ there is an operation $p^j(x,y)$ such that $p^j(x,y)=x$ on $C_r$ for $r\le j$ and $p^j(x,y)=y$ on $\zD_i$, $i\in[\ell]$. The operation $p^1=p_1$ gives the base case of induction. Let us assume that $p^{j-1}$ is already found. If $p^{j-1}(x,y)=x$ on $C_j$, set $p^j=p^{j-1}$. Otherwise, set $p^j(x,y)=p_j(x,p^{j-1}(x,y))$. As is easily seen, $p^j$ satisfies the required conditions. Finally, let $p=p^k$ and set $g^\ddg(x,y,z)=p(g^\dg(x,y,z),x)$, which is as desired.

The last step in the proof is to make sure that $g$ and $h$ act correctly on semilattice edges, and that $h$ acts correctly on majority edges. For the latter it suffices to set $h^\ddg(x,y,z)=p(x,h^\dg(x,y,z))$. For the former, note that the operation $f$ from item (i) is a projection on the $C_j$'s and an operation of the form $\al x+(1-\al)y$ on the $\zD_i$'s. In the latter case by iterating $f$ we can assume that $\al$ is an idempotent. The operation $f(f(x,y),x)$ is semilattice on every semilattice edge of algebras from $\cK$ and the first projection on the $C_j$'s and $\zD_i$'s. Replace $f$ with this operation. Then 
\begin{align*}
g(x,y,z) &=g^\ddg(f(x,f(y,z)),f(y,f(z,x)),f(z,f(x,y))),\quad \text{and}\\
h(x,y,z) &=h^\ddg(f(x,f(y,z)),f(y,f(z,x)),f(z,f(x,y)))
\end{align*}
are as required.
\end{proof}

Operations $f,g,h$ from Theorem~\ref{the:uniform} above can be chosen to satisfy certain identities. Lemma~\ref{lem:fgh-identities} is proved in \cite{Bulatov24:graph} for algebras from $\cK$. However, since $\cV$ is generated by $\cK$, the same identities hold in $\cV$.

\begin{lemma}[Lemma~4.1 of \cite{Bulatov24:graph}]\label{lem:fgh-identities}
Operations $f,g,h$ identified in Theorem~\ref{the:uniform} can be chosen 
such that
\begin{itemize}\itemsep0pt
\item[(1)]
$f(x,f(x,y))=f(x,y)$ for all $x,y\in\zA\in\cV$;
\item[(2)]
$g(x,g(x,y,y),g(x,y,y))=g(x,y,y)$ for all $x,y\in\zA\in\cV$;
\item[(3)]
$h(h(x,y,y),y,y)=h(x,y,y)$ for all $x,y\in\zA\in\cV$.
\end{itemize}
\end{lemma}

We will assume that for a smooth class $\cK$ operations $f,g,h$ satisfying the conditions of Theorem~\ref{the:uniform} and Lemma~\ref{lem:fgh-identities} are chosen and fixed. 

Thin edges also introduced in \cite{Bulatov24:graph} offer a better technical tool. Note that thin edges are defined 
for any algebra from the variety generated by $\cK$. Although we keep the same notation $ab$ for thin edges as for usual edges introduced above, thin edges are directed and the direction depends on the choice of the operations $f,g,h$ in Theorem~\ref{the:uniform}.

For a smooth class $\cK$, $\zA\in\cK$, and $a,b\in\zA$, the pair $ab$ is called a 
\emph{thin semilattice edge} if the equality relation
witnesses that it is a semilattice edge and $f(a,b)=f(b,a)=b$, which defines the direction of $ab$.
The binary operation $f$ from Theorem~\ref{the:uniform} can be 
chosen to satisfy a special property. 

\begin{prop}[Proposition~4.2, \cite{Bulatov24:graph}]\label{pro:good-operation}
Let $\cK$ be a smooth class. The operation $f$ identified in Theorem~\ref{the:uniform} and Lemma~\ref{lem:fgh-identities} can be chosen to also satisfy the following: for 
any $a,b\in\zA$, $\zA\in\cK$, either $a=f(a,b)$ or the pair $(a,f(a,b))$ 
is a thin semilattice edge.
\end{prop}

We assume that operation $f$ satisfying the conditions of 
Proposition~\ref{pro:good-operation} is fixed, and use $\cdot$ to
denote it (think multiplication). If $ab$ is a thin semilattice edge,
that is, $a\cdot b=b\cdot a=b$, we write $a\le b$.

Let again $\cK$ be a smooth class and $\cV$ the variety it generates. Let $\zA\in\cV$, $a,b\in\zA$, 
$\zB=\Sg{a,b}$, and let $\th$ be a congruence of $\zB$. Pair $ab$ is said to be 
\emph{minimal} with respect to $\th$ if for any $b'\in b\fac\th$, 
$b\in\Sg{a,b'}$. A ternary term $g'$ is said to satisfy the \emph{majority
condition} (with respect to $\cK$) if it satisfies the conditions of Lemma~\ref{lem:fgh-identities}(2) and $g'$ is a majority operation on every 
thick majority edge of every algebra from $\cK$. A ternary term  
$h'$ is said to satisfy the \emph{minority condition} if it satisfies the conditions of Lemma~\ref{lem:fgh-identities}(3) and 
for any $\zB\in\cK$ and every affine edge $ab$ of $\zB$ witnessed 
by a congruence $\th$ of $\Sgg\zB{a,b}$, the operation $h'$ is a Mal'tsev 
operation on $\Sgg\zB{a,b}\fac\th$. By 
Theorem~\ref{the:uniform} operations satisfying the majority and 
minority conditions exist.

Let $\zA\in\cV$ and $a,b\in\zA$. 
The (ordered) pair $ab$ is a \emph{thin semilattice edge} 
if the term $\cdot$ of $\cV$ is a semilattice operation on $\{a,b\}$ and $a\cdot b=b$. 

A pair $ab$ is called a \emph{thin majority edge} if 
\begin{itemize}
\item[(*)] 
for any term operation $g'$ satisfying the majority condition
the subalgebras $\Sg{a,g'(a,b,b)},\Sg{a,g'(b,a,b)},\Sg{a,g'(b,b,a)}$
contain $b$.
\end{itemize}

A pair $ab$ is called a \emph{thin affine edge} if 
\begin{itemize}
\item[(**)] 
$h(b,a,a)=b$ (where $h$ 
is the fixed operation satisfying the conditions of Theorem~\ref{the:uniform}), 
and for any term operation $h'$ satisfying the minority condition $b\in\Sg{a,h'(a,a,b)}$.
\end{itemize}
The operations $g,h$ from Theorem~\ref{the:uniform} do not 
have to satisfy any specific conditions on the set $\{a,b\}$, when
$ab$ is a thin majority or affine edge, except what follows from their 
definition. Also, both thin majority and thin affine edges are 
directed, since $a,b$ in the definition occur asymmetrically. Note also, that 
what pairs of an algebra $\zA$ are thin majority and affine 
edges depends not only on the algebra itself, but also on the underlying
class $\cK$. We will assume that $\cK$ is closed under taking subalgebras and factor-algebras. Thus, if we start with just a single algebra, we set 
$\cK=\H\S(\zA)$. 

\textbf{Convention:} From now on we fix a smooth class $\cK$ and the variety $\cV$ it generates along with operations $f,g,h$ satisfying the conditions of Theorem~\ref{the:uniform}, Lemma~\ref{lem:fgh-identities}, and Proposition~\ref{pro:good-operation}. All algebras in the rest of the paper are from $\cV$ and smooth ones are from $\cK$, unless otherwise stated.

\begin{lemma}[Corollary~4.3, Lemmas~4.6,4.10, \cite{Bulatov24:graph}]%
\label{lem:thin-semilattice}
Let $\zA\in\cK$, and let $ab$ be a semilattice (majority, affine) edge, $\th$ a congruence of
$\Sg{a,b}$ that witnesses this, and $c\in a\fac\th$. Then, if $ab$ is 
a semilattice edge such that $f(a\fac\th,b\fac\th)=b\fac\th$ or a majority edge, then for any $d\in b\fac\th$ such that 
$cd$ is a minimal pair with respect to $\th\red{\Sg{c,d}}$ the pair $cd$ is a thin 
semilattice [respectively, thin majority] edge. If $ab$ is affine then for any 
$d\in b\fac\th$ such that $ad$ is a minimal pair with respect to 
$\th\red{\Sg{a,d}}$ and $h(d,a,a)=d$ the pair $ad$ is a thin affine edge. Moreover,
$d\in b\fac\th$ satisfying these conditions exists.
\end{lemma}

The following simple properties of thin edges will be useful. Note that a 
subdirect product of algebras (a relation) is also an algebra, and so edges 
and thin edges can be defined for relations as well.

\begin{prop}[Proposition~1.8, \cite{Bulatov24:maximal}]\label{pro:very-good-operation}
For every $\zA\in\cV$ and for any $a,b\in\zA$ either $a=a\cdot b$ or 
the pair $(a,a\cdot b)$ is a thin semilattice edge.
\end{prop}

Items (1) and (2) of the following lemma are Lemma~1.11 from \cite{Bulatov24:maximal}, and item (3) follows from the definitions.

\begin{lemma}[Lemma~1.11, \cite{Bulatov24:maximal}]%
\label{lem:thin-properties}
\begin{itemize}
\item[(1)] 
Let $\zA\in\cV$, $\oa\ob$ be a thin edge in $\zA\fac\th$, and $a\in\oa$. 
Then there is $b\in\ob$ such that $ab$ is a thin edge in $\zA$
of the same type.
\item[(2)] 
Let $\zA\in\cV$ and $ab$ be a thin edge. Then $ab$
is a thin edge in any subalgebra of $\zA$ containing $a,b$, and $a\fac\th b\fac\th$
is a thin edge of the same type in $\zA\fac\th$ for any congruence $\th$ not containing $(a,b)$.\footnote{That $a\fac\th b\fac\th$ has the same type as $ab$ follows from the proof of Lemma~1.11 from \cite{Bulatov24:maximal}.} 
\item[(3)] Let $\zA\in\cV$ and $\zB$ its subalgebra. Then every thin edge of $\zB$ is a thin edge of $\zA$ of the same type.
\end{itemize}
\end{lemma} 

We will need operations that act in a specific way on pairs of thin
edges. 

\begin{lemma}[Lemma~4.14, \cite{Bulatov24:graph}, Lemma~1.9, \cite{Bulatov24:maximal}]%
\label{lem:op-s-on-affine}
\begin{itemize}
\item[(1)] 
Let $a=b$ or let $ab$ be a thin majority edge of an algebra $\zA\in\cV$. There is a term
operation $t_{ab}$ such that $t_{ab}(a,b)=b$ and $t_{ab}
(c,d)\eqc\eta c$ for all affine edges $cd$ of all $\zB\in\cV$, where the
type of $cd$ is witnessed by the congruence $\eta$.
\item[(2)] 
Let $a=b$ or let $ab$ be a thin affine edge of an algebra $\zA\in\cV$. There is a term
operation $h_{ab}$ such that $h_{ab}(a,a,b)=b$ and
$h_{ab}(c,d,d)\eqc\eta c$ for all affine edges $cd$ of all
$\zB\in\cV$, where the type of $cd$ is witnessed by the congruence $\eta$.
Moreover, $h_{ab}(x,c',d')$ is a permutation of $\Sg{c,d}\fac\eta$
for any $c',d'\in\Sg{c,d}$.
\item[(3)] Let $ab$ and $cd$ be thin edges in $\zA_1,\zA_2\in\cV$.
If they have different types there is a binary term operation $p$ such that
$p(b,a)=b$, $p(c,d)=d$. If both edges are affine then there is a term operation
$h'$ such that $h'(a,a,b)=b$ and $h'(d,c,c)=d$.
\end{itemize}
\end{lemma}

\begin{proof}
Firstly, if $a=b$ in items (1) and (2) then $t_{ab},h_{ab}$ can be chosen to be the first projection. So, we assume that $a\ne b$ in (1) and (2). 
Item (3) is items (1) and (2) of Lemma~1.9 from \cite{Bulatov24:maximal}. Lemma~4.14 from \cite{Bulatov24:graph} proves items (1) and (2) for smooth algebras from $\cK$. We generalize this statement for non-smooth algebras from $\cV$. The proof is similar to that of Lemma~4.14 from~\cite{Bulatov24:graph}. 

(1) Let $b'=g(a,b,b)$. By the definition of thin majority edges 
$b\in\Sgg\zA{a,b'}$ and there is a binary term operation $r$ such that
$b=r(a,b')$. By Theorem~\ref{the:uniform} $g(x,y,y)$ is the first projection on 
$\Sg{c,d}\fac{\eta}$ for any affine edge $cd$ of any algebra $\zB\in\cV$, where the congruence $\eta$ witnesses this fact. Let $t_{ab}(x,y)=r(x,g(x,y,y))$. Then
\begin{eqnarray*}
t_{ab}(a,b) &=& r(a,g(a,b,b))=b,\\
t_{ab}(c,d) &=& r(c,g(c,d,d))\eqc{\eta}c.
\end{eqnarray*}
This means that $t_{ab}$ satisfies the required conditions.

(2) Let $b'=h(a,a,b)$. By the definition of thin affine edges 
$b\in\Sgg\zA{a,b'}$ and there is a binary term operation $r$ such that
$b=r(a,b')$. By Theorem~\ref{the:uniform} $h(x,y,y)$ is the first projection on 
$\Sg{c,d}\fac{\eta}$ for any affine edge $cd$ of any algebra $\zB\in\cV$, where the congruence $\eta$ witnesses this fact. Let $h_{ab}(x,y,z)=r(x,h(x,y,z))$. Then
\begin{eqnarray*}
h_{ab}(a,a,b) &=& r(a,h(a,a,b))=b,\\
h_{ab}(c,d,d) &=& r(c,h(c,d,d))\eqc{\eta}c.
\end{eqnarray*}
This means that $h_{ab}$ satisfies the required conditions. To prove the second statement let $c',d'\in\Sg{c,d}$. Since $\zC=\Sg{c,d}\fac{\eta}$ is a module, in particular, it is an Abelian algebra and $h_{ab}(x,c^*,c^*)=x$ for all $c^*\in\zC$, the second result follows.
\end{proof}

%%%%%%%%%%%%%%%%%%%%%%%%%%%%%%%%
\subsection{Maximality}\label{sec:maximality}

% connectivity
Let $\zA\in\cV$. A \emph{path} in $\zA$ is a sequence 
$a_0,a_1\zd a_k$ such 
that $a_{i-1}=a_i$ or $a_{i-1}a_i$ is a thin edge for all $i\in[k]$ (note that thin edges 
are always assumed to be directed).
We will distinguish paths of several types depending on what types of
edges are allowed. If $a_{i-1}\le a_i$ for all $i\in[k]$ then the path is
called a \emph{semilattice} or \emph{s-path}. If for every $i\in[k]$
either $a_{i-1}\le a_i$ or $a_{i-1}a_i$ is a thin
affine edge then the path is called \emph{affine-semilattice} or
\emph{as-path}. 
The path is called \emph{asm-path} when all types of edges are allowed. If 
there is a path $a=a_0,a_1\zd a_k=b$ which is arbitrary (semilattice, 
affine-semilattice) then $a$ is said to be \emph{asm-connected} (or 
\emph{s-connected}, or \emph{as-connected}) to $b$. We will also 
say that $a$ is \emph{connected} to $b$ if it is asm-connected. 
We denote this by $a\sqq^{asm} b$ (for asm-connectivity), $a\sqq b$, 
and $a\sqq^{as} b$ for s-, and as-connectivity, respectively. 

Let $\cG_s(\zA)$ ($\cG_{as}(\zA),\cG_{asm}(\zA)$) denote the digraph 
whose nodes are the elements of $\zA$, and the edges are the thin 
semilattice edges (thin semilattice and affine edges, arbitrary thin edges, 
respectively). The strongly connected component of $\cG_s(\zA)$ 
containing $a\in\zA$ will be denoted by $\se a$. The set of strongly 
connected components of $\cG_s(\zA)$ is ordered in the natural 
way (if $a\le b$ then $\se a\le \se b$), the elements belonging to 
maximal ones will be called \emph{maximal}, and the set of all 
maximal elements from $\zA$ will be denoted by $\mmax(\zA)$. 

The strongly connected component of $\cG_{as}(\zA)$ containing 
$a\in\zA$ will be denoted by $\as(a)$. A maximal strongly connected 
component of this graph is called an
\emph{as-component}, an element from an as-component
is called \emph{as-maximal}, and the set of all as-maximal elements is
denoted by $\amax(\zA)$. 

Finally, the strongly connected component of $\cG_{asm}(\zA)$ containing 
$a\in\zA$ will be denoted by $\asm(a)$. A maximal strongly connected 
component of $\cG_{asm}(\zA)$ is called a \emph{universally maximal 
component} (or \emph{u-maximal component} for short), an element 
from a u-maximal component is called \emph{u-maximal}, and the set of all 
u-maximal elements is denoted by $\umax(\zA)$. 

Sometimes we use the notation $\cG_s(B),\cG_{as}(B),\cG_{asm}(B)$ and $\mmax(B)$, $\amax(B)$, $\umax(B)$ for a subset $B$ of $\zA$ that is not necessarily a subalgebra. In this case $\cG_s(B)$, $\cG_{as}(B)$, $\cG_{asm}(B)$ denote the subgraphs of $\cG_s(\zA),\cG_{as}(\zA),\cG_{asm}(\zA)$, respectively, induced by $B$, and $\mmax(B),\amax(B),\umax(B)$ refer to elements of the maximal strongly connected components of those subgraphs. 

Alternatively, maximal, as-maximal, and u-maximal elements 
can be characterized as follows: an element $a\in\zA$ is 
maximal (as-maximal, u-maximal) if for every $b\in\zA$ such 
that $a\sqq b$ ($a\sqq^{as}b$, $a\sqq^{asm} b$)
it also holds that $b\sqq a$ ($b\sqq^{as}a$, $b\sqq^{asm}a$). 
Sometimes it will be convenient to specify what the algebra is, 
in which we consider maximal components, as-components, 
or u-maximal components, and the corresponding connectivity. 
In such cases we will specify it by writing $\see \zA a$, 
$\as_\zA(a)$, or $\asm_\zA(a)$. For connectivity we will use
$a\sqq_\zA b$, $a\sqq_\zA^{as}b$, and $a\sqq_\zA^{asm}b$.

As a straightforward implication of Lemma~\ref{lem:thin-properties}(3) we have the following statement.

\begin{corollary}\label{cor:subalgebra-path}
Let $\zA\in\cV$, $\zB$ its subalgebra, and $a,b\in\zB$. Then if $a\sqq_\zB b$ (or $a\sqq^{as}_\zB b$, $a\sqq^{asm}_\zB b$) then $a\sqq_\zA b$ (respectively, $a\sqq^{as}_\zA b$, $a\sqq^{asm}_\zA b$). Moreover, if $\zB$ contains a maximal component (an as-component) $C$ of $\zA$, then $C$ is a maximal component (as-component) of $\zB$.
\end{corollary}

As the following result shows, the (as-, u-) maximal elements of $\cG_{asm}(\zA)$ are asm-connected to each other.

\begin{prop}[Corollary~2.4, \cite{Bulatov24:maximal}]%
\label{pro:as-connectivity}
Let $\zA\in\cV$. Then
any $a,b\in\mmax(\zA)$ (or $a,b\in\amax(\zA)$, or $a,b\in\umax(\zA)$) are 
connected in $\cG_{asm}(\zA)$ with a directed path.
\end{prop}

Since for every $a\in\zA$ there is a maximal $a'\in\zA$ such that $a\sqq a'$, 
Proposition~\ref{pro:as-connectivity} implies that there is only one u-maximal 
component. Moreover, Proposition~\ref{pro:as-connectivity} implies the following connection between maximal, as-maximal, and u-maximal elements.

\begin{corollary}\label{cor:amax-umax}
Let $\zA\in\cV$. Then $\mmax(\zA),\amax(\zA)\sse\umax(\zA)$.
\end{corollary}

U-maximality has an additional useful property, it is somewhat 
hereditary, as it is made precise in the following

\begin{lemma}\label{lem:u-max-congruence}
Let $\zB$ be a subalgebra of $\zA\in\cV$, containing a u-maximal element of $\zA$. 
Then every element u-maximal in $\zB$ is also 
u-maximal in $\zA$. In particular, if $\al$ is a congruence of $\zA$ and $\zB$ 
is a u-maximal $\al$-block, that is, $\zB$ is a u-maximal element in 
$\zA\fac\al$, then $\umax(\zB)\sse\umax(\zA)$. 
\end{lemma}

\begin{proof}
Let $a\in \zB$ be an element u-maximal in $\zA$, let $b\in\umax(\zB)$. For any 
$c\in\zA$ with $b\sqq^{asm}_\zA c$ we also have $c\sqq^{asm}_\zA a$. Finally, since 
$b\in\umax(\zB)$ and $a\in\zB$, we have $a\sqq^{asm}_\zA b$, and therefore $c\sqq^{asm}_\zA b$. For the second 
part of the lemma we need to find a u-maximal element in $\zB$. Let 
$b\in\umax(\zA)$. Then as $\zB$ is u-maximal in $\zA\fac\al$ applying 
Lemma~\ref{lem:thin-properties}(1) we get that there is
 $a'\in\zB$ such that $b\sqq^{asm}_\zA a'$. Clearly, $a'\in\umax(\zA)$.
\end{proof}

\begin{lemma}[The Maximality Lemma, Lemma~1.15, 
Corollaries~1.16,1.17, \cite{Bulatov24:maximal}]\label{lem:to-max}
Let $\rel$ be a subdirect product of $\vc\zA n\in\cV$, and $I\sse[n]$.
\begin{itemize}
\item[(1)] For any $\ba\in\rel$, $\ba^*=\pr_I\ba$, $\bb\in\pr_I\rel$ 
such that $\ba^*\bb$ is a thin edge, there is 
$\bb'\in\rel$, $\pr_I\bb'=\bb$, such that $\ba\bb'$ is a thin edge of the same 
type. 
\item[(2)] 
If $\ba\bb$ is a thin edge in $\rel$ then $\pr_I\ba\,\pr_I\bb$
is a thin edge in $\pr_I\rel$ of the same type (including the 
possibility that $\pr_I\ba=\pr_I\bb$).
\item[(3)] 
For any $\ba\in\rel$, and an s- (as-, asm-) path 
$\vc\bb k\in\pr_I\rel$ with $\pr_I\ba=\bb_1$, there is an
s- (as-, asm-) path $\vc{\bb'}k\in\rel$ such that 
$\bb'_1=\ba$ and $\pr_I\bb'_k=\bb_k$.
\item[(4)] 
If $\vc\ba k$ is an s- (as-, asm-) path in $\rel$, then 
$\pr_I\ba_1\zd\pr_I\ba_k$ is an s- (as-, asm-) path in $\pr_I\rel$.
\item[(5)] 
For any maximal (as-maximal, u-maximal) (in $\pr_I\rel$) element 
$\bb\in\pr_I\rel$, there is $\bb'\in\rel$ which is maximal (as-maximal, 
u-maximal) in $\rel$ and such that $\pr_I\bb'=\bb$. In particular, 
$\pr_{[n]-I}\bb'$ is maximal (as-maximal, u-maximal) in 
$\pr_{[n]-I}\rel$.
\item[(6)] 
If $\ba$ is maximal (as-maximal, u-maximal) in $\rel$, then 
$\pr_I\ba$ is maximal (as-maximal, u-maximal) in $\pr_I\rel$.
\end{itemize}
\end{lemma}

We will need the following statement that easily follows from Lemma~\ref{lem:to-max}(3).

\begin{corollary}\label{cor:as-extension}
Let $\rel$ be a subdirect product of $\vc\zA n\in\cV$, $I\sse[n]$, $B,C$ as-components of $\pr_I\rel$ and $\pr_{[n]-I}\rel$, respectively, such that $\rel'=\rel\cap(B\tm C)\ne\eps$. Then $\rel'$ is a subdirect product of $B$ and $C$.
\end{corollary}

The following lemma considers a special case of maximal components (as well as as- and u-maximal components) in subdirect products, and is straightforward.

\begin{lemma}[Lemma~1.18, \cite{Bulatov24:maximal}]\label{lem:as-product}
Let $\rel$ be a subdirect product of $\zA_1,\zA_2\in\cV$, $B,C$ maximal components (as-components, u-maximal components) of 
$\zA_1,\zA_2$, respectively, and $B\tm C\sse\rel$. 
Then $B\tm C$ is a maximal component (as-component, u-maximal component) 
of $\rel$.
\end{lemma}

Lemma~\ref{lem:thin-properties} applied to paths implies

\begin{lemma}\label{lem:factor-as}
Let $\zA\in\cV$, let $\th$ be a congruence of $\zA$.
\begin{itemize}
\item[(1)] 
If $C$ is a maximal component (as-component, u-maximal component) of $\zA$, then $C\fac\th$ is a maximal component (as-component, u-maximal component) of $\zA\fac\th$.
\item[(2)] 
If $C$ is a maximal component (as-component, u-maximal component) of $\zA\fac\th$, then every $\oa\in C$ contains a maximal (as-maximal, u-maximal) element of $\zA$.
\item[(3)]
If $\oa\sqq_{\zA\fac\th}\ob$ (respectively, $\oa\sqq_{\zA\fac\th}^{as}\ob$, $\oa\sqq_{\zA\fac\th}^{asm}\ob$), then for all $a\in\oa$ there exists $b\in\ob$ with $a\sqq_{\zA}b$ (respectively, $a\sqq_\zA^{as}b$, $a\sqq_\zA^{asm}b$).
\item[(4)]
If $a\sqq_{\zA}b$ (respectively, $a\sqq_\zA^{as}b$, $a\sqq_\zA^{asm}b$), then $\oa\sqq_{\zA\fac\th}\ob$ (respectively, $\oa\sqq_{\zA\fac\th}^{as}\ob$, $\oa\sqq_{\zA\fac\th}^{asm}\ob$). 
\end{itemize}
\end{lemma}

We complete this section with an auxiliary statement that will be needed
later.

\begin{lemma}\label{lem:as-type-2}
Let $\zA\in\cV$, $\al\prec\beta$, $\al,\beta\in\Con(\zA)$, let $B$ be a $\beta$-block and 
$\typ(\al,\beta)=\two$. Then $B\fac\al$ is term equivalent to a module. 
In particular,
every pair of elements of $B\fac\al$ is a thin affine edge in $\zA\fac\al$.
\end{lemma}

\begin{proof}
As $\zA$ is an idempotent algebra that generates a variety omitting type \one, 
and $(\alpha,\beta)$ is a prime interval in $\Con(\zA)$ of type \two, by 
\cite[Theorem~7.12]{Hobby88:structure}  there is a term operation of 
$\zA$ that is Mal'tsev on $B\fac\al$. Since $\beta$ is Abelian on $B\fac\al$, 
we get the result.
\end{proof}

%%%%%%%%%%%%%%%%%%%%%%%%%%%%%%%%
\subsection{Quasi-decomposition and quasi-majority}

We make use of the property of quasi-2-decomposability proved in
\cite{Bulatov24:maximal}.

\begin{theorem}[The Quasi-2-Decomposition Theorem, Corollary~4.3, 
\cite{Bulatov24:maximal}]\label{the:quasi-2-decomp}
Let\lb $\vc\zA n\in\cV$.
If $\rel$ is a subdirect product of $\zA_1\zd\zA_n$ and tuple $\ba$ is such that
$\pr_J\ba\in\pr_J\rel$ for any $J\sse[n]$, $|J|=2$, then there is a tuple $\bb\in\rel$ with $\pr_J\ba\sqq^{as}\pr_J\bb$ for any $J\sse[n]$, $|J|=2$. 
\end{theorem}

One useful implication of the Quasi-2-Decomposition 
Theorem~\ref{the:quasi-2-decomp} is the 
existence of a term operation resembling a majority operation. 

\begin{theorem}[Theorem~4.4, \cite{Bulatov24:maximal}]%
\label{the:pseudo-majority}
There is a term operation $\maj$ 
of $\cV$ such that for any $\zA\in\cV$ and any 
$a,b\in\zA$, it holds $a\sqq^{as}\maj(a,a,b),\maj(a,b,a),\maj(b,a,a)$.

In particular, if $a$ is as-maximal, then $\maj(a,a,b),\maj(a,b,a),\maj(b,a,a)$ 
belong to $\as(a)$.
\end{theorem}

A function $\maj$ satisfying the properties from 
Theorem~\ref{the:pseudo-majority} is called a 
\emph{quasi-majority operation}.

%%%%%%%%%%%%%%%%%%%%%%%%%%%%%%%%%%%%
\subsection{Rectangularity}\label{sec:rectangularity}

Let $\rel$ be a subdirect product of (arbitrary) algebras $\zA_1,\zA_2$. By 
$\rel[c], \rel^{-1}[c']$ for $c\in\zA_1,c'\in\zA_2$ we denote 
the sets $\{b\mid (c,b)\in\rel\}, \{a\mid (a,c')\in\rel\}$, respectively, and for 
$C\sse\zA_1,C'\sse\zA_2$ we use $\rel[C]=\bigcup_{c\in C}\rel[c]$, 
$\rel^{-1}[C']=\bigcup_{c'\in C'}\rel^{-1}[c']$, respectively. Binary relations 
$\tol_1,\tol_2$ on $\zA_1,\zA_2$ given by 
$\tol_1(\rel)=\{(a,b)\mid \rel[a]\cap\rel[b]\ne\eps\}$ and
$\tol_2(\rel)=\{(a,b)\mid \rel^{-1}[a]\cap\rel^{-1}[b]\ne\eps\}$, 
respectively, are called \emph{link tolerances} of $\rel$. 
They are tolerances of $\zA_1$, $\zA_2$, respectively, that is, invariant reflexive 
and symmetric relations. The transitive closures $\lnk_1,\lnk_2$ 
of $\tol_1(\rel),\tol_2(\rel)$ are called 
\emph{link congruences}, and they are, indeed,
congruences. Relation $\rel$ is said to be \emph{linked} 
if the link congruences are full congruences.

\begin{lemma}[Lemma~3.1, \cite{Bulatov24:maximal}]%
\label{lem:as-rectangularity} 
Let $\rel$ be a subalgebra of $\zA_1\tm\zA_2$, $\zA_1,\zA_2\in\cV$, and let $a\in\zA_1$ and 
$B=\rel[a]$. For any $b\in\zA_1$ such that $ab$ is a thin edge, and any 
$c\in\rel[b]\cap B$, $d\in\rel[b]$ whenever $c\sqq^{as}d$ in $B$.
\end{lemma}

Any subalgebra $\rel$ of a direct product of Mal'tsev algebras satisfies 
the \emph{rectangularity} property: if $(a,c),(a,d),(b,c)\in\rel$ then 
$(b,d)\in\rel$. This implies in particular that for any $\lnk_1$-block 
$B_1$ and any $\lnk_2$-block $B_2$, it holds $B_1\tm B_2\sse\rel$
whenever $\rel\cap(B_1\tm B_2)\ne\eps$. The following two 
statements proved in \cite{Bulatov24:maximal} make this property 
more general.

\begin{prop}[Corollary~3.4, Lemma~3.2 \cite{Bulatov24:maximal}]%
\label{pro:linkage-rectangularity}
Let $\rel$ be a subdirect product of $\zA_1,\zA_2\in\cV$.\\[2mm] 
(1) Let $\lnk_1,\lnk_2$ be the link congruences of $\rel$, and let $B_1,B_2$ be 
as-components of an $\lnk_1$-block and an $\lnk_2$-block, respectively, such that $\rel\cap(B_1\tm B_2)\ne\eps$.
Then $B_1\tm B_2\sse\rel$.\\[2mm]
(2) Let $a\in\zA_1$ and $B_1,B_2$ as-components of $\zA_1,\zA_2$, respectively, such that $\rel\cap(B_1\tm B_2)\ne\eps$ and $\{a\}\tm B_2\sse\rel$. Then $B_1\tm B_2\sse\rel$.
\end{prop}

\begin{prop}[see Proposition~3.5, \cite{Bulatov24:maximal}]%
\label{pro:umax-rectangular}
Let $\rel$ be a subdirect product of $\zA_1,\zA_2\in\cV$, 
$\lnk_1,\lnk_2$ the link congruences, and let $B_1$ be an 
as-component of an $\lnk_1$-block and $B_2=\rel[B_1]$. Then 
$B_1\tm\umax(B_2)\sse\rel$, $B_1\tm\amax(B_2)\sse\rel$, and $B_1\tm\max(B_2)\sse\rel$.\footnote{Note that as $B_2$ is not necessarily a subalgebra, the latter two inclusions do not follow from the first by Corollary~\ref{cor:amax-umax}}
\end{prop}

\begin{proof}
The inclusion $B_1\tm\umax(B_2)\sse\rel$ is Proposition~3.5, \cite{Bulatov24:maximal}, so we need to prove the remaining two. Let $(a_1,a_2)\in\rel$ such that $a_1\in B_1$ and $a_2\in\amax(B_2)$. By the first statement there is $b_2\in B_2$ (arbitrary from $\umax(B_2)$) such that $B_1\tm\{b_2\}\sse\rel$, in particular, $(a_1,b_2)\in\rel$. Consider the subalgebra $\relo$ of $\rel$ generated by $(a_1,a_2)$ and $(a_1,b_2)$. Let $(c_1,c_2)\in\amax(\relo)\sse\umax(\relo)$ be such that $(a_1,a_2)\sqq^{as}_\relo(c_1,c_2)$; note that $c_1=a_1$. There is an asm-path $(a_1,b_2)=(a_1,d_1)\zd(a_1,d_\ell)=(a_1,c_2)$ in $\relo$. As $B_1\tm\{d_1\}\sse\rel$, applying Lemma~\ref{lem:as-rectangularity} to every edge of the path in turn we obtain $B_1\tm\{c_2\}\sse\rel$. Next, since $a_2$ is as-maximal in $B_2$ and $a_2\sqq^{as}_{B_2}c_2$, there is an as-path from $c_2$ to $a_2$ in $B_2$. Repeating the above argument we obtain $B_1\tm\{a_2\}\sse\rel$.

The proof of $B_1\tm\max(B_2)\sse\rel$ is essentially the same.
\end{proof}

We complete this section with a technical lemma that will be useful later.

\begin{lemma}\label{lem:as-square-congruence}
Let $\zA\in\cV$, $|\zA|>1$, be an algebra and $C$ its as-component such that $\zA=\Sg{C}$,
let $\rel=\zA\tm\zA=\Sg{C\tm C}$, and let $\al$ be a congruence 
of $\rel$. Then for some $a,b\in C$, $a\ne b$, 
the pair $(a,b)$ is as-maximal in an $\al$-block.
\end{lemma}

\begin{proof}

We start with a general claim.

\medskip

{\sc Claim.} If $\beta,\gm\in\Con(\rel)$ are such that 
$\beta\join\gm=\zo_\rel$,  then 
$\beta\red{C^2}\circ\gm\red{C^2}=\gm\red{C^2}\circ\beta\red{C^2}
=C^2\tm C^2$. 

\smallskip

Let $\rel_1\sse\rel\fac\beta\tm\rel$, $\rel_2\sse\rel\fac\gm\tm\rel$
be given by
$$
\rel_1=\{(a\fac\beta,a)\mid a\in\rel\},\qquad 
\rel_2=\{(a\fac\gm,a)\mid a\in\rel\}.
$$ 
Consider a subdirect product $\relo$ of $\rel\fac\beta\tm\rel\fac\gm$ defined 
as follows
$$
\rela(x,y,z)=\rel_1(x,z)\meet\rel_2(y,z),
$$
and $\relo=\pr_{12}\rela$. As is easily seen, for a $\beta$-block $B_1$ and 
a $\gm$-block $B_2$, $(B_1,B_2)\in\relo$ if and only if $B_1\cap B_2\ne\eps$. As 
$\beta\join\gm=\zo_\rel$, the relation $\relo$ is linked. By $C^2\fac\beta,C^2\fac\gm$ we denote the sets of $\beta$- and $\gm$-blocks of $\rel$ that intersect $C^2$. As $C$ is an as-component of $\zA$, by Lemma~\ref{lem:as-product} $C^2$ is an as-component of $\rel$. Then by Lemma~\ref{lem:factor-as} 
$C^2\fac\beta$ is an as-component of $\rel\fac\beta$ and $C^2\fac\gm$
is an as-component of $\rel\fac\gm$. Therefore, Proposition~\ref{pro:linkage-rectangularity}(1)
implies that $C^2\fac\beta\tm C^2\fac\gm\sse\relo$. Therefore 
for any $\beta$- and $\gm$-blocks $B_1,B_2$ such that $B_1\cap C^2\ne\eps$,
$B_2\cap C^2\ne\eps$ we have $B_1\cap B_2\ne\eps$. 

Now, consider the relation $\rela$ defined above.
For any $a\in C^2$ we have $(a\fac\beta,a\fac\gm,a)\in\rela$. By Lemma~\ref{lem:as-product} $C^2$ is an as-component of $\rel$, and therefore, again by Lemma~\ref{lem:as-product}, $C^2\tm C^2$ is an as-component of $\rel^2$.
By Lemma~\ref{lem:factor-as} $C^2\fac\beta\tm C^2\fac\gm$ is an as-component of $\rel\fac\beta\tm\rel\fac\gm$. Since $C^2\fac\beta\tm C^2\fac\gm\sse\relo$, by Corollary~\ref{cor:subalgebra-path} $C^2\fac\beta\tm C^2\fac\gm$ is an as-component of $\relo$. Thus,
for any $(B_1,B_2)\in C^2\fac\beta\tm C^2\fac\gm$, we have 
$(a\fac\beta,a\fac\gm)\sqq^{as}(B_1,B_2)$ in $\relo$. By the Maximality 
Lemma~\ref{lem:to-max}(3) there is $b\in C^2$, $a\sqq^{as} b$, such that 
$(B_1,B_2,b)\in\rela$. The element $b$ belongs to $B_1\cap B_2\cap C^2$,
as required.

\medskip

Let $\beta$ be a maximal congruence of $\zA$ and $\gm_1=\beta\tm\zo_\zA$, $\gm_2=\zo_\zA\tm\beta$. As is easily seen, $\gm_1,\gm_2$ are 
maximal congruences of $\rel$, and $\gm_1\meet\gm_2=\beta\tm\beta$. Indeed, if $\dl\in\Con(\rel)$ is such that $\gm_1\subsetneq\dl$, there are $(a_1,b_1)\eqc\dl(a_2,b_2)$ such that $(a_1,a_2)\not\in\beta$. Therefore for any $c_1,c_2\in\zA$ there are $d_1,d_2\in\zA$ with $(c_1,d_1)\eqc\dl(c_2,d_2)$. Since $\gm_1\sse\dl$, this holds for any $d_1,d_2\in\zA$. The statement for $\gm_2$ is similar and $\gm_1\meet\gm_2=\beta\tm\beta$ is straightforward. There are two cases.

\medskip

{\sc Case 1.} $\al\join(\beta\tm\beta)=\zo_\rel$.

\smallskip

By the Claim for any $\beta$-blocks $B_1,B_2$ such that 
$B_1\cap C,B_2\cap C\ne\eps$ and an $\al$-block $B$ with 
$B\cap C^2\ne\eps$, we also
have $B\cap(B_1\tm B_2)\cap C^2\ne\eps$. As $C$ generates $\zA$ and $\beta$ is a proper congruence, $C$ intersects with at least two distinct $\beta$-blocks, and therefore one can assume $B_1\ne B_2$. Fix such $B_1,B_2$ and let $(a,b)\in B\cap(B_1\tm B_2)\cap C^2$. Let $(a',b')\in\amax(B)$ be such that $(a,b)\sqq^{as}_B(a',b')$. Since $a,b\in C$ and $a\sqq^{as}_\zA a', b\sqq^{as}_\zA b'$, we have $a',b'\in C$. Let $(a',b')\in B'_1\tm B'_2$, where $B'_1,B'_2$ are $\beta$-blocks.
By Lemma~\ref{lem:thin-properties}(2) $(B_1,B_2)\sqq^{as}(B'_1,B'_2)$ in 
$\rel\fac{\beta\tm\beta}$. As $(B_1,B_2)$ is as-maximal in 
$\rel\fac{\beta\tm\beta}$, there is also an as-path from 
$(B'_1,B'_2)$ to $(B_1,B_2)$ in $\rel\fac{\beta\tm\beta}$. Let $(B'_1,B'_2)=(B_1^1,B_2^1),(B_1^2,B_2^2)\zd (B_1^\ell,B_2^\ell)=(B_1,B_2)$ be this path. Since $a',b'\in C$, by Lemma~\ref{lem:thin-properties}(1) every block $B_1^i\tm B_2^i$ contains a pair from $C^2$. By the Claim it means that $B\cap(B_1^i\tm B_2^i)\cap C^2\ne\eps$ for every $i\in[\ell]$. Consider the sequence $B\cap(B_1^1\tm B_2^1)\zd B\cap(B_1^\ell\tm B^\ell_2)$ of blocks of the congruence $\beta'=(\beta\tm\beta)\red B$. Since $B\fac{\beta'}$ is a subalgebra of $\rel\fac{\beta\tm\beta}$ \footnote{Strictly speaking instead of $B\fac{\beta'}$ we need to use $B^{\beta\tm\beta}\fac{\beta\tm\beta}$, where $B^{\beta\tm\beta}$ is the union of all $\beta\tm\beta$-blocks of $\rel$ intersecting with $B$, see \cite[Definition 6.16]{Burris81:universal}.}, by Lemma~\ref{lem:thin-properties}(2) this sequence  is an as-path in $B\fac{\beta'}$. By Lemma~\ref{lem:thin-properties}(1) there are $(a',b')=(a^1,b^1)\zd(a^\ell,b^\ell)$ such that $(a^i,b^i)\in B\cap(B_1^i\tm B_2^i)$, which form an as-path in $B$. Since $(a',b')\in\amax(B)$, we also have $(a^\ell,b^\ell)\in\amax(B)$, and as $a',b'\in C$, we have $a^\ell,b^\ell\in C$. Finally, as $(a^\ell,b^\ell)\in B_1\tm B_2$ and $B_1\ne B_2$, we have $a^\ell\ne b^\ell$, as required.

\medskip

{\sc Case 2.} $\al\join(\beta\tm\beta)\ne\zo_\rel$.

\smallskip

In this case consider $\zA'=\zA\fac\beta$, $\rel'=\rel\fac{\beta\tm\beta}$, 
$\al'=(\al\join(\beta\tm\beta))\fac{\beta\tm\beta}$; note that $\zA'$ is a simple idempotent algebra, and as
$\rel=\zA\tm\zA$, we have $\rel'=\zA'\tm\zA'$. The idea here is to find an $\al'$-block $D$ that does not contain pairs of the form $(a,a)$ and find a pair $(a,b)\in D\cap C^2$. Then, if $B$ denotes the $\al$-block containing $(a,b)$, then any $(a',b')\in\amax(B)$ such that $(a,b)\sqq^{as}_B(a',b')$ satisfies the required conditions.  Indeed, $(a',b')\in C^2$, because $a\sqq^{as}_\zA a', b\sqq^{as}_\zA b'$, and $(a,b)\in C^2$. 

By \cite{Kearnes96:idempotent} and \cite{Valeriote90:finite,Szendrei87:simple} (see also \cite[Proposition~3.3]{Bulatov24:graph}) either $\zA'$ has an \emph{absorbing 
element} $a$, that is, $f(\vc ak)=a$ for any term operation $f$ of $\zA'$,
whenever $a_i=a$ for some essential variable $x_i$ of $f$, or $\zA'$ is a
module, or the only nontrivial congruences of $\zA'^2$ are 
$\gm'_1=\gm_1\fac{\beta\tm\beta}$, $\gm'_2=\gm_2\fac{\beta\tm\beta}$. 
Since $C$ is a nontrivial as-component, the first option is impossible. Indeed, if $a$ is an absorbing element, for any $b\in\zA'$ the pair $ba$ is a thin semilattice edge, and $\zA'$ has only one as-component, $\{a\}$. By Lemma~\ref{lem:factor-as} $C\fac\beta=\{a\}$, which contradicts the assumptions that $\beta$ is proper and $C$ generates $\zA$.
If $\zA'$ is a simple module, there is a unique proper nonzero congruence that is different from $\gm'_1,\gm'_2$ with 
$\Dl=\{(a,a)\mid a\in\zA'\}$ as a congruence block, the skew congruence. If $\al'$ is the 
skew congruence, let 
$D$ be any $\al\join(\beta\tm\beta)$-block different from $\Dl$. As $C$ is not contained in a $\beta$-block, $C^2$ is not contained in a $\al\join(\beta\tm\beta)$-block, and $D$ can be chosen such that $D\cap C^2\ne\eps$ (in fact, as $\zA'$ is a module, it is not difficult to see that any $\al\join(\beta\tm\beta)$-block satisfies this condition). Then we complete the proof as indicated in the beginning of Case~2.

So, suppose $\al\le\gm_1$. If $\al\le\gm_1\meet\gm_2=\beta\tm\beta$, 
choose a $\beta\tm\beta$-block $B_1\tm B_2$ such that $B_1\ne B_2$ and $B_1\cap C,B_2\cap C\ne\eps$;
clearly $B_1,B_2$ are $\beta$-blocks. Then for any $\al$-block 
$B\sse B_1\tm B_2$ such that $B\cap(B_1\tm B_2)\cap C^2\ne\eps$ we can find an element $(a,b)\in B\cap C^2$ and $a\ne b$ as required.

Finally, suppose $\al\not\le\gm_2$, then $\al\join\gm_2=\zo_\rel$.
Take an $\al$-block $B$, $B\cap C^2\ne\eps$, we have $B\sse B_1\tm\zA$ 
for some $\beta$-block $B_1$. Moreover, by the Claim for any $\beta$-block 
$B_2$ with $B_2\cap C\ne\eps$ there is $(a,b)\in B\cap C^2$ such 
that $b\in B_2$. Choose $B_2\ne B_1$. Then we complete the proof as in Case~1.
\end{proof}

%%%%%%%%%%%%%%%%%%%%%%%%%%%%%%%%%%%
%%%%%%%%%%%%%%%%%%%%%%%%%%%%%%%%%%%
\section{Separating congruences}\label{sec:polynomials-maximality}

In this section we introduce and study the relationship between prime
 intervals in the congruence lattice of an algebra, or in the congruence 
lattices of factors in subdirect products. It was first introduced in 
\cite{Bulatov02:maltsev-3-element}
and used in the CSP research 
in~\cite{Bulatov17:semilattice,Bulatov19:semilattice}.

%%%%%%%%%%%%%%%%%%%%%%%%%%%%%%%%%
\subsection{Special polynomials, mapping pairs}\label{special-polynomials}

We start with several technical results. They demonstrate the connection 
between minimal sets of an algebra $\zA\in\cV$ and the structure of its 
graph $\cG_{asm}(\zA)$. Let $\zA\in\cV$ be an algebra and let 
$\relo_{ab}^\zA$, $a,b\in\zA$, denote the subdirect 
product of $\zA^2$ generated by $\{(x,x)\mid x\in\zA\}\cup\{(a,b)\}$.

\begin{lemma}\label{lem:Qab-tolerance}
Let $\zA\in\cV$ and $a,b\in\zA$.
\begin{itemize}
\item[(1)] 
$\relo^\zA_{ab}=\{(f(a),f(b))\mid f\in\Polo(\zA)\}$.
\item[(2)] 
For any $f\in\Polo(\zA)$, $(f(a),f(b))\in\tol_1(\relo^\zA_{ab})$. 
In particular, $\lnk_1(\relo^\zA_{ab})=\Cg{a,b}$; denote this congruence 
by $\al$.
\item[(3)] 
$\relo^\zA_{ab}\sse\al$.
\item[(4)] 
Let $B$ be an $\al$-block, and $C_1,C_2$ as-components of 
$B$ such that $f(a)\in C_1$ and $f(b)\in C_2$ for a 
polynomial $f\in\Polo(\zA)$. Then $C_1\tm C_2\sse\relo_{ab}^\zA$.
\end{itemize}
\end{lemma}

\begin{proof}
(1) follows directly from the definitions.

(2) Take $f\in\Polo(\zA)$, then $\cl{f(a)}{f(b)}\in\relo^\zA_{ab}$ by item (1) 
and $\cl{f(b)}{f(b))}\in\relo^\zA_{ab}$ since $\{(x,x)\mid x\in\zA\}\sse\relo^\zA_{ab}$ by definition. Thus $(f(a),f(b))\in\tol_1(\relo^\zA_{ab})$. This implies $\Cg{a,b}\sse\lnk_1(\relo^\zA_{ab})$. Then (1) implies that $\lnk_1(\relo^\zA_{ab})\sse\Cg{a,b}$.

%% Take $f\in\Polo(\zA)$ and let $f(x)=g(x,\vc ak)$ for a term operation $g$ of $\zA$. Then $\cl{f(a)}{f(b)}=g\left(\cl ab,\cl{a_1}{a_1}\zd\cl{a_k}{a_k}\right)\in\relo^\zA_{ab}$

(3) follows from (1), and 
(4) follows from (2),(3), and Proposition~\ref{pro:linkage-rectangularity}(1).
\end{proof}

Lemma~\ref{lem:Qab-tolerance}(4) immediately implies

\begin{corollary}\label{cor:mapping-pairs}
Let $\zA\in\cV$ and $\al\in\Con(\zA)$ such that $\zz\prec\al$. Then for any $a,b\in\zA$ with 
$a\eqc\al b$, $a\ne b$, and any $c,d\in\zA$ such that 
$c\eqc\al d$ and $c,d$ belong to the same
as-component of $c\fac\al$, there is $f\in\Polo(\zA)$ 
such that $c=f(a)$, $d=f(b)$.
\end{corollary}

\begin{corollary}\label{cor:max-min-set}
Let $\zA\in\cV$, $\al\in\Con(\zA)$ such that $\zz\prec\al$, and let $c,d\in\zA$, 
$c\eqc\al d$, $c\ne d$, be as-maximal in $B=c\fac\al$. 
\begin{itemize}
\item[(1)] 
If $c,d$ belong to the same as-component of $B$, 
then $\{c,d\}$ is a $(\zz,\al)$-subtrace.
\item[(2)] 
If there is a $(\zz,\al)$-subtrace $\{c',d'\}$ such that $c'\in\as_B(c)$
and $d'\in\as_B(d)$ then $\{c,d\}$ is a $(\zz,\al)$-subtrace as well.
\end{itemize}
\end{corollary}

\begin{proof}
(1) Take any $(\zz,\al)$-minimal set $U$, and 
$a,b\in U$ with $a\eqc\al b$, $a\ne b$. By Corollary~\ref{cor:mapping-pairs} there is 
$f\in\Polo(\zA)$ with $c=f(a)$, $d=f(b)$. By Lemma~\ref{lem:minimal-sets}(2)
$U'=f(U)$ is a $(\zz,\al)$-minimal set.

(2) Since $(c',d')\in\relo^\zA_{c'd'}$, we have 
$\relo^\zA_{c'd'}\cap(\as_B(c)\tm\as_B(d))\ne\eps$. By 
Lemma~\ref{lem:Qab-tolerance}(4), $\as_B(c)\tm\as_B(d)\sse\relo^\zA_{c'd'}$, in particular, $(c,d)\in\relo^\zA_{c'd'}$, hence there is a polynomial $f$ such that $f(c')=c,f(d')=d$. Let $U$ be a $(\zz,\al)$-minimal set containing $c',d'$. Then $f(U)$ is a $(\zz,\al)$-minimal set containing $c,d$.
\end{proof}

\begin{lemma}\label{lem:type23}
Let $\zA\in\cV$. For any $\al\in\Con(\zA)$ with $\zz\prec\al$ such that $|D|>1$ for some as-component $D$ of an $\al$-block, the prime interval $(\zz,\al)$ has type 
\two\ or \three.
\end{lemma}

\begin{proof}
Let $a,b\in D$ for an as-component $D$ of an $\al$-block. Then by 
Corollary~\ref{cor:mapping-pairs} there is a polynomial $f$ such that 
$f(a)=b$ and $f(b)=a$. Also, $a,b$ belong to some $(\zz,\al)$-minimal 
set. This rules out types \four\ and \five. Since $\zA$ omits type $\one$, 
this only leaves types \two\ and \three.
\end{proof}

\begin{lemma}\label{lem:type2-condition}
Let $\zA\in\cV$, and let $\al\in\Con(\zA)$ with $\zz\prec\al$ be such that some $\al$-block contains 
a semilattice or majority edge. Then the prime interval $(\zz,\al)$ has type 
\three, \four\, or \five.
\end{lemma}

\begin{proof}
We need to show that $(\zz,\al)$ does not have type \two. Let $B$ be the 
$\al$-block containing a semilattice or majority edge. Then $B$ contains a 
non-Abelian subalgebra. By Lemma~\ref{lem:as-type-2} if $(\zz,\al)$ has type \two, then $B$ is term equivalent to a module. Since every subalgebra of a module is also term equivalent to a module, $(\zz,\al)$ cannot have type \two. 
\end{proof}

%%%%%%%%%%%%%%%%%%%%%%%%%%%%%%%%
\subsection{Separation}\label{sec:separation}

% Separation, single algebra
The first several definitions and results of this section are valid for arbitrary algebras, not only for algebras from $\cV$.

Let $\zA$ be an algebra, and let $\al\prec\beta$ and $\gm\prec\dl$ be prime
intervals in $\Con(\zA)$. We say that $(\al,\beta)$ can be 
\emph{separated} from 
$(\gm,\dl)$  if there is a unary polynomial $f\in\Polo(\zA)$ such that 
$f(\beta)\not\sse\al$, but $f(\dl)\sse\gm$. The polynomial $f$ in this case is 
said to \emph{separate} $(\al,\beta)$ from $(\gm,\dl)$. 

% Separation, relations
Since we often consider relations rather than single algebras, we also introduce 
separability in a slightly different way. Let $\rel$ be a subdirect product of 
$\zA_1\zd\zA_n$. Let $I,J\sse[n]$ and let
$\al_I\prec\beta_I$, $\al_J\prec\beta_J$ be prime intervals in $\Con(\pr_I\rel)$ and 
$\Con(\pr_J\rel)$, respectively. Interval $(\al_I,\beta_I)$ can be separated from 
$(\al_J,\beta_J)$ 
if there is a unary polynomial $f$ of $\rel$ such that $f(\beta_I)\not\sse\al_I$ but 
$f(\beta_J)\sse\al_J$. Similarly, the polynomial $f$ in this case is said to 
\emph{separate} $(\al_I,\beta_I)$ from $(\al_J,\beta_J)$.

% Properties
First, we observe a connection between separation in a single algebra and in 
relations. Clearly $\al_I,\al_J,\beta_I,\beta_J$ give rise to congruences of $\rel$ and therefore separation as defined above can be expressed as separation of intervals in $\Con(\rel)$, although one needs to be careful to preserve the primality of congruence intervals. But we can go a bit further.

\begin{lemma}\label{lem:separation-separation}
\begin{itemize}
\item[(1)] 
Let $\rel$ be the binary equality relation on an arbitrary idempotent algebra $\zA$. Let $\al_1=\al,\beta_1=\beta$
be viewed as congruences of the first factor of $\rel$, and 
$\al_2=\gm,\beta_2=\dl$ as congruences of the second factor of $\rel$. 
The prime interval $(\al,\beta)$ can be separated from $(\gm,\dl)$ as 
intervals in $\Con(\zA)$ if and only if $(\al_1,\beta_1)$ can be separated 
from $(\al_2,\beta_2)$ in $\rel$.
\item[(2)] 
Let $\rel$ be a subdirect product of arbitrary idempotent algebras $\vc\zA n$, $I,J\sse[n]$, and $\rel^*$ constructed as follows: $K=I\cap J$, say, $K=[k]$ and $K^*=\{n+1\zd n+k\}$, 
\[
\rel^*=\{(\vc an,\vc ak)\mid (\vc an)\in\rel\},
\]
$I^*=I,J^*=(J-K)\cup K^*$. Then $\pr_I\rel,\pr_J\rel$ are isomorphic to $\pr_{I^*}\rel^*,\pr_{J^*}\rel^*$; let $\vf_I,\vf_J$ be the natural isomorphisms. 
\begin{itemize}
\item
For every polynomial $f$ of $\rel$ there is a polynomial $f^*$ of $\rel^*$ such that $f^*$ acts on $\pr_{I^*}\rel^*,\pr_{J^*}\rel^*$ as $\vf_I(f),\vf_J(f)$, respectively, and for every polynomial $f$ of $\rel^*$ there is a polynomial $f'$ of $\rel$ such that $f'^*=f$. 
\item
Let $\al_I,\beta_I\in\Con(\pr_I\rel)$, $\al_J,\beta_J\in\Con(\pr_J\rel)$ be such that $\al_I\prec\beta_I$, $\al_J\prec\beta_J$ and $\al^*_{I^*}=\vf_I(\al_I)$, $\beta^*_{I^*}=\vf_I(\beta_I)$, $\al^*_{J^*}=\vf_J(\al_J)$, $\beta^*_{J^*}=\vf_J(\beta_J)$ be congruences of $\pr_{I^*}\rel^*,\pr_{J^*}\rel^*$. Then $\al^*_{I^*}\prec\beta^*_{I^*}$, $\al^*_{J^*}\prec\beta^*_{J^*}$, and $(\al^*_{I^*},\beta^*_{I^*})$ can be separated from $(\al^*_{J^*},\beta^*_{J^*})$ if and only if $(\al_I,\beta_I)$ can be separated from $(\al_J,\beta_J)$.
\end{itemize}
\end{itemize}
\end{lemma}

\begin{proof}
(1) Note that for any polynomial $f$ of $\rel$ its action on the first and second projections of 
$\rel$ is the same polynomial of $\zA$. Therefore $(\al,\beta)$ can be 
separated from $(\gm,\dl)$ in $\Con(\zA)$ if and only if there is 
$f\in\Polo(\zA)$, $f(\beta)\not\sse\al$ while $f(\dl)\sse\gm$. This condition 
can be expressed as follows: there is $f\in\Polo(\rel)$, 
$f(\beta_1)\not\sse\al_1$ while $f(\beta_2)\sse\al_2$, which precisely 
means that $(\al_1,\beta_1)$ can be separated from $(\al_2,\beta_2)$ 
in $\rel$.

\smallskip

(2) The isomorphism part of the lemma is straightforward, while the separation part can be proved in a way similar to item (1).
\end{proof}

In what follows when proving results about separation we will always assume 
that we deal with a relation --- a subdirect product ---  and that the prime 
intervals in question are from congruence lattices of projections of the 
subdirect product that do not overlap. If this is not the case, one
can duplicate some of the factors and apply 
Lemma~\ref{lem:separation-separation}(2).

Let $\rel$ be a subdirect product of $\zA_1\zd\zA_n$, $I\sse[n]$,
and let $f$ be a polynomial of $\pr_I\rel$, that is, there are a term operation 
$g$ of $\rel$ and $\vc\ba k\in\pr_I\rel$ such that 
$f(\vc x\ell)=g(\vc x\ell,\vc\ba k)$. The tuples $\ba_i$ can be extended
to tuples $\ba'_i\in\rel$. Then the polynomial of $\rel$ given by
$f(\vc x\ell)=g(\vc x\ell,\vc{\ba'}k)$ is said to be an 
\emph{extension} of $f$ to a polynomial of $\rel$.

\begin{lemma}\label{lem:min-set-separation}
Let $\vc\zA n$ be arbitrary algebras. Let $\rel$ be a subdirect product of $\zA_1\zd\zA_n$, $I,J\sse[n]$, and 
$\al_I\prec\beta_I$, $\al_J\prec\beta_J$ for $\al_I,\beta_I\in\Con(\pr_I\rel)$, 
$\al_J,\beta_J\in\Con(\pr_J\rel)$.
\begin{itemize}
\item[(1)] 
Suppose that $(\al_I,\beta_I)$ can be separated from $(\al_J,\beta_J)$. Then for any $(\al_I,\beta_I)$-minimal set $U$ there is a unary idempotent polynomial $f$ of $\rel$ that separates $(\al_I,\beta_I)$ from $(\al_J,\beta_J)$ and such that $U=f(\pr_I\rel)$.
\item[(2)] 
Let $\relo\sse\rel$ also be a subdirect product of $\zA_1\zd\zA_n$ with $\pr_I\relo=\pr_I\rel$, $\pr_J\relo=\pr_J\rel$. If $(\al_I,\beta_I)$ can be separated from $(\al_J,\beta_J)$ in $\relo$, $(\al_I,\beta_I)$ can also be separated from $(\al_J,\beta_J)$ in $\rel$.
\end{itemize}
\end{lemma}

\begin{proof}
(1) Let $g$ be a polynomial of $\rel$ separating $(\al_I,\beta_I)$ from $(\al_J,\beta_J)$. 
Since $g(\beta_I)\not\sse\al_I$, by Lemma~\ref{lem:minimal-sets}(6) there 
is an $(\al_I,\beta_I)$-minimal set $U'$ such that $g(\beta_I\red{U'})\not\sse\al_I$.
Let $V=g(U')$, by Lemma~\ref{lem:minimal-sets}(2) $V$ is a 
$(\al_I,\beta_I)$-minimal set. Let $h,h_0$ be unary 
polynomials such that $h_0$ is a polynomial isomorphism from $U$ to $U'$, and $h$ maps $V$ onto $U$ and $h\circ g\circ h_0\red U$ is the 
identity mapping. Let also $h'$ be an extension of $h$ to a polynomial 
of $\rel$. Let $p$ be an idempotent unary
polynomial of $\pr_I\rel$ such that $p(\pr_I\rel)=U$ and let $p'$ be an extension of $p$
to a polynomial of $\rel$. If $(h'\circ g)^m$ is an idempotent polynomial of $\rel$ (the idempotent power of $h'\circ g$),
then let $q=p'\circ (h'\circ g)^m$. Now we can choose $f$ as the idempotent
power of $q$. By the choice of $q$ we have $q(\pr_I\rel)=U$,
while on the projection on $J$ the polynomial $q$ collapses $\beta_J$ to $\al_J$, so the idempotent power of $q$ has all the required properties.

\smallskip

(2) Let a polynomial $g$ of $\relo$ separate $(\al_I,\beta_I)$ from $(\al_J,\beta_J)$ in $\relo$. Then, as $\relo\sse\rel$, $g$ gives rise to a polynomial of $\rel$ and, as $\pr_I\relo=\pr_I\rel, \pr_J\relo=\pr_J\rel$, its action on $\pr_I\rel$ and $\pr_J\rel$ is the same. The result follows.
\end{proof}

%%%%%%%%%%%%%%%%%%%%%%%%%%%%%%%
\subsection{Chaining}\label{sec:chaining-def}

For a subdirect product $\rel\sse\zA_1\tms\zA_n$ the relation `cannot be 
separated' on prime intervals of the $\zA_i$'s  is clearly reflexive and
transitive. Almost all of our main results require that this relation is also symmetric. If the algebras $\zA_i$ are Mal'tsev, this relation is symmetric (for partial results see \cite{Bulatov02:maltsev-3-element,Bulatov17:semilattice}). 
Moreover, it can be shown that it remains `almost' symmetric when the 
$\zA_i$'s contain no majority edges. In the general case however the situation 
is more complicated. Next we introduce conditions that make the 
`cannot be separated' relation to some extent symmetric, at least in what 
concerns our needs, as it will be demonstrated in Theorem~\ref{the:relative-symmetry}. From now on we again assume that algebras we work with belong to $\cV$.

For an algebra $\zA\in\cV$, $\al\in\Con(\zA)$, a set $\cU$ of unary polynomials of $\zA$, and $B\sse\zA^2$,
we denote by $\Cgg{\zA,\al,\cU}B$ the transitive-symmetric closure
of the set $T(\al,B,\cU)=\{(f(a),f(b))\mid (a,b)\in B, f\in\cU\}\cup\al$. We will write $\{c,d\}\in T(\al,B,\cU)$ if $(c,d)\in T(\al,B,\cU)$ or $(d,c)\in T(\al,B,\cU)$. Let  
$\al,\beta\in\Con(\zA)$, $\al\le\beta$, and $D$ a subuniverse of $\zA$.
We say that $\al$ 
and $\beta$ are \emph{$\cU$-chained} with respect to $D$ if for any $a,b\in D$, $a\eqc\beta b$, such that $\beta=\Cgg\zA{\{(a,b)\}}\join\al$ and any $\beta$-block $B$ such that $B\cap\umax(D)\ne\eps$, we have 
$(\umax(B\cap D))^2\sse\Cgg{\zA,\al,\cU}{\{(a,b)\}}$. (Observe that by Lemma~\ref{lem:u-max-congruence} $\umax(B\cap D)\sse B\cap\umax(D)\sse D$.) Note that if $\beta=\Cgg\zA{\{(a,b)\}}\join\al$ for no $a,b\in D$, then $\al,\beta$ are trivially $\cU$-chained.

Let $\rel$ be a subdirect product of $\vc\zA n\in\cV$, $\beta_i\in\Con(\zA_i)$, let $B_i$ be a $\beta_i$-block for $i\in[n]$,
and let $\relov=\rel\cap\ov B$. A unary polynomial $f$ of $\rel$ is said to be \emph{$\ov B$-preserving} if $f(\relov)\sse\relov$. 

\begin{observation}\label{obs:B-preservation}
Since $\rel'$ is a block of a congruence of $\rel$, if $f(\relov)\cap\relov\ne\eps$ then $f$ is $\ov B$-preserving.
\end{observation}

We call an $n$-ary relation $\rel$ \emph{chained} with respect to $\ov\beta,\ov B$ if\\[2mm]
(Q1) for any $I\sse[n]$ and $\al,\beta\in\Con(\pr_I\rel)$ such that 
$\al\le\beta\le\ov\beta_I$, $\al,\beta$ are $\cU_{\ov B}$-chained with respect to 
$\pr_I\relov$, where $\cU_{\ov B}$ is the set of all 
$\ov B$-preserving polynomials of $\rel$;\\[2mm]
(Q2) for any $I,J\sse[n]$ and $\al,\beta\in\Con(\pr_I\rel)$, $\gm,\dl\in\Con(\pr_J\rel)$ (note that it may happen that $I\cap J\ne\eps$), such that $\al\prec\beta\le\ov\beta_I$, $\gm\prec\dl\le\ov\beta_J$, 
and $(\al,\beta)$ can be separated from $(\gm,\dl)$, the congruences 
$\al$ and $\beta$ are $\cU(\gm,\dl,\ov B)$-chained with respect to 
$\pr_I\relov$, where $\cU(\gm,\dl,\ov B)$ 
is the set of all $\ov B$-preserving polynomials $g$ of $\rel$ such that 
$g(\dl)\sse\gm$.

Next we make an observation about a link between the property of chaining and factor algebras. It is straightforward from the definitions.

\begin{lemma}\label{lem:chaining-inher}
Let $\vc\zA n,\rel,\ov\beta,\ov B$, and $\relov$ be as in the definition of chaining.
\begin{itemize}
\item[(1)] 
Let $K\sse[n]$ and consider $\rel$ as a subdirect product $\rel^*$ of $\pr_K\rel$ and $\zA_i$, $i\in[n]-K$. Let $\beta^*_i=\beta_i$ and $B^*_i=B_i$ for $i\in[n]-K$ and $\beta^*_K=\ov\beta_K$, $B^*_K=\ov B_K\cap\pr_K\rel$. If $\rel$ is chained with respect to $\ov\beta,\ov B$ then relation $\rel^*$ is chained with respect to $\ov\beta^*,\ov B^*$.
\item[(2)] 
Let $\rel$ be chained with respect to $\ov\beta,\ov B$ and let $\al_i\in\Con(\zA_i)$ and $\al_i\le\beta_i$, $i\in[n]$. Then $\rel\fac{\ov\al}$ is chained with respect to $\ov\beta\fac{\ov\al},\ov B\fac{\ov\al}$.
\end{itemize}
\end{lemma}

\begin{proof}
(1) Suppose $K=[k]$. Let $\rel$ be chained with respect to $\ov\beta,\ov B$. Let $I,J\sse\{K,k+1\zd n\}$ and $\al,\beta\in\Con(\pr_I\rel^*)$, $\gm,\dl\in\Con(\pr_J\rel^*)$ as in the definition of chaining. Also, if $K\in I$ [if $K\in J$] set $I'=(I-\{K\})\cup K$ [respectively, $J'=(J-\{K\})\cup K$] and $I'=I$ [respectively, $J'=J$] otherwise. Then the congruences $\al,\beta,\gm,\dl$ viewed as congruences of $\pr_I\rel^*,\pr_J\rel^*$ satisfy the same conditions when viewed as congruences of $\pr_{I'}\rel,\pr_{J'}\rel$. 

(2) is straightforward from the definitions. 
\end{proof}

In many cases Lemma~\ref{lem:chaining-inher} allows us to replace large sets $I,J$ in the definition of chaining with singletons. 

We conclude this subsection with an auxiliary statement.
For an algebra $\zA\in\cV$, an arbitrary congruence $\al\in\Con(\zA)$, $a',b'\in\zA$, and a set $\cU$ of unary polynomials of $\zA$ we use $T_\al(a,b,\cU)$, where 
$a=a'\fac\al,b=b'\fac\al$, to denote $T(\al,\{(a',b')\},\cU)\fac\al$. 

\begin{lemma}\label{lem:good-polys}
Let $\vc\zA n,\rel,\ov\beta,\ov B$, and $\relov$ be as in the definition of chaining, and $\rel$ is chained with respect to $\ov\beta,\ov B$. Let $I,J\sse[n]$, $\al\le\ov\beta_I$, and $\gm\prec\dl\le\ov\beta_J$. Let also $\cU\in\{\cU_{\ov B},\cU(\gm,\dl,\ov B)\}$.
\begin{itemize}
\item[(1)] 
Any constant polynomial from $\rel'$ belongs to $\cU$.
\item[(2)] 
If $f$ is a $k$-ary term operation of $\rel$ and $\vc gk\in\cU$, 
then $f(g_1(x)\zd g_k(x))\in\cU$. 
\item[(3)] 
For any $\ba,\bb\in\pr_I\rel'\fac\al$, a congruence $\beta\in\Con(\pr_I\rel)$, $\beta\le\ov\beta_I$, such that $\Cgg{\pr_I\rel\fac\al}{\{(\ba,\bb)\}}=\beta\fac\al$ and $(\al,\beta)$ can be separated from $(\gm,\dl)$ if $\cU=\cU(\gm,\dl,\ov B)$, and 
any $\beta\fac\al$-block $E$ with $E\cap\umax(\pr_I\rel'\fac\al)\ne\eps$, it holds 
that $(\bc,\bd)\in T_\al(\ba,\bb,\cU)$ for any $\bc,\bd$ from the same as-component 
of $E\cap \pr_I\rel'\fac\al$.
\item[(4)] 
For any $\ba,\bb\in \pr_I\rel'\fac\al$, a congruence $\beta\in\Con(\pr_I\rel)$, $\beta\le\ov\beta_I$, such that $\Cgg{\pr_I\rel\fac\al}{\{(\ba,\bb)\}}=\beta\fac\al$ and $(\al,\beta)$ can be separated from $(\gm,\dl)$ if $\cU=\cU(\gm,\dl,\ov B)$, any 
$\beta\fac\al$-block $E$ with $E\cap\umax(\pr_I\rel'\fac\al)\ne\eps$, and any 
$\bc,\bd\in\umax(E')$, where $E'=E\cap \pr_I\rel'\fac\al$, there is a 
sequence $\bc=\bc_1\zd \bc_k=\bd$ in $E'$ such that 
$\{\bc_i,\bc_{i+1}\}\in T_\al(\ba,\bb,\cU)$ for $i\in[k-1]$.\footnote{Note that since $(\ba,\bb)\in\beta\fac\al$, the same is true for any pair from $T_\al(\ba,\bb,\cU)$, implying that $\vc\bc k\in E'$.} In particular, if $|\umax(E')|>1$, for any $\bc\in\umax(E')$ there exists $\bc'\in E'$, $\bc\ne \bc'$, such that $\{\bc,\bc'\}\in T_\al(\ba,\bb,\cU)$.
\item[(5)] 
Let $\ba,\bb\in \pr_I\rel'\fac\al$. If $(\bc,\bd)\in T_\al(\ba,\bb,\cU)$ then $T_\al(\bc,\bd,\cU)\sse T_\al(\ba,\bb,\cU)$.
\item[(6)] 
Let $\beta\in\Con(\pr_I\rel)$, $\al\prec\beta\le\ov\beta_I$,
and $(\al,\beta)$ can be separated from $(\gm,\dl)$. Let $E$ be a 
$\beta\fac\al$-block such that $E\cap\umax(\pr_I\rel'\fac\al)\ne\eps$,   
and let $E'=E\cap \pr_I\rel'\fac\al$. Let also $|\umax(E')|>1$. For any $\ba,\bb\in E'$, $\ba\ne \bb$, there are $\ba',\bb'\in E'$, $\ba'\ne \bb'$, such that $(\ba',\bb')\in T_\al(\ba,\bb,\cU)$, $(\ba',\bb')\in T_\al(\ba',\bb',\cU)$, and $(\ba'',\bb'')\in T_\al(\ba'',\bb'')$ for all $(\ba'',\bb'')\in T_\al(\ba',\bb')$.\footnote{If $\cU=\cU_{\ov B}$ this claim is trivial for any $\ba',\bb'$, because the identity operation witnesses that $(\ba',\bb')\in T_\al(\ba',\bb',\cU_{\ov B})$. However if $\cU=\cU(\gm,\dl,\ov B)$, it does not have to be the case.} If $\ba\in\umax(E')$, $\ba',\bb'$ can be chosen such that $\ba\in\{\ba',\bb'\}$. Moreover, if $E'$ contains a nontrivial as-component, 
then there is a set $T\sse\beta\fac\al$ such that $T\sse T_\al(\bc,\bd,\cU)$ for any 
$\bc,\bd\in \pr_I\rel'\fac\al$, $\bc\ne \bd$, $\bc\eqc{\beta\fac\al} \bd$, and $T=T_\al(\bc',\bd',\cU)$ for some $\bc',\bd'$. 
\end{itemize} 
\end{lemma}

\begin{proof}
Items (1),(2) are straightforward, and (4) follows from the definitions and the assumption that $\rel$ is chained with respect to $\ov\beta,\ov B$.
Let $T(\ba,\bb)$ denote $T_\al(\ba,\bb,\cU)$. 

(3) By item (2) $\relo=T(\ba,\bb)\cap(\pr_I\rel'\fac\al)^2$ is a subalgebra of 
$(\pr_I\rel\fac\al)^2$. Since $\rel$ is chained, $\umax(E\cap \pr_I\rel'\fac\al)$ 
is a subset of a block of the link congruence of $\relo$. Therefore, as by Corollary~\ref{cor:amax-umax} $\amax(E\cap \pr_I\rel'\fac\al)\sse\umax(E\cap \pr_I\rel'\fac\al)$, for any as-component $D$ 
of $E\cap \pr_I\rel'\fac\al$ we have $D^2\cap\relo\ne\eps$ by item (1), and hence $D^2\sse\relo$ by 
Proposition~\ref{pro:linkage-rectangularity}(1).

(5) Let $\{\ba',\bb'\}\in T(\bc,\bd)$. Then there are polynomials $f,g\in\cU$ with
$\{\bc,\bd\}=f(\{\ba,\bb\})$ and $\{\ba',\bb'\}=g(\{\bc,\bd\})$. Then $g\circ f\in\cU$ by 
item (2) of the lemma and the definition, and $g\circ f(\{\ba,\bb\})=\{\ba',\bb'\}$.

(6) Consider $T(\ba,\bb)$ and let $(\ba_1,\bb_1)\in E'^2\cap T(\ba,\bb)$ be such that $T(\ba_1,\bb_1)$ is minimal possible with respect to inclusion. Let $(\ba',\bb')$ be any pair from $E'^2\cap T(\ba_1,\bb_1)$. Using (5) and the minimality of $T(\ba_1,\bb_1)$, it is easy to see that $(\ba',\bb')\in T(\ba_1,\bb_1)=T(\ba',\bb')$. Moreover, for the same reason, for any $(\ba_2,\bb_2)\in E'\cap T(\ba',\bb')$ we have $(\ba_2,\bb_2)\in T(\ba',\bb')=T(\ba_2,\bb_2)$. If $\ba\in\umax(E')$ by item (4) $\{\ba,\bb''\}\in T(\ba',\bb')$ for some $\bb''\in E'$. By what is proved above we have $T(\ba',\bb')=T(\ba,\bb'')$ or $T(\ba',\bb')=T(\bb'',\ba)$. For the second part of the claim take $\ba,\bb\in C$ where $C$ is a nontrivial as-component in $E'$. 
By item (3) $\{\ba,\bb\}\in T(\bc,\bd)$ for any appropriate $\bc,\bd$. Therefore by item (5)
$T=T(\ba,\bb)\sse T(\bc,\bd)$.
\end{proof}

%%%%%%%%%%%%%%%%%%%%%%%%%%%%%%%%
\subsection{Symmetricity of separation}\label{sec:symmetry-sep}

The following theorem establishes the weak symmetricity of separability
relation mentioned before.

%% \begin{theorem}\label{the:relative-symmetry}
%% Let $\rel$ be a subdirect product of $\vc\zA n\in\cV$, for each $i\in[n]$, 
%% $\beta_i\in\Con(\zA_i)$, $B_i$ a $\beta_i$-block such that $\rel$ is chained 
%% with respect to $\ov\beta,\ov B$; $\relov=\rel\cap\ov B$, 
%% $B'_i=\pr_i\relov$. Let also $i,j\in[n]$, $\al\prec\beta\le\beta_i$, $\gm\prec\dl=\beta_j$, 
%% where $\al,\beta\in\Con(\zA_i)$, $\gm,\dl\in\Con(\zA_j)$. 
%% If $B'_j\fac\gm$  has a 
%% nontrivial as-component $C_j$ and $(\al,\beta)$ can be separated from $(\gm,\dl)$,
%% then there is a $\ov B$-preserving polynomial $g$ of $\rel$ such that 
%% $g(\beta\red{B'_i})\sse\al$ and $g(\dl\red{B'_j})\not\sse\gm$. Moreover, for any
%% $c,d\in C_j$, $c\ne d$, the polynomial $g$ can be chosen such that $g(c)=c,g(d)=d$.
%% \end{theorem}

\begin{theorem}\label{the:relative-symmetry}
Let $\rel$ be a subdirect product of $\vc\zA n\in\cV$, for each $i\in[n]$, 
$\beta_i\in\Con(\zA_i)$, $B_i$ a $\beta_i$-block such that $\rel$ is chained 
with respect to $\ov\beta,\ov B$; $\relov=\rel\cap\ov B$, 
$B'_i=\pr_i\relov$. Let also $I,J\sse[n]$, $\al\prec\beta\le\ov\beta_I$, $\gm\prec\ov\beta_J$, 
where $\al,\beta\in\Con(\pr_I\rel)$, $\gm,\dl\in\Con(\pr_J\rel)$. Also, let $\dl=\ov\beta_J$.
If $\pr_J\rel'\fac\gm$  has a 
nontrivial as-component $C_J$ and $(\al,\beta)$ can be separated from $(\gm,\dl)$,
then there is a $\ov B$-preserving polynomial $g$ of $\rel$ such that 
$g(\beta\red{\pr_I\rel'})\sse\al$ and $g(\dl\red{\pr_J\rel'})\not\sse\gm$. Moreover, for any
$\bc,\bd\in C_J$, $\bc\ne \bd$, the polynomial $g$ can be chosen such that $g(\bc)=\bc,g(\bd)=\bd$.
\end{theorem}

By Lemmas~\ref{lem:separation-separation}(2) and~\ref{lem:chaining-inher}(1) we can assume $|I|=|J|=1$ and $I=\{1\},J=\{2\}$. Before proving Theorem~\ref{the:relative-symmetry} we describe two auxiliary constructions and prove several intermediate results. Fix a relation $\rel$ as specified in the statement of Theorem~\ref{the:relative-symmetry}. In the rest of Section~\ref{sec:symmetry-sep} we use the notation from Theorem~\ref{the:relative-symmetry}. As is easily seen, by Lemma~\ref{lem:chaining-inher}(2) we can assume that  $\al,\gm$ are equality relations. So, let $\al=\zz_1,\gm=\zz_2$.

We use the following two constructions. By $\relo^*_\ba(a,b,c,d)\sse\zA_1^2\tm\zA_2^2\tm\rel$ we denote the relation generated by 
$\{(a,b,c,d,\ba)\}\ \cup\{(x,x,y,y,\bz)\mid \bz\in\rel,\bz[1]=x,\bz[2]=y\}$, where $a,b\in B'_1,c,d\in B'_2$ and $\ba\in\relov$. 
Let $\relo_\ba(a,b,c,d)=\pr_{1234}\relo^*_\ba(a,b,c,d)$ and 
$\relo'_\ba(a,b,c,d)= \pr_{1234}(\relo^*_\ba(a,b,c,d)\cap(B'_1\tm B'_1\tm B'_2\tm B'_2\tm\ov B))$. 

Also, by $\relp^*_\ba(c_1,c_2,c_3)\sse\zA_2^3\tm\rel$ we denote the relation generated by 
$\{(c_1,c_2,c_3,\ba)\}\ \cup\{(x,x,x,\bz)\mid \bz\in\rel,\bz[2]=x\}$, where $c_1,c_2,c_3\in B'_2$ and $\ba\in\relov$. 
Let $\relp_\ba(c_1,c_2,c_3)=\pr_{123}\relp^*_\ba(c_1,c_2,c_3)$ and 
$\relp'_\ba(c_1,c_2,c_3)=\pr_{123}(\relp^*_\ba(c_1,c_2,c_3)\cap(B'_2\tm B'_2\tm B'_2\tm\ov B))$. 

Recall that $\relo^\zA_{ab}=\{(f(a),f(b))\mid f\in\Pol_1(\zA)\}$.

\begin{lemma}\label{lem:polys-max}
For any $a,b\in B'_1$, $a\eqc\beta b$, any $c,d,c_1,c_2,c_3\in B'_2$, and any $\ba\in\relov$,
\begin{itemize}\itemsep0pt
\item[(1)]
$\relo=\relo_\ba(a,b,c,d)$ is the set of quadruples $(f(a),f(b),f(c),f(d))$ 
for unary polynomials $f$ of $\rel$, and $\relo'=\relo'_\ba(a,b,c,d)$ is the set of quadruples $(f(a),f(b), f(c),f(d))$ for $\ov B$-preserving unary polynomials $f$ of $\rel$.\footnote{In fact, that we are looking for $\ov B$-preserving polynomials  is the reason we consider the big relation $\relo^*$ rather than starting directly with the 4-ary $\relo_\ba(a,b,c,d)$.} 
\item[(2)]
$\pr_1\relo'=\pr_2\relo'=B'_1$, $\pr_3\relo'=\pr_4\relo'=B'_2$, and $\pr_{12}\relo=\relo^{\zA_1}_{ab}$, $\pr_{34}\relo=\relo^{\zA_2}_{cd}$.
\item[(3)]
$\relp=\relp_\ba(c_1,c_2,c_3)$ is the set of triples $(f(c_1),f(c_2),f(c_3))$ 
for unary polynomials $f$ of $\rel$, and $\relp'=\relp'_\ba(c_1,c_2,c_3)$ is the set of triples $(f(c_1),f(c_2),f(c_3))$ for $\ov B$-preserving unary polynomials $f$ of $\rel$.
\item[(4)]
$\pr_1\relp'=\pr_2\relp'=\pr_3\relp'=B'_2$, and $\pr_{12}\relp=\relo^{\zA_2}_{c_1c_2}$, $\pr_{23}\relp=\relo^{\zA_2}_{c_2c_3}$, $\pr_{13}\relp=\relo^{\zA_2}_{c_1c_3}$.
\item[(5)]
Let $\relo'_{12}=\pr_{12}\relo',\relo'_{34}=\pr_{34}\relo'$. Assuming $c\ne d$ it holds that $C^2_2\sse\relo'_{34}$ and 
%% and let $\lnk_{12},\lnk_{34}$ denote the link congruences of $\relo'$ viewed as a subdirect product of $\relo'_{12}=\pr_{12}\relo'$ and $\relo'_{34}=\pr_{34}\relo'$.\footnote{Note that these congruences may be different from the link congruences of $\relo$ restricted to $\pr_{12}\relo\cap(B'_1\tm B'_1)$, $\pr_{34}\relo\cap(B'_2\tm B'_2)$, respectively.} 
there are $a',b'\in B'_1$ and $c',d'\in C_2$, $c'\ne d'$, such that $(a',b',c',d')\in\relo'$ and $(a',b')\in\amax(\relo'_{12}\cap E^2)$ for a $\beta$-block $E$ such that $E\cap\umax(B'_1)\ne\eps$. 
%% Moreover, if $(c',d')$ can be chosen from $\amax(D_{34})$, where $D_{34}$ is the $\lnk_{34}$-block containing $(c',d')$, then $(a',b')$ can also be chosen from $\amax(D_{12})$, where $D_{12}$ is the $\lnk_{12}$-block containing $(a',b')$, and $\{(a',b')\}\tm C_{34}\sse\relo'$, where $C_{34}$ is the as-component of $D_{34}$ containing $(c',d')$.
\end{itemize}
\end{lemma}

\begin{remark}
Observe that by Lemma~\ref{lem:polys-max}(1,3) the choice of $\ba\in\rel'$ is immaterial for $\relo_\ba(a,b,c,d),\relo'_\ba(a,b,c,d),\relp_\ba(c_1,c_2,c_3),\relp'_\ba(c_1,c_2,c_3)$. We therefore will often omit $\ba$ from this notation in the future.
\end{remark} 

\begin{proof}
(1) To prove the second part of (1) observe that $(a',b',c',d')\in\relo'$ if and only if there are a term operation $g(x,\vc yk)$ and $\bb_1=(a_1,a_1,c_1,c_1,\ba_1)\zd\lb\bb_k=(a_k,a_k,c_k,c_k,\ba_k)$ with $\ba_i\in\rel$, $\ba_i[1]=a_i,\ba_i[2]=c_i$ for $i\in[k]$, such that $(a',b',c',d',\bb')=g((a,b,c,d,\ba),\bb_1\zd\bb_k)$ and $\bb'\in\ov B$. Consider the unary polynomial $f(x)=g(x,\ba_1\zd\ba_k)$; clearly $\bb'=f(\ba)$. By Observation~\ref{obs:B-preservation} this means that %% As $\rel'=\rel\cap\ov B$ is a congruence block, the latter implies that 
$f$ is $\ov B$-preserving.  Finally, since $a_i=\ba_i[1],c_i=\ba_i[2]$, we have $a'=f(a),b'=f(b)$ in $\pr_1\rel$ and $c'=f(c), d'=f(d)$ in $\pr_2\rel$. The first part of (1) is proved in a similar way, except we do not need to care about polynomials being $\ov B$-preserving.

\smallskip

(2) follows from the definitions.

\smallskip

(3) and (4) are proved in a way similar to (1) and (2).

(5) Consider the relation $\rela=\relo'\cap(B'_1\tm B'_1\tm\Sg{C_2}\tm\Sg{C_2})$. By item (2) $\pr_{34}\relo=\relo^{\zA_2}_{cd}$. Furthermore, since $\rel$ is chained, we can apply Lemma~\ref{lem:good-polys}(3) with parameters $I,\al,\beta,\gm,\dl,\ba,\bb$ from Lemma~\ref{lem:good-polys}(3) set to be $\{2\},\zz_2,\beta_2,\zz_1,\beta,c,d$, respectively. We then obtain that $(c'',d'')\in\relo^{\zA_2}_{cd}$ for any $c'',d''$ from the same as-component of $B'_2$. In particular, $C^2_2\sse\pr_{34}\rela\sse\relo'_{34}$. Also, $\pr_{34}\rela$ is generated by $C^2_2$. 
Let 
\[
\rela\fac\beta=\{(a'\fac\beta,b'\fac\beta,c',d')\mid (a',b',c',d')\in\rela\}\sse B'_1\fac\beta\tm B'_1\fac\beta\tm\Sg{C_2}\tm\Sg{C_2}.
\]
Let $\lnk^\beta_{12},\lnk^\beta_{34}$ be the link congruences of $\rela\fac\beta$ treated as a subdirect product of $\pr_{12}\rela\fac\beta$, $\pr_{34}\rela\fac\beta$. By Lemma~\ref{lem:as-square-congruence} there is a $\lnk^\beta_{34}$-block $D^\beta_{34}$ such that for some $c',d'\in C_2$, $c'\ne d'$, the pair $(c',d')$ is as-maximal in $D^\beta_{34}$. Let $C^\beta_{34}$ be the as-component of $D^\beta_{34}$ containing $(c',d')$; as is easily seen, $C^\beta_{34}\sse C^2_2$. Note that, as $a\eqc\beta b$, $\pr_{12}\rela\fac\beta$ is the equality relation. Let $E$ be a $\beta$-block such that $(E,E)$ is as-maximal in the $\lnk^\beta_{12}$-block $D^\beta_{12}$ corresponding to $D^\beta_{34}$ (i.e.\ $\rela\fac\beta\cap(D^\beta_{12}\tm D^\beta_{34})\ne\eps$) and such that $(E,E,c',d')\in\rela\fac\beta$. Such a $\beta$-block exists by the Maximality Lemma~\ref{lem:to-max}(5). Note that $E\cap\umax(B'_1)\ne\eps$. Indeed, as $D^\beta_{34}$ contains $(c',d')$ that is as-maximal in $\pr_{34}\rela\fac\beta$, and, since $C^2_2\sse\pr_{34}\rela\fac\beta$, $(c',d')$ is also as-maximal in $\pr_{34}\relo'$, by the Maximality Lemma~\ref{lem:to-max}(5) $D^\beta_{12}$ contains elements as-maximal in $\pr_{12}\relo'\fac\beta=\{(e_1\fac\beta,e_2\fac\beta)\mid (e_1,e_2)\in\pr_{12}\relo'\}$. By Lemma~\ref{lem:u-max-congruence}  every element u-maximal in $D^\beta_{12}$ is u-maximal in $\pr_{12}\relo'\fac\beta$. As $(E,E)$ is as-maximal (and therefore u-maximal) in $D^\beta_{12}$, the pair $(E,E)$ is u-maximal in $\pr_{12}\relo'\fac\beta$. By the Maximality Lemma~\ref{lem:to-max}(6) this means that $E$ is u-maximal in $B'_1\fac\beta$, and again by Lemma~\ref{lem:u-max-congruence} $E$ contains elements u-maximal in $B'_1$. We need several observations concerning $\rela$ and the block~$E$.

(i) By Proposition~\ref{pro:linkage-rectangularity}(1) $\{(E,E)\}\tm C^\beta_{34}\sse\rela\fac\beta$. This means that for any $(c'',d'')\in C^\beta_{34}$ there is $(a',b')\in \rela^E_{12}$, where $\rela^E_{12}=\pr_{12}\rela\cap E^2$, such that $(a',b',c'',d'')\in\rela$. Also, by construction for any $(a',b')\in\rela^E_{12}$ there is $(c'',d'')\in D^\beta_{34}$ such that $(a',b',c'',d'')\in\rela$. 

(ii) Let $D'_{12}=\{(a',b')\in\pr_{12}\rela\mid (a'\fac\beta,b'\fac\beta)\in D^\beta_{12}\}$. We claim that for any $(a',b',c'',d'')\in\rela$ with $(a',b')\in D'_{12}$ and $(c'',d'')\in C^\beta_{34}$, and any $(a'',b'')\in \relo'_{12}$ with $(a',b')\sqq^{as}_{\relo'_{12}}(a'',b'')$, there is $(c^*,d^*)\in C^2_2$ such that $(a'',b'',c^*,d^*)\in\rela$. Moreover, if $(a'',b'')\in D'_{12}$ and $(a',b')\sqq^{as}_{D'_{12}}(a'',b'')$ then $(c^*,d^*)$ can be chosen from $C^\beta_{34}$. Indeed, the first claim follows from the Maximality Lemma~\ref{lem:to-max}(3) applied to $\relo'$: there is $(c^*,d^*)\in\relo'_{34}$ such that $(a'',b'',c^*,d^*)\in\relo'$ and $(c'',d'')\sqq^{as}_{\relo'_{34}}(c^*,d^*)$. Then $(c^*,d^*)\in C^2_2$, and so $(a'',b'',c^*,d^*)\in\rela$.
Assuming $(a',b')\sqq^{as}_{D'_{12}}(a'',b'')$, again by the Maximality Lemma~\ref{lem:to-max}(3) the as-path from $(c'',d'')$ to $(c^*,d^*)$ lies in $D^\beta_{34}$. Therefore $(c^*,d^*)\in C^\beta_{34}$. 

(iii) In particular, (ii) implies that if $(a',b')\in\amax(\rela^E_{12})$, $(a',b',c'',d'')\in\rela$ for some $(c'',d'')\in C^\beta_{34}$, and $(a',b')\sqq^{as}_{\relo'_{12}\cap E^2}(a'',b'')$, then $(a'',b'')\in\pr_{12}\rela$. Indeed, the path from $(a',b')$ to $(a'',b'')$ can be extended in $\relo'$ to a path from $(a',b',c'',d'')$ to some $(a'',b'',c^*,d^*)\in\relo'$. Since $c^*,d^*\in C_2$, $(a'',b'')\in\pr_{12}\rela$. This also implies that $(a',b')\in\amax(\relo'_{12}\cap E^2)$.

(iv) The last observation we need is that by Lemma~\ref{lem:factor-as}(2) and the choice of $E$, $\rela^E_{12}$ contains a pair $(a',b')\in\amax(D'_{12})$. Moreover, for any $(a',b')\in\rela^E_{12}$ there is $(a'',b'')\in\amax(\rela^E_{12})\cap\amax(D'_{12})$ and such that $(a',b')\sqq^{as}_{D'_{12}}(a'',b'')$. To see this, observe that there is $(a_1,b_1)\in\amax(D'_{12})$ such that $(a',b')\sqq^{as}_{D'_{12}}(a_1,b_1)$. Since $(E,E)$ is as-maximal in $D^\beta_{12}$ and $(E,E)\sqq^{as}_{D^\beta_{12}}(a_1\fac\beta,b_1\fac\beta)$, there is an as-path in $D^\beta_{12}$ from $(a_1\fac\beta,b_1\fac\beta)$ to $(E,E)$. By Lemma~\ref{lem:factor-as}(3) there is $(a'',b'')\in\rela^E_{12}$ such that $(a_1,b_1)\sqq^{as}_{D'_{12}}(a'',b'')$. Clearly, $(a'',b'')$ can be chosen from $\amax(\rela^E_{12})$. Also, as $(a_1,b_1)\in\amax(D'_{12})$, so is $(a'',b'')$.

\smallskip

Hence, starting from any $(a',b',c',d')\in\rela$, where $(a',b')\in\rela^E_{12}$ and $c',d'$ are as identified in the beginning of the proof, we first find $(a'',b'',c^*,d^*)\in\rela$ such that $(a'',b'')\in\amax(\rela^E_{12})$ and $(c^*,d^*)\in C^\beta_{34}$. Since $\{(E,E)\}\tm C^\beta_{34}\sse\rela\fac\beta$ there is an as-path in $\rela\cap(E^2\tm C^\beta_{34})$ from $(a'',b'',c^*,d^*)$ to $(a^\dg,b^\dg,c',d')$ for some $(a^\dg,b^\dg)\in\amax(\rela^E_{12})$. Note that by (iii) $(a^\dg,b^\dg)$ can also be assumed to be in $\amax(\relo'_{12}\cap E^2)$. The tuple $(a^\dg,b^\dg,c',d')$ is as required. 
%% 
%% For the second claim of item (5) suppose that $(c',d')\in\amax(D_{34})$, where $D_{34}$ is the $\lnk_{34}$-block containing $(c',d')$. Recall that $C_{34}$ denotes the as-component of $D_{34}$ containing $(c',d')$. Let $D_{12}$ be the $\lnk_{12}$-block containing $(a^\dg,b^\dg)$. Then similar to (ii), (iii) there is $(a^*,b^*)\in\amax(D_{12})\cap\amax(\relo'_{12}\cap E'^2)$ for some $\beta$-block $E'$ such that $(a^\dg,b^\dg)\sqq^{as}_{D_{12}}(a^*,b^*)$. By Proposition~\ref{pro:linkage-rectangularity}(1) $\{(a^*,b^*)\}\tm C_{34}\sse\rela$ and hence the pairs $(a^*,b^*)$ and $(c',d')$ satisfy the desired conditions.
Part (5) is proved.
\end{proof}  

\begin{lemma}\label{lem:Q-linkage}
Let $c,d\in C_2$, $c\ne d$, and let $a,b\in B'_1$, $a\eqc\beta b$, $a\ne b$, be such that $(a,b)\in T_0=T(\zz_1,\{(a,b)\},\cU(\zz_2,\dl,\ov B))$, $B\cap\umax(B'_1)\ne\eps$, where $B$ is the $\beta$-block containing $a,b$, and such that $(a,c)\in\pr_{12}\relov$. Let $\lnk_{12},\lnk_{34}$ denote the link congruences of $\relo'=\relo'(a,b,c,d)$ viewed as a subdirect product of $\relo'_{12}=\pr_{12}\relo'$ and $\relo'_{34}=\pr_{34}\relo'$. Then
\begin{itemize}
\item[(1)]
$(\zz_1\tm\beta)\red{\relo'_{12}}\sse\lnk_{12}$ and 
$(\zz_2\tm\dl)\red{\relo'_{34}}\sse\lnk_{34}$.
\item[(2)]
Let $B'$ be a $\beta$-block and $E'=B'\cap B'_1$, and assume that $E'\cap\umax(B'_1)\ne\eps$.
Then $(\beta\tm\beta)\red{\umax(E')\tm\umax(E')}\sse\lnk_{12}$, or equivalently, $((\umax(E')\tm\umax(E'))\cap\relo'_{12})^2\sse\lnk_{12}$.
\item[(3)]
for any $c',d'\in C_2$ the pair $(c',d')$ is as-maximal in a $\lnk_{34}$-block $D'_{34}$, and the as-component of $D'_{34}$ containing $(c',d')$ is either $\{c'\}\tm C_2$, or $C_2\tm C_2$.
\item[(4)]
There is $a'\in B'_1$ such that $(a',a',c,d)\in\relo'$. In other words, there is a $\ov B$-preserving polynomial $g$ of $\rel$ such that $g(a)=g(b)$ and $g(c)=c,g(d)=d$.
\end{itemize}
\end{lemma}

\begin{proof}
(1) The relation $\relo'$ contains tuples $(a,b,c,d)$, $(a,b,c',c')$, $(a,a,c',c')$, $(a,a,c,c)$ for some $c'\in B'_2$. Indeed, $(a,b,c,d)\in\relo'$ by definition, 
$(a,a,c,c)\in\relo'$ because $(a,c)\in\pr_{12}\relov$, and $(a,b,c',c'),(a,a,c',c')$
can be chosen to be the images of $(a,b,c,d)$ and $(a,a,c,c)$, respectively,
under a $\ov B$-preserving polynomial $g$ such that $g(a)=a$, 
$g(b)=b$ and $g(\dl)\sse\zz_2$. Such a polynomial exists by the choice of $a,b$ such that $(a,b)\in T_0$ and because $(\zz_1,\beta)$ can be separated from 
$(\zz_2,\dl)$. This implies that $(a,b)\eqc{\lnk_{12}}(a,a)$, $(c,d)\eqc{\lnk_{34}}(c,c)$. The relation $\relo'_{12}$ consists of pairs $(g(a),g(b))$ for $\ov B$-preserving unary polynomials $g$ of $\zA_1$. Since $(a,b,c',c'),(a,a,c',c')\in\relo'$, for any $(a',b')=(g(a),g(b))\in\relo'_{12}$ it holds that 
$$
(a',b',c^*,c^*)=(g(a),g(b),g(c'),g(c')),(a',a',c^*,c^*)=(g(a),g(a),g(c'),g(c'))\in\relo',
$$
witnessing that $(a',b')\eqc{\lnk_{12}}(a',a')$. Thus, $(\zz_1\tm\beta)\red{\relo'_{12}}\sse\lnk_{12}$. 

To prove that $(\zz_2\tm\dl)\red{\relo'_{34}}\sse\lnk_{34}$ we observe that, since for any $(a'',b'',c'',d'')\in\relo'$ there is a $\ov B$-preserving polynomial $g$ with $(a'',b'',c'',d'')=(g(a),g(b),g(c),g(d))$ and $(a,c)\in\pr_{12}\rel'$, we have $(a'',c'')\in\pr_{12}\rel'$. Therefore by definition $(a'',a'',c'',c'')\in\relo'$. As $(a'',b'')\eqc{\lnk_{12}}(a'',a'')$, it also holds $(c'',d'')\eqc{\lnk_{34}}(c'',c'')$. Thus, $(\zz_2\tm\dl)\red{\relo'_{34}}\sse\lnk_{34}$.

\smallskip

(2) By the definition of $T_0$, for any pair $(a',b')\in T_0\cap (E')^2$ there is  
a $\ov B$-preserving polynomial $g$ satisfying $g(a)=a'$, 
$g(b)=b'$, and $g(B'_2)=\{c'\}\sse B'_2$.
Applying $g$ to tuples $(a,a,c,c)$, and $(b,b,d',d')$ for
any $d'$ such that $(b,d')\in\pr_{12}\relov$, we obtain 
$(a',a',c',c'),(b',b',c',c')\in\relo'$
Therefore, $(a',a')\eqc{\lnk_{12}}(b',b')$. 
Since $\rel$ is chained with respect to $\ov\beta,\ov B$, by (Q2) and Lemma~\ref{lem:good-polys}(4) $(a'',a'')\eqc{\lnk_{12}}(b'',b'')$
for any $a'',b''\in\umax(E')$. Together with part (1) this proves the result.

\smallskip

(3) Suppose first that for some $e,e'\in B'_2$, $e\ne e'$, we have $(e,e)\eqc{\lnk_{34}}(e',e')$. 
Then for any pair $(e'',e''')\in T(\zz_2,\{(e,e')\},\cU_{\ov B})$ 
there is a $\ov B$-preserving polynomial $g$ with 
$g(e)=e'', g(e')=e'''$. Applying this polynomial to 
the tuples witnessing that $(e,e)\eqc{\lnk_{34}}(e',e')$ we get 
$(e'',e'')\eqc{\lnk_{34}}(e''',e''')$. Therefore by condition (Q1) all tuples of the 
form $(x,x)$, $x\in\umax(B'_2)$, are $\lnk_{34}$-related. 
%% Since by Lemma~\ref{lem:good-polys}(3) $(c',d')$ is a pair from $T(\zz_2,\{(e,e')\},\cU_{\ov B})$, 
If $(x,y),(z,t)\in\relo'^{\; u}_{34}$, 
where $\relo'^{\; u}_{34}=\relo'_{34}\cap(\umax(B'_2)\tm \umax(B'_2))$, then by part~(1) of Lemma~\ref{lem:Q-linkage} $((x,y),(x,x)),((z,t),(z,z))\in\lnk_{34}$. By what is proved above $((x,x),(z,z))\in\lnk_{34}$, and so $((x,y),(z,t))\in\lnk_{34}$. Thus, $\lnk_{34}\red{\relo'^{\; u}_{34}}=(\dl\tm\dl)\red{\relo'^{\; u}_{34}}=\relo'^{\; u}_{34}\tm\relo'^{\; u}_{34}$. In particular, 
$C_2^2$ is contained in $\relo'_{34}$, and is contained in a $\lnk_{34}$-block $D'_{34}$. All elements of $C_2^2$ are as-maximal in $D'_{34}$. 

Suppose $(e,e)\eqc{\lnk_{34}}(e',e')$ for no $e,e'\in B'_2$, $e\ne e'$. The inclusion 
$(\zz_2\tm\dl)\red{\relo'_{34}}\sse\lnk_{34}$ implies 
that  if $(c_1,d_1)\eqc{\lnk_{34}}(c_2,d_2)$ then 
$(c_1,c_1)\eqc{\lnk_{34}}(c_2,c_2)$. Therefore, by part (1) 
we have $\lnk_{34}\red{\relo'^{\; u}_{34}}=(\zz_2\tm\dl)\red{\relo'^{\; u}_{34}}$. 
In particular, $\{c'\}\tm C_2$ is contained in a $\lnk_{34}$-block. Since $c',d'$ are 
as-maximal, $(c',d')$ is as-maximal in this $\lnk_{34}$-block. Part (3) is proved.

\smallskip

(4) We start with an auxiliary claim.

\smallskip
{\sc Claim.} Let $F$ be a $\beta$-block such that $F\cap\umax(B'_1)\ne\eps$ and $(a',b')\in\amax(\relo'^F_{12})$, $\relo'^F_{12}=\relo'_{12}\cap F^2$. Then $(a',a')\in\amax(\relo'^F_{12})$.
%% and it belongs to the same as-component of $\relo'^E_{12}$ as $(a',b')$.

\smallskip

Let $F'=F\cap B'_1$ and let $\lnk^{Q'}_1,\lnk^{Q'}_2$ be the link congruences of the two copies of $B'_1$ with respect to $\relo'_{12}$. (Note that the link congruences of $\relo'^F_{12}$ are just restrictions of $\lnk^{Q'}_1,\lnk^{Q'}_2$ to $F'$.) Also, let $C_1$ be the as-component of $F'$ containing $a'$; such an as-component exists by the choice of $a'$. We first prove that $C_1$ is a subset of a $\lnk^{Q'}_1$- and a $\lnk^{Q'}_2$-blocks. As by Lemma~\ref{lem:polys-max}(2) $\relo'_{12}\sse\relo^{\zA_1}_{ab}$, we have $\lnk^{Q'}_1,\lnk^{Q'}_2\le\beta\red{B'_1}$. On the other hand, $\relo'_{12}$ consists of pairs of the form $(x,x)$ and pairs from $T_1=T(\zz_1,\{(a,b)\},\cU_{\ov B})$. As $\rel$ is chained, by Lemma~\ref{lem:good-polys}(3), for any $e,e'\in C_1$ we have $\{e,e'\}\in T_1$. This means that either $(e,e')\in\relo'_{12}$ or $(e',e)\in\relo'_{12}$. Since $(e,e),(e',e')\in\relo'_{12}$ by construction, in either case we have $(e,e')\in\lnk^{Q'}_1,\lnk^{Q'}_2$ witnessing that $C_1^2\sse\lnk^{Q'}_1,\lnk^{Q'}_2$.
Now, as $(a',a')\in\relo'_{12}\cap(C_1\tm C_1)\ne\eps$ and $C_1\sse\umax(F')$, by Proposition~\ref{pro:linkage-rectangularity}(1)  
$C_1\tm C_1\sse\relo'_{12}$. Since $C_1$ is an as-component in $F'$, by
Lemma~\ref{lem:as-product} $C_1\tm C_1$ is an as-component in $\relo'^F_{12}$. In particular $(a',a')$ is as-maximal in $\relo'^F_{12}$. The Claim is proved.

\smallskip

Next, let $(c',d')\in C_2^2$ (recall that $C^2_2\sse\relo'_{34}$ by Lemma~\ref{lem:polys-max}(5)) and $D'_{34}$ denote the $\lnk_{34}$-block $(c',d')$ belongs to. By Lemma~\ref{lem:Q-linkage}(3) $(c',d')$ is as-maximal in $D'_{34}$, and the as-component $C'_{34}$ it belongs to is $\{c'\}\tm C_2$ or $C^2_2$. Let $D'_{12}$ be the $\lnk_{12}$ block such that $\relo'\cap(D'_{12}\tm D'_{34})\ne\eps$. By the Maximality Lemma~\ref{lem:to-max}(5) there is $(a',b')\in\amax(D'_{12})$ such that $(a',b',c',d')\in\relo'$. By Proposition~\ref{pro:linkage-rectangularity}(1) $\{(a',b')\}\tm \{c'\}\tm C_2\sse\relo'$. This implies that for any $(a'',b'')\in\amax(D'_{12})$ such that $(a'',b'',c'',d'')\in\relo'$ for some $c'',d''\in C_2$, $\{(a'',b'')\}\tm \{c''\}\tm C_2\sse\relo'$.
%% As $(c,d)\in\amax(\pr_{34}\relo)$, there is $(a',b')\in\amax(\relo'_{12})\sse\umax(\relo')$ such that $(a',b',c',d')\in\relo'$. 
%% By Lemma~\ref{lem:polys-max}(5) there is $(a',b',c',d')\in\relo'$ such that $(a',b')\in\amax(\relo'^E_{12})$, $\relo'^E_{12}=\relo'_{12}\cap E^2$, for some $\beta$-block $E$ and $c',d'\in C_2$, $c'\ne d'$. 
%% Let $D'_{34}$ denote the $\lnk_{34}$-block $(c,d)$ belongs to. By Lemma~\ref{lem:Q-linkage}(3) $(c,d)$ is as-maximal in $D'_{34}$, and the as-component $C'_{34}$ it belongs to is $\{c\}\tm C_2$ or $C^2_2$. Let $D'_{12}$ be the $\lnk_{12}$ block such that $\relo'\cap(D'_{12}\tm D'_{34})\ne\eps$. By the Maximality Lemma~\ref{lem:to-max}(5) there is $(a'',b'')\in\amax(D'_{12})$ such that $(a'',b'',c,d)\in\relo'$ and such that $(a',b')\sqq^{as}_{D'_{12}}(a'',b'')$. Let $(a^*,b^*)\in\amax(\relo'^{E'}_{12})$, $\relo'^{E'}_{12}=\relo'_{12}\cap E'^2$, where $E'$ is the $\beta$-block containing $a'',b''$, such that $(a'',b'')\sqq^{as}_{\relo'^{E'}_{12}}(a^*,b^*)$. There is $(c^*,d^*)\in C'_{34}$ with $(a^*,b^*,c^*,d^*)\in\relo'$. By Proposition~\ref{pro:linkage-rectangularity}(1) $\{(a^*,b^*)\}\tm \{c\}\tm C_2\sse\relo'$. In particular, $(a^*,b^*,c,d)\in\relo'$

Now, let $D_{34}$ be the $\lnk_{34}$-block containing $(c,d)$ (where $c,d$ are the elements from the statement of the lemma), $D_{12}$ the $\lnk_{12}$-block with $\relo'\cap(D_{12}\tm D_{34})\ne\eps$, and $(a',b')\in\amax(D_{12})$ such that $(a',b',c,d)\in\relo'$ and $(a,b)\sqq^{as}_{D_{12}}(a',b')$. Let also $B'$ be the $\beta$-block containing $a',b'$. Note that $a'\eqc\beta b'$, as $a\eqc\beta b$. Also, $B'\cap\umax(B'_1)\ne\eps$. Indeed, we know that $B\cap\umax(B'_1)\ne\eps$, where $B$ is the $\beta$-block containing $a,b$, say $a^*\in B\cap\umax(B'_1)$. By Lemma~\ref{lem:factor-as}(4) there is an as-path from $B$ to $B'$ in $B'_1\fac\beta$, which by Lemma~\ref{lem:factor-as}(3) implies that $a^*\sqq^{as}_{B'_1}a^\dg$ for some $a^\dg\in B'$, and so $a^\dg\in B'\cap\umax(B'_1)$.

We consider two cases.

\smallskip

{\sc Case 1.} The as-component of $D_{34}$ is $C^2_2$.

\smallskip

Let $\relo'^{B'}_{12}=\relo'_{12}\cap(B'\tm B')$ and take $(a'',b'')\in\amax(\relo'^{B'}_{12})$ such that $(a',b')\sqq^{as}_{\relo'^{B'}_{12}}(a'',b'')$. By the Maximality Lemma~\ref{lem:to-max}(3) there are $c'',d''\in C_2$ with $(a'',b'',c'',d'')\in\relo'$ and an as-path from $(a',b',c,d)$ to $(a'',b'',c'',d'')$ in $\relo'^{B'}_{12}\tm C^2_2$. Therefore, as  $(a',b')\in\amax(D_{12})$, we have $(a'',b'')\in\amax(D_{12})$. By Proposition~\ref{pro:linkage-rectangularity}(1) $\{(a'',b'')\}\tm C^2_2\sse\relo'$. Since $(a,c)\in\pr_{12}\rel'$, $\pr_{13}\relo'\sse\pr_{12}\rel'$, hence, by the definition of $\relo'$ we have $(a'',a'',c'',c'')\in\relo'$ and by Claim $(a'',a'')\in\amax(\relo'^{B'}_{12})$. Now we apply Proposition~\ref{pro:linkage-rectangularity}(2) to the relation $\relo'\cap(\relo'^{B'}_{12}\tm(B'_2)^2)$; we get $\{(a'',a'')\}\tm C^2_2\sse\relo'$. The tuple $(a'',a'',c,d)$ is as required.

\smallskip

{\sc Case 2.} The as-component of $D_{34}$ is $\{c\}\tm C_2$.

\smallskip

In this case by Proposition~\ref{pro:linkage-rectangularity}(1) $\{(a',b')\}\tm \{c\}\tm C_2\sse\relo'$. Similar to Case~1 let $\relo'^{B'}_{12}=\relo'_{12}\cap(B'\tm B')$ and take $(a'',b'')\in\amax(\relo'^{B'}_{12})$ such that $(a',b')\sqq^{as}_{\relo'^{B'}_{12}}(a'',b'')$. However, there is no guarantee that $(a'',b'')\in D_{12}$ and that it is as-maximal in its $\lnk_{12}$-block. 
%% By Claim $(a'',a'')$ is also as-maximal in $\relo'^{B'}_{12}$, so we may assume $a''=b''$. 
The following argument is a special case of Lemma~3.1 from \cite{Bulatov24:maximal}, we give it for the sake of self-containment.

Let $(a',b')=(a_1,b_1)\zd(a_k,b_k)=(a'',b'')$ be an as-path in $\relo'^{B'}_{12}$. By the Maximality Lemma~\ref{lem:to-max}(3) it can be extended to an as-path $(a_1,b_1,c_1,d_1)\zd(a_k,b_k,c_k,d_k)$ in $\relo'$ such that $c_i,d_i\in C_2$, $i\in[k]$. We prove that $\{(a_i,b_i)\}\tm\{c_i\}\tm C_2\sse\relo'$ for each $i\in[k]$. Since $(a_1,b_1)\in\amax(D_{12})$, it is true for $i=1$. Suppose it is true for $i$ and let $D=\{d^*\in C_2\mid (a_{i+1},b_{i+1},c_{i+1},d^*)\in\relo'\}$. If $D\ne C_2$, there are $e\in D$ and $e'\in C_2-D$ such that $ee'$ is a semilattice or affine thin edge. The pair $(a_i,b_i,c_i)(a_{i+1},b_{i+1},c_{i+1})$ is also a semilattice or affine thin edge. If at least one of them is not affine, then 
\[
(a_{i+1},b_{i+1},c_{i+1},e')\in\{p((a_i,b_i,c_i,e'),(a_{i+1},b_{i+1},c_{i+1},e)),p((a_{i+1},b_{i+1},c_{i+1},e),(a_i,b_i,c_i,e'))\},
\]
where $p$ is the operation from Lemma~\ref{lem:op-s-on-affine}(3), a contradiction with $e'\not\in D$. If both edges are affine, then
\[
(a_{i+1},b_{i+1},c_{i+1},e')=h'((a_i,b_i,c_i,e'),(a_i,b_i,c_i,e),(a_{i+1},b_{i+1},c_{i+1},e)),
\]
where $h'$ is the operation from Lemma~\ref{lem:op-s-on-affine}(3), again a contradiction with $e'\not\in D$.

Thus, $(a'',b'',c_k,d^*)\in\relo'$ for some $d^*\ne c_k$, $d^*\in C_2$. Let $E'=B'\cap B'_1$ and note that $\pr_i(\relo'^{B'}_{12})=E'$ for $i=1,2$. Thus, $\amax(\relo'^{B'}_{12})\sse(\amax(E'))^2$ by Maximality Lemma~\ref{lem:to-max}(6). By item~(2) $(\relo'^{B'}_{12}\cap(\umax(E')^2))^2\sse\lnk_{12}$, and since $(a'',b'')\in\amax(\relo'^{B'})\cap D'_{12}$, we get $\relo'^{B'}_{12}\cap(\umax(E')^2)\sse D'_{12}$. Consider the algebra $\rela=\relo'\cap((B')^2\tm D'_{34})$, recall that $D'_{34}$ is the $\lnk_{34}$-block containing $(c',d')$, and therefore $(c_k,d^*)$. By the previous argument $\amax(\relo'^{B'}_{12})\sse\pr_{12}\rela$, so $\amax(\relo'^{B'}_{12})\sse\amax(\pr_{12}\rela)$ by Corollary~\ref{cor:subalgebra-path}. Also, as $a'',b''\in\amax(E')$, $\{c_k\}\tm C_2\sse\pr_{34}\rela$. By Claim $(a'',a'')$ is also as-maximal in $\relo'^{B'}_{12}$ and, as in Case~1, $(a'',a'',c_k,c_k)\in\relo'$. Therefore, by Proposition~\ref{pro:linkage-rectangularity}(2) we also have $(a'',a'',c_k,d^*)\in\relo'$. Let $f$ be a $\ov B$-preserving polynomial witnessing that\lb $(a'',a'',c_k,d^*)\in\relo'$.

Finally, as $\rel$ is chained, by Lemma~\ref{lem:good-polys}(3) there is a $\ov B$-preserving polynomial $f'$ such that $f'(c_k)=c, f'(d^*)=d$. As is easily seen, $g=f'\circ f$ satisfies the desired conditions.
\end{proof}

\begin{corollary}\label{cor:claim7}
Let $c,d\in C_2$, $c\ne d$, and a $\beta$-block $E$, $E\cap\umax(B'_1)\ne\eps$, be such that $(a,c)\in\pr_{12}\relov$ for some $a\in\umax(E')$, $E'=E\cap B'_1$. Then there is a $\ov B$-preserving polynomial $h$ such that $|h(\umax(E'))|=1$ and $h(c)=c, h(d)=d$. 
\end{corollary}

\begin{proof}
If $|\umax(E')|=1$ then $h$ can be chosen to be the identity mapping, so we assume $|\umax(E')|>1$. Since $a\in\umax(E')$, by Lemma~\ref{lem:good-polys}(6) $\{a,b\}\in T_0=T(\zz_1,\{(a,b)\},\cU(\zz_2,\dl,\ov B))$ for some $b\in E'$, $a\ne b$, and $(a',b')\in T(\zz_1,\{(a',b')\},\cU(\zz_2,\dl,\ov B))$ for all $(a',b')\in T_0$, $a'\ne b'$. Let us consider $T_0$ as a directed graph with all the loops present. Observe that if $f\in \cU(\zz_2,\dl,\ov B)$ and $g$ is $\ov B$-preserving, then $g\circ f\in\cU(\zz_2,\dl,\ov B)$. Therefore, by Lemma~\ref{lem:good-polys}(1) any $\ov B$-preserving polynomial $g$ is an endomorphism of $T_0$. Moreover, let $T_{0c}$ denote the subgraph of $T_0$ induced by $(\pr_{12}\rel')^{-1}[c]$, that is, $a'$ is a vertex of $T_{0c}$ if and only if $(a',c)\in\pr_{12}\rel'$.  Note that any $\ov B$-preserving polynomial $g$ such that $g(c)=c$ maps $T_{0c}$ to itself. Since $\rel$ is chained, the set $\umax(E')$ is contained in a connected component (not necessarily strongly connected) of $T_0$.

By Lemma~\ref{lem:Q-linkage}(4) for any $(a',b')\in T_0$, $a'\in V(T_{0c})$, there is a polynomial $h^{a'b'}_{cd}$ such that $h^{a'b'}_{cd}(a')=h^{a'b'}_{cd}(b')$ and $h^{a'b'}_{cd}(c)=c, h^{a'b'}_{cd}(d)=d$. If conversely $(b',a')\in T_0$ with $a'\in V(T_{0c})$, then $(b',a')\in T(\zz_1,\{(b',a')\},\cU)$ by the choice of $b$, so $(a',b')\in T(\zz_1,\{(a',b')\},\cU)$, and again by Lemma~\ref{lem:Q-linkage}(4) we get the desired $h^{a'b'}_{cd}$. We construct a sequence of induced subgraphs $T_0^i$ of $T_0$ as follows. Let $T_0^1=T_0$. Then if $T_0^i$ contains an edge $(a',b')$ or $(b',a')$ such that $a'\ne b'$ and $a'\in V(T_{0c})$, set $T^{i+1}_0=h^{a'b'}_{cd}(T^i_0)$. Observe that a connected component of $T^i_0$ is mapped to a connected component of $T^{i+1}_0$. This means that when the sequence stops at, say, $T^k_0$, every connected component of $T^k_0$ containing a vertex from $T_{0c}$ is a singleton. Therefore, if $g_i$ is the polynomial that is used to map $T^i_0$ to $T^{i+1}_0$, then $h=g_{k-1}\circ\dots\circ g_1$ is such that $h(c)=c, h(d)=d$, and $|h(\umax(E'))|=1$.
The polynomial $h$ is as required.
\end{proof}

Next, we show the existence of polynomials of $\rel$ that satisfy certain conditions on $B'_2$.

\begin{lemma}\label{lem:as-moving-polys}
Let $c,d\in C_2$, $c\ne d$, and $c'\in B'_2$, $c'\not\in\{c,d\}$. Either there is a $\ov B$-preserving polynomial $f$ of $\rel$ such that $f(c)=f(c')=c$ and $f(d)=d$, or for any $\ov B$-preserving polynomial $f$ such that $f(c')=c, f(d)=d$, it holds that $f(c)\ne d$.
\end{lemma}

\begin{proof}
We consider $\relp(c',c,d)$, and $\relp'(c',c,d)$ introduced before Lemma~\ref{lem:polys-max}. Let $\relp'=\relp'(c',c,d)$. As $\rel$ is chained, by Lemma~\ref{lem:polys-max}(4) and Lemma~\ref{lem:good-polys}(3) $C_2^2\sse\pr_{12}\relp',\pr_{13}\relp',\pr_{23}\relp'$. Moreover, as $(c,c,c)\in\relp'$, by the Maximality Lemma~\ref{lem:to-max}(3) $\relp'\cap C^3_2$ is subdirect. Suppose that there is a $\ov B$-preserving polynomial $f$ of $\rel$ such that $f(c')=c$ and $f(c)=f(d)=d$. This means that $(c,d,d)\in\relp'$. Since $(d,d,d)\in\relp'$, and since $\rel$ is chained, for any $x,y\in C_2$ it holds that $(x,y,y),(y,y,y)\in\relp'$. Indeed, the chaining condition and Lemma~\ref{lem:good-polys}(3) imply that there is a $\ov B$-preserving polynomial $h$ such that $h(c)=x,h(d)=y$ or $h(c)=y,h(d)=y$. Then $(x,y,y),(y,y,y)$ are the images of $(c,d,d)$ and $(d,d,d)$ under $h$. Therefore $C_2$ is a subset of a class of the link congruence of $\pr_1\relp'$, when $\relp'$ is considered a subdirect product of $\pr_1\relp'$ and $\pr_{23}\relp'$. Since $C_2$ and $C_2^2$ are as-components in $\pr_1\relp'$ and $\pr_{23}\relp'$, respectively, by Proposition~\ref{pro:linkage-rectangularity}(1) $C_2^3\sse\relp'$. In particular, $(c,c,d)\in\relp'$ witnessing that there is a $\ov B$-preserving polynomial $f$ of $\rel$ such that $f(c)=f(c')=c$ and $f(d)=d$.
\end{proof}

\begin{proof}[Proof of Theorem~\ref{the:relative-symmetry}]
We need to show that there is a $\ov B$-preserving polynomial $g$ such that $g$ collapses $\beta\red{B'_1}$ but does not collapse $\beta_2=\dl$. 

First we show that for any $c,d\in C_2$ and any 
$(a,b)\in\beta\red{B'_1}$ there is a $\ov B$-preserving polynomial $h^{ab}_{cd}$ 
of $\rel$ such that 
\begin{itemize}
\item[(1)]
$h^{ab}_{cd}$ is idempotent;
\item[(2)]
$h^{ab}_{cd}(a)=h^{ab}_{cd}(b)$;
\item[(3)]
$h^{ab}_{cd}(c)=c$, $h^{ab}_{cd}(d)=d$.
\end{itemize}
Fix $c,d\in C_2$ with $c\ne d$. By Corollary~\ref{cor:claim7} such a $\ov B$-preserving polynomial exists whenever $a,b\in\umax(E\cap B'_1)$  for a $\beta$-block $E$, $E\cap\umax(B'_1)\ne\eps$, and $(a',c)\in\pr_{12}\relov$ for some $a'\in\umax(E\cap B'_1)$.  We need to prove this for arbitrary $a,b\in B'_1$ with $a\eqc\beta b$.

Take any $c'\in B'_2$ such that $(a,c')\in\pr_{12}\rel'$. If $c'=d$ we swap the roles of $c$ and $d$. If there exists $h^{ab}_{dc}$ with the required properties, this polynomial can also be chosen as $h^{ab}_{cd}$. Otherwise, i.e.\ if $c'\ne d$, by Lemma~\ref{lem:good-polys}(3) there is a $\ov B$-preserving polynomial $g$ such that $g(c')=c,g(d)=d$. 
%% If, see Lemma~\ref{lem:as-moving-polys}, there is a polynomial that maps $c',c$ to $c$ and $d$ to $d$, we set $g$ to be that polynomial. 
Let $\relo'=\relo'(g(a),g(b),c,d)$ be as defined before Lemma~\ref{lem:polys-max}. We use the notation from Lemma~\ref{lem:polys-max}. 
By Lemma~\ref{lem:polys-max}(1) it suffices to prove that  
a tuple of the form $(a^*,a^*,c'',d'')$, $c'',d''\in C_2$, $c''\ne d''$ belongs to $\relo'$. Indeed, if this is the case, there is a $\ov B$-preserving polynomial $h$ of $\rel$ such that $h(g(a))=h(g(b))$ and $h(c)=c'',h(d)=d''$. Since $\rel$ is chained, there is also a $\ov B$-preserving polynomial $h'$ such that $h'(c'')=c, h'(d'')=d$. Set $g'=h'\circ h\circ g$. If $g$ is such that $g(c)=c$, then  $g'(a)=g'(b)$ and $g'(c)=c, g'(d)=d$, and we can set $h^{ab}_{cd}$ to be the idempotent power of $g'$. Otherwise by Lemma~\ref{lem:as-moving-polys} $g'(c)\ne d$ and there is a $\ov B$-preserving polynomial $h''$ such that $h''(g'(c))=c$ and $h''(d)=d$. As is easily seen, the idempotent power of the polynomial $h^{ab}_{cd}=h''\circ g'$ is as required.

Next, we show that a tuple of the form $(a^*,a^*,c'',d'')$ exists in $\relo'$. By Lemma~\ref{lem:polys-max}(5) $\relo'$ contains a tuple $(a'',b'',c'',d'')$ such that $a'',b''$ are as-maximal in $E\cap B'_1$ for a $\beta$-block $E$, $E\cap\umax(B'_1)\ne\eps$, and $c'',d''\in C_2$, $c''\ne d''$. If $a''=b''$, we are done. Otherwise, as by Corollary~\ref{cor:amax-umax} every as-maximal element is also u-maximal and $(a'',c'')\in\pr_{12}\relov$, the elements $a'',b'',c'',d''$ satisfy all the conditions of Corollary~\ref{cor:claim7}.
Therefore there is a $\ov B$-preserving polynomial $h^{a''b''}_{c''d''}$ such that $h^{a''b''}_{c''d''}(a'')=h^{a''b''}_{c''d''}(b'')$ and $h^{a''b''}_{c''d''}(c'')=c'', h^{a''b''}_{c''d''}(d'')=d''$. Again, a tuple of the desired form belongs to $\relo'$.

Finally, we use polynomials $h^{ab}_{cd}$ to construct a single $\ov B$-preserving polynomial
that collapses $\beta$ on $E'=E\cap B'_1$ for every $\beta$-block $E$. 
Fix $c,d$ and $h^{ab}_{cd}$ for every pair $a,b\in B'_1$, $a\eqc\beta b$. 
Let $\vc Vk$ be the list of all such pairs, and if $V_\ell=\{a,b\}$ is the pair 
number $\ell$, $h^\ell$ denotes $h^{ab}_{cd}$. Take a sequence 
$1=\ell_1,\ell_2,\ldots$ such 
that $h^{(1)}=h^{\ell_1}$, $V_{\ell_2}$ is a subset of $h^{(1)}(\zA_1)$, 
and, for $s>2$, $V_{\ell_s}$ is a subset of 
the range of $h^{(s-1)}=h^{\ell_{s-1}}\circ\ldots\circ h^{\ell_1}$. Since 
$|\Im(h^{(s)})|<|\Im(h^{(s-1)})|$, there is $r$ such that 
$\Im(h^{(r)})$ contains no pair $V_\ell$ for any $\ell$. Therefore setting 
$h(x)=h_{\ell_r}\circ\ldots\circ h_{\ell_1}(x)$ we have that $h$ collapses 
all the pairs $V_\ell$, and $h$ acts identically on $\{c,d\}$. 
The result follows.
\end{proof}

%%%%%%%%%%%%%%%%%%%%%%%%%%%%%%%%%%%%%%
\subsection{Separation and minimal sets}\label{sec:sep-min}

The results of this section hold for arbitrary algebras, not only for algebras from $\cV$.

We say that prime intervals $(\al,\beta)$ and $(\gm,\dl)$ \emph{cannot be 
separated} if $(\al,\beta)$ cannot be separated from 
$(\gm,\dl)$ and $(\gm,\dl)$ cannot be separated from $(\al,\beta)$. 
In this section we show a connection between the fact that two prime 
intervals cannot be separated, their types, and link congruences.

\begin{lemma}\label{lem:e-related}
Let $\zA$ be an algebra. 
\begin{itemize}
\item[(1)] 
If prime intervals $(\al,\beta)$ and $(\gm,\dl)$ 
in $\Con(\zA)$ are perspective, they cannot be separated. 
\item[(2)] 
If $\al\prec\beta$ and $\gm\prec\dl$ from $\Con(\zA)$ cannot be separated, 
then a set $U$ is an $(\al,\beta)$-minimal set if and only if it is a 
$(\gm,\dl)$-minimal set.
%
%% \item[(3)] 
%% Let $\rel$ be a subdirect product of $\zA$ and $\zB$, $\al,\beta\in\Con(\zA)$,
%% $\gm,\dl\in\Con(\zB)$ such that $\al\prec\beta$, $\gm\prec\dl$, and let $\al\prec\beta$ 
%% and $\gm\prec\dl$ cannot be separated. Then for any $(\al,\beta)$-minimal 
%% set $U$ there is a unary idempotent polynomial $f$ such that $f(\zA)=U$ and 
%% $f(\zB)$ is a $(\gm,\dl)$-minimal set.
\end{itemize}
\end{lemma}

\begin{proof}
(1) Suppose $f$ separates $(\al,\beta)$ from $(\gm,\dl)$. By Lemma~\ref{lem:min-set-separation}(1) $f$ can be assumed to be idempotent and $U=f(\zA)$ is an $(\al,\beta)$-minimal set. By Lemma~\ref{lem:perspective-intervals} $U$ is also a $(\gm,\dl)$-minimal set implying $f(\dl)\not\sse\gm$, a contradiction with the choice of $f$.

(2) Let $f$ be a  polynomial of $\zA$ 
such that $f(\zA)=U$ and $f(\beta)\not\sse\al$. Since $(\al,\beta)$ cannot be 
separated from $(\gm,\dl)$, we have $f(\dl)\not\sse\gm$ and therefore by Lemma~\ref{lem:minimal-sets}(6) $U$ 
contains a $(\gm,\dl)$-minimal set $U'$. If $U'\ne U$, there is a polynomial $g$
with $g\circ f(\dl)\not\sse\gm$ and $g\circ f(\zA)=U'$. In particular, 
$|g(U)|<|U|$, and so $g\circ f(\beta)\sse\al$; a contradiction with the assumption
that $(\gm,\dl)$ cannot be separated from $(\al,\beta)$.
%% 
%% (3) Take an idempotent polynomial $g$ of $\rel$ such that $g(\zB)$ is a 
%% $(\gm,\dl)$-minimal 
%% set. Then, as $(\gm,\dl)$ cannot be separated from $(\al,\beta)$, 
%% $g(\beta)\not\sse\al$. By Lemma~\ref{lem:minimal-sets}(6) there is an 
%% $(\al,\beta)$-minimal set $U'\sse g(\zA)$. Let $g',h$ be polynomials 
%% of $\rel$ such that $g'(U)=U'$, $h(U')=U$ and $h(\zA)=U$, which exist by
%% Lemma~\ref{lem:minimal-sets}(1). Then $h'=h\circ g\circ g'$ is such that 
%% $h'(\zA)=h'(U)=U$, $h'(\beta)\not\sse\al$ and therefore $h'(\dl)\not\sse\gm$.
%% Then iterating $h'$ sufficiently many times we get an idempotent polynomial $f$
%% satisfying the same properties.
\end{proof}

\begin{lemma}\label{lem:type-equal}
Let $\rel$ be a subdirect product of $\zA$ and $\zB$ and let 
$\al,\beta\in\Con(\zA)$, $\gm,\dl\in\Con(\zB)$ be such that $\al\prec\beta$, 
$\gm\prec\dl$, and $(\al,\beta),(\gm,\dl)$ cannot be separated. Then 
$\typ(\al,\beta)=\typ(\gm,\dl)$.
\end{lemma}

\begin{proof}
We use Theorem~5.2(1) from \cite{Kiss97:easy} that states that for a finite algebra $\zC$ and $\al,\beta,\gm,\dl\in\Con(\zC)$ such that $\al\prec\beta,\gm\prec\dl$ and $\zC$ is $(\al,\beta)$- and $(\gm,\dl)$-minimal it holds that $\typ(\al,\beta)=\typ(\gm,\dl)$. Let $\al_1,\beta_1,\gm_2,\dl_2\in\Con(\rel)$ be congruences of $\rel$ given by $\al_1=\al\tm\zo_\zB, \beta_1=\beta\tm\zo_\zB, \gm_2=\zo_\zA\tm\gm, \dl_2=\zo_\zA\tm\dl$. Let also $\eta_1,\eta_2$ be the projection congruences of $\rel$, that is, $\eta_1=\zz_\zA\tm\zo_\zB, \eta_2=\zo_\zA\tm\zz_\zB$. Since $\zA \cong R\fac{\eta_1}$, we have $\typ_\zA(\al,\beta) = \typ_{R\fac\eta_1}(\al_1\fac{\eta_1},\beta_1\fac{\eta_1})$, and by Corollary~5.3 of \cite{Hobby88:structure} $\typ_\zA(\al,\beta)=\typ_R(\al_1,\beta_1)$. In a similar way $\typ_\zB(\gm,\dl)=\typ_R(\gm_2,\dl_2)$. Now, let $C$ be an $(\al_1,\beta_1)$-minimal set. By Lemma~\ref{lem:e-related}(2) it is also a $(\gm_2,\dl_2)$-minimal set. Let $\zC=\zR\red C$ be the algebra induced on $C$ by the polynomials of $R$. Then $\al_1\red C\prec \beta_1\red C$ and $\zC$ is $(\al_1\red C,\beta_1\red C)$-minimal by \cite[Lemma 2.13]{Hobby88:structure} and $\typ_\zR(\al_1,\beta_1)=\typ_\zC(\al_1\red C,\beta_1\red C)$ by definition. Similarly, $\zC$ is $(\gm_2\red C,\dl_2\red C)$-minimal and $\typ_\zR(\gm_2,\dl_2)=\typ_\zC(\gm_2\red C,\dl_2\red C)$. Then Theorem~5.2(1) from \cite{Kiss97:easy} applies to $\zC$ and gives $\typ_\zC(\al_1\red C,\beta_1\red C)=\typ_\zC(\gm_2\red C,\dl_2\red C)$.
\end{proof}

%%%%%%%%%%%%%%%%%%%%%%%%%%%%%%%%%%
%%%%%%%%%%%%%%%%%%%%%%%%%%%%%%%%%%
\section{Centralizers}\label{sec:centralizer}

%%%%%%%%%%%%%%%%%%%%%%%%%%%%%%%%%%
\subsection{Centralizers and quasi-centralizers}\label{sec:definition-centralizer}

The results of Section~\ref{sec:centralizer} hold for arbitrary algebras. If $\zA$ is an algebra, $g\in\Pol_{n+1}(\zA)$, $n\in\nat$ and $a\in\zA$, then $g^a$ denotes 
the polynomial $g^a(\vc xn) = g(a,\vc xn)$. Recall that for 
$\al,\beta\in\Con(\zA)$, the \emph{centralizer} (see, e.g.\  
\cite{Freese87:commutator})
$(\al:\beta)$ is the largest congruence $\th\in\Con(A)$
which centralizes $\beta$ modulo $\al$, i.e., satisfies the condition 
$C(\th,\beta;\al)$ given by
\begin{quote}
for any $n$, any $f\in\Pol_{n+1}(\zA)$, any $(a,b)\in\th$ and any $(c_i,d_i)\in\beta$, $i\in[n]$,
it holds that $f^a(\vc cn)\eqc\al f^a(\vc dn)$ if and only if $f^b(\vc cn)\eqc\al f^b(\vc dn)$.
\end{quote}
In \cite{Bulatov17:dichotomy-corr} we also introduced a somewhat 
related notion of \emph{quasi-centralizer} $\zeta(\al,\beta)$:
\begin{quote}
$(a,b)\in\zeta(\al,\beta)$ if for any $g\in\Pol_2(\zA)$, $g^a(\beta)\sse\al$ 
if and only if $g^b(\beta)\sse\al$.
\end{quote}
A relation basically identical to quasi-centralizer also appeared in 
\cite{Hobby88:structure}, but completely inconsequentially, they did not 
study it at all. Kearnes observed that $(\al:\beta)\sse\zeta(\al,\beta)$,
and later Willard \cite{Willard19:centralizer} demonstrated that the 
notions of the centralizer and the quasi-centralizer are equivalent at least 
in the case important for the purpose of this paper.
With Willard's permission we reproduce his proof here 
and will use the usual centralizer in the sequel.

\begin{prop}\label{pro:two-centralizers}
Let $\zA$ be a finite algebra, $\al,\beta\in\Con(\zA)$, $\al\prec\beta$. 
If $\typ(\al,\beta)\ne\one$, then $\zeta(\al,\beta)=(\al:\beta)$.
\end{prop}

\begin{proof}
Let $\th=(\al:\beta)$. We first show $\th\sse\zeta(\al,\beta)$. Suppose 
$(a,b)\in\th$, $g\in\Pol_2(\zA)$, 
and $g^a(\beta)\sse\al$. Pick $(c,d)\in\beta$. Then
$g(a,c)=g^a(c)\eqc\al g^a(d)=g(a,d)$. As $C(\th,\beta;\al)$, $(a,b)\in\th$ 
and $(c,d)\in\beta$, we can replace the $a$’s with $b$’s to
get $g(b,c)\eqc\al g(b,d)$, i.e., $(g^b(c),g^b(d))\in\al$, proving 
$g^b(\beta)\sse\al$. The reverse implication is proved similarly, so 
$(a,b)\in\zeta(\al,\beta)$. This proves $\th\sse\zeta(\al,\beta)$.

Next we prove that $\zeta(\al,\beta)\sse\th$. Fix an $(\al,\beta)$-minimal set 
$U$ and a unary idempotent polynomial $e\in\Pol_1(\zA)$ satisfying  
$e(\zA)=U$. Now consider cases.

\medskip

\noindent
{\sc Case 1.} $\typ(\al,\beta)\ne\two$.

\smallskip

We will show $\zeta(\al,\beta)\cap\beta=\al$, as 
$C(\zeta(\al,\beta),\beta;\al)$ will then follow 
(\cite[Proposition 3.4(4)]{Hobby88:structure}), which will then imply 
$\zeta(\al,\beta)\le\th$. 
Clearly $\al\sse\zeta(\al,\beta)\cap\beta$, so it remains to prove the opposite 
inclusion. Since $\al\prec\beta$, it suffices to show that $\beta\not\sse\zeta(\al,\beta)$. Let $N$ be the unique $(\al,\beta)$-trace in $U$. Write 
$N=\{1\}\cup O$, where $N^2\sse\beta$ and $\{1\},O$ are the two 
$\al\red U$-classes contained in $N$, and for which there exists a 
pseudo-meet operation $p$, see Lemma~\ref{lem:pseudo-meet}, of 
$\zA\red U$ for $N$ with neutral element 
$1$. Thus $p$ is a binary polynomial of $\zA\red U$ satisfying 
(among other things) $p(1,x)=x$ and $p(o,x)\ne 1$ for all $x\in U$ 
and $o\in O$. Fix $o\in O$. Then, as $p(o,x)\ne 1$, it holds that $p^o(\beta)\sse\al$, while $p^1(x)=x$ for all $x\in N$, proving $(o,1)\in\beta-\zeta(\al,\beta)$.
%% 
%% Suppose $(a,b)\in(\zeta(\al,\beta)\cap\beta)-\al$. 
%% By Lemma~\ref{lem:minimal-sets}(4) 
%% there exists $f\in\Pol_1(\zA)$ such that $(f(a),f(b))\in\beta\red U -\al\red U$; 
%% thus without loss of generality we can assume
%% $a,b\in U$. Because $N$ is the unique $(\al,\beta)$-trace of $U$, one of $a,b$ 
%% must equal 1 and the other must be in $O$. Assume for concreteness 
%% that $b=1$ and $a\in O$. Define $g(x,y)=p(e(x),e(y))\in\Pol_2(\zA)$. 
%% Then $g^1(1) = p(1,1)=1$ while $g^1(a)=p(1,a)=a$, so 
%% $g^1(\beta\red U)\not\sse\al$ which certainly implies $g^1(\beta)\not\sse\al$. 
%% Suppose there exists $(c,d)\in\beta$ such that $(g^a(c),g^a(d))\not\in\al$. 
%% Let $c'=e(c)$ and $d'=e(d)$, so $(c',d')\in\beta\red U$ and 
%% $(p(a,c'), p(a,d'))\in\beta\red U-\al\red U$. Since $N$ is the 
%% unique $\beta\red U$-class containing more than one $\al\red U$-class, 
%% one of $p(a,c')$ or $p(a,d')$ must equal 1, which is impossible as 
%% $a\in O$. Thus such $(c,d)$ does not exist, which proves 
%% $g^a(\beta)\sse\al$. This contradicts the assumption that 
%% $(a,1) = (a,b)\in\zeta(\al,\beta)$.

\medskip

\noindent
{\sc Case 2.} $\typ(\al,\beta)=\two$.

\smallskip

Suppose $(a,b)\in\zeta(\al,\beta)$. To prove $(a,b)\in\th$, it suffices by 
the proof of \cite[Lemma 2.6]{Kearnes16:finite} to show that for all $f\in\Pol_2(\zA)$, 
$e\circ f^a\red U$ is a permutation if and only if $e\circ f^b\red U$ is a permutation. 
Suppose for concreteness that $e\circ f^a\red U$ is a permutation but 
$e\circ f^b\red U$ is not. Let $g(x,y)=e(f(x,e(y)))\in\Pol_2(\zA)$. Because 
$e\circ f^a\red U$ is a permutation, $e\red U$ is the identity map, and 
$\beta\red U\not\sse\al$, we get $g^a(\beta)\not\sse\al$. On the other 
hand, because $e\circ f^b\red U$ is not a permutation, we get 
$e\circ f^b(\beta\red U)\sse\al$ by Lemma~\ref{lem:minimal-sets}(2) and hence 
$g^b(\beta)\sse\al$, contradicting $(a,b)\in\zeta(\al,\beta)$.
\end{proof}

%%%%%%%%%%%%%%%%%%%%%%%%%%%%%%%%%%
\subsection{Alignment}\label{sec:alignment}

In this and the next sections we 
prove several properties of the centralizer. The first one concerns 
properties of a relation with respect to centralizer blocks.

Let $\rel$ be a subdirect product of $\zA_1\tm\dots\tm\zA_n$, $i,j\in[n]$, 
and $\al_i\in\Con(\zA_i)$, 
$\al_j\in\Con(\zA_j)$. The coordinate positions $i,j$ are said to be 
\emph{$\al_i\al_j$-aligned in $\rel$} 
if, for any $(a,c),(b,d)\in\pr_{ij}\rel$, 
$(a,b)\in\al_i$ if and only if $(c,d)\in\al_j$. Or in other words, the link 
congruences of $\zA_i,\zA_j$ with respect to $\pr_{ij}\rel$ are contained in
$\al_i,\al_j$, respectively, and the isomorphism between $\zA_i\fac{\lnk_i}$ and $\zA_j\fac{\lnk_j}$ induced by $\pr_{ij}\rel$ bijectively maps $\al_i\fac{\lnk_i}$-blocks to $\al_j\fac{\lnk_j}$-blocks.

\begin{lemma}\label{lem:delta-alignment}
Let $\rel$ be a subdirect product of 
$\zA_1\tm\zA_2$, $\al_i,\beta_i\in\Con(\zA_i)$, $\al_i\prec\beta_i$, 
for $i=1,2$. If $(\al_1,\beta_1)$ and $(\al_2,\beta_2)$ cannot be separated 
from each other, then the coordinate positions 1,2 are $\zeta(\al_1,\beta_1)\zeta(\al_2,\beta_2)$-aligned in $\rel$. If, in addition, $\typ(\al_1,\beta_1),\typ(\al_2,\beta_2)\ne\one$, then the coordinate positions 1,2 are $(\al_1:\beta_1)(\al_2:\beta_2)$-aligned in $\rel$.
\end{lemma}

\begin{proof}
Let us assume the contrary, that is, without loss of generality there are 
$a,b\in\zA_1$ and $c,d\in\zA_2$
with $(a,c),(b,d)\in\rel$, $(a,b)\in\zeta(\al_1,\beta_1)$, but 
$(c,d)\not\in\zeta(\al_2,\beta_2)$. Therefore there is $g(x,y)\in\Pol_2(\zA_2)$ 
such that $g^c(\beta_2)\sse\al_2$
but $g^d(\beta_2)\not\sse\al_2$, or the other way round.
Extend $g$ to a polynomial of $\rel$. We have 
$g^a(\beta_1)\sse\al_1$ if and only if $g^b(\beta_1)\sse\al_1$. Therefore, 
there is a polynomial of $\rel$ that separates $(\al_1,\beta_1)$ from 
$(\al_2,\beta_2)$ or the other way round, a contradiction. The second statement of the lemma follows from Proposition~\ref{pro:two-centralizers}.
\end{proof}

By \cite[Theorem 5.5]{Hobby88:structure} if $\typ(\al,\beta)\ne\one,\two$ then 
$(\al:\beta)\meet\beta\le\al$. This and Lemmas~\ref{lem:type-equal},\ref{lem:delta-alignment} 
imply the following

\begin{lemma}\label{lem:34-links}
Let $\rel$ be a subdirect product of $\zA$ and $\zB$ and let 
$\al,\beta\in\Con(\zA)$, $\gm,\dl\in\Con(\zB)$ be such that $\al\prec\beta$, 
$\gm\prec\dl$, and $(\al,\beta),(\gm,\dl)$ cannot be separated. Let also 
$\lnk_1,\lnk_2$ be the link congruences of $\zA,\zB$, respectively. If 
$\typ(\al,\beta)\ne\one,\two$ then $\lnk_1\meet\beta\le\al$, $\lnk_2\meet\dl\le\gm$.
\end{lemma}

%%%%%%%%%%%%%%%%%%%%%%%%%%%%%%%%%%
\subsection{Centralizer and congruence blocks}\label{sec:congruence-centralizer}

In this section we prove several properties indicating relationship between
congruence blocks inside a centralizer block. All the algebras are assumed to be from~$\cV$.

\begin{lemma}\label{lem:central-mapping}
Let $\zA\in\cV$, $\al,\beta\in\Con(\zA)$ be such that $\al\prec\beta$ and 
$\typ(\al,\beta)=\two$. Let $\zeta=(\al:\beta)$, let $B,C$ be $\beta$-blocks from the same $\zeta$-block such that either $BC$ is 
a thin edge in $\zA\fac\beta$ or $B=C$. Let $b\in B,c\in C$ be such that $bc$ is a thin 
edge of the same type as $BC$ if $B\ne C$ (such $b,c$ exist by Lemma~\ref{lem:thin-properties}(1)), or $bc$ is a thin edge of any type if $B=C$, or $b=c$. Let $f$ be a unary polynomial of $\zA$ and assume that one of the following holds:
\begin{itemize}
\item 
$b=c$ or $bc$ is a thin semilattice edge, and $f(x)=x\cdot c$;
\item 
$b=c$ or $bc$ is a thin majority edge, $t_{bc}$ is any term satisfying the conclusion of Lemma~\ref{lem:op-s-on-affine}(1) for $(b,c)$, and $f(x)=t_{bc}(x,c)$;
\item 
$b=c$ or $bc$ is a thin affine edge, $h_{bc}$ is any term satisfying the conclusion of Lemma~\ref{lem:op-s-on-affine}(2) for $(b,c)$, and $f(x)=h_{bc}(x,b,c)$.
\end{itemize}
%% the polynomial $f(x)=x\cdot c$ if $b\le c$, $f(x)=t_{bc}(x,c)$ if $bc$ is 
%% majority, and $f(x)=h_{bc}(x,b,c)$ if $bc$ is affine, where $t_{bc},h_{bc}$ 
%% are the operations from Lemma~\ref{lem:op-s-on-affine}, 
Then $f$ is an injective mapping from $B\fac\al$ to $C\fac\al$ and $f(b)=c$. Also, if $b\eqc\al c$ then $f(x)\eqc\al x$ for any $x\in B$.

Moreover, if $BC$ is a thin semilattice edge then for any $a\in B\fac\al$, $a\le f(a)$ and
$a\not\le d$ for any other $d\in C\fac\al$.
\end{lemma}

\begin{proof}
%% We can assume that $\al$ is the equality relation $\zz_1$. 
Suppose $f(a_1)\eqc\al f(a_2)$ for some $a_1,a_2\in B$, $a_1\not\eqc\al a_2$. 
%% Since $\typ(\zz_1,\beta)=\two$, by 
%% Corollary~\ref{cor:max-min-set}(1) every pair of elements
%% of $B$ is a $(\zz_1,\beta)$-subtrace. Let $f'$ be an idempotent unary 
%% polynomial such that $f'(a_1)=a_1$, $f'(a_2)=a_2$, and $f'(\zA)$
%% is a $(\zz_1,\beta)$-minimal set. Note that if $B=C$ and $b\ne c$ then $bc$ has to be of the affine type.
%% 
If $b\le c$, let $s(x,y)=x\cdot y$; if $bc$ is a thin majority edge, let $s(x,y)=t_{bc}(x,y)$; and, if $bc$ is a thin affine edge, let $s(x,y)=h_{bc}(x,b,y)$. In each case we have $s(x,c)=f(x)$. Also, $s(x,b)\eqc\al x$ for all $x\in B$. Indeed, if $b\le c$ then $x\le s(x,b)=x\cdot b\eqc\al x$ for $B\fac\al$ is a module and contains no semilattice edge. If $bc$ is a thin majority edge then $s(x,b)=t_{bc}(x,b)\eqc\al x$ by Lemma~\ref{lem:op-s-on-affine}(1). If $bc$ is a thin affine edge then $s(x,b)=h_{bc}(x,b,b)\eqc\al x$ by Lemma~\ref{lem:op-s-on-affine}(2). Now let $g(x,y)=s(y,x)$. We have $g^c(a_1)=f(a_1)\eqc\al f(a_2)=g^c(a_2)$. Since $(b,c)\in\zeta$ and $(a_1,a_2)\in\beta$, this implies $g^b(a_1)\eqc\al g^b(a_2)$, so $a_1\eqc\al a_2$. Note that if $b\eqc\al c$ then we also obtain $f(x)\eqc\al x$ for all $x\in B$.
%% let $g(x,y)=f'(y)\cdot x$. Then $g^c(x)=g(c,x)=f(x)$ on 
%% $\{a_1,a_2\}$,
%% that is, $g^c(a_1)=g^c(a_2)$ implying $g^c(\beta)\sse\zz_1$. On the other hand, 
%% $g^b(x)=f'(x)$ on $\{a_1,a_2\}$ implying $g^b(\beta)\not\sse\zz_1$, 
%% a contradiction with the assumption $b\eqc\zeta c$. 
%% 
%% If $bc$ is a thin majority edge, set $g(x,y)=t_{bc}(f'(y),x)$. Then 
%% $g^c(a_1)=f(a_1)=f(a_2)=g^c(a_2)$, and so $g^c(\beta)\sse\zz_1$. On the
%% other hand, since $B$ is a module, $a_1b,a_2b$ are affine edges 
%% and $\zz_1$ witnesses that. Therefore $g^b(a_1)=a_1$ and $g^b(a_2)=a_2$,
%% implying $g^b(\beta)\not\sse\zz_1$, and we have a contradiction again.
%% 
%% If $bc$ is a thin affine edge, we consider the polynomials 
%% $g(x,y,z)=h_{bc}(f'(x),y,z)$ and $g^{bc}(x)=g(x,b,c)$, 
%% $g^{a_1a_1}(x)=g(x,a_1,a_1)$. Again, 
%% $g^{bc}(a_1)=f(a_1)=f(a_2)=g^{bc}(a_2)$, while 
%% $$
%% g^{a_1a_1}(a_1)=h_{bc}(f'(a_1),a_1,a_1)=a_1\ne a_2=
%% h_{bc}(f'(a_2),a_1,a_1)=g^{a_1a_1}(a_2),
%% $$
%% since $a_1a_2$ is an affine edge as witnessed by the equality congruence.
%% Here $b,c$ and $a_1,a_2$ play the role $a,b$ and $c,d$, respectively, in Lemma~\ref{lem:op-s-on-affine}.
%% This implies that $g^{bc}(\beta)\sse\zz_1$ and $g^{a_1a_1}(\beta)\not\sse\zz_1$,
%% a contradiction.

Finally, to prove the last statement of the lemma, suppose that there are $c,d\in C\fac\al$, $c\ne d$, such that $a\le c$ and $a\le d$. Consider the term operation $t(x,y)=x\cdot y$. Then $t^a(c)=c$ and $t^a(d)=d$. However, since $C\fac\al$ is a module, it contains no semilattice edges, and hence $t^c(c)=c=t^c(d)$. This contradicts the assumption $(a,c)\in\zeta$.
\end{proof}

\begin{corollary}\label{cor:centralizer-multiplication}
Let $\zA\in\cV$, $\al,\beta\in\Con(\zA)$ be such that $\al\prec\beta$ and 
$\typ(\al,\beta)=\two$, and so $\beta\le\zeta=(\al:\beta)$. Let $a,b,c\in\zA$ be such that $a\eqc{(\al:\beta)} b\eqc\beta c$. Then $ab\eqc\al ac$.
\end{corollary}

\begin{proof}
We have $ab\eqc\beta ac$ and $a\le ab,a\le ac$. If $a\eqc\beta ab$, then $a\eqc\al ab\eqc\al ac$, because $B\fac\al$, where $B$ is the $\beta$-class containing $a,ab,ac$ is a module and contains no semilattice edges. Otherwise
by the last statement of Lemma~\ref{lem:central-mapping} $ab\eqc\al ac$.
\end{proof}

Another straightforward application of Lemma~\ref{lem:central-mapping} is
the following

\begin{corollary}\label{cor:equal-centralizer}
Let $\zA\in\cV$, $\al,\beta\in\Con(\zA)$ be such that $\al\prec\beta$ and 
$\typ(\al,\beta)=\two$, and let $\zeta=(\al:\beta)$. Then for any 
$\beta$-blocks $B_1,B_2$ that belong to the same $\zeta$-block $C$ and 
such that $B_1\sqq^{asm}B_2$ and $B_2\sqq^{asm} B_1$ in $C\fac\beta$, 
$|B_1\fac\al|=|B_2\fac\al|$.
\end{corollary}

%%%%%%%%%%%%%%%%%%%%%%%%%%%%%%%%%
%%%%%%%%%%%%%%%%%%%%%%%%%%%%%%%%%
\section{Collapsing polynomials}

In this section we introduce and prove the existence of polynomials that 
collapse all prime intervals in congruence lattices of factors of a subproduct,
except for a set of factors that cannot be separated from each other.

We start with an auxiliary statement.

\begin{lemma}\label{lem:maximal-minimal}
Let $\zA\in\cV$ and $\al,\beta\in\Con(\zA)$ be such that $\al\prec\beta$ 
and $\typ(\al,\beta)=\two$; let $B$ be a $\beta$-block containing more than one $\al$-block, and $a,b\in B$ with $a\sqq^{asm}_B b$. 
\begin{itemize} 
\item[(1)] 
There exists a polynomial $f$ such that $f(a)=b$ and $f$ is bijective on $B\fac\al$; in particular, $f(\beta\red B)\not\sse\al$.
\item[(2)] 
If $a$ belongs to an $(\al,\beta)$-trace, so does $b$. In particular, every element from $\umax(B)$ belongs to an $(\al,\beta)$-trace.
\item[(3)] 
Let $a\eqc\al b$.
%% , $a\sqq^{asm}b$ in $a\fac\al$.
%% , and $N$ an $(\al,\beta)$-trace with  $a\in N$. 
Then there is a polynomial $f$
such that $f(a)=b$ and $f$ restricted on $B\fac\al$ is the identity map.
%% , $N'=f(N)$ is an $(\al,\beta)$-trace containing $b$, and $N'\fac\al=N\fac\al$.
\end{itemize}
\end{lemma}

\begin{proof}
(1) follows from Lemma~\ref{lem:central-mapping}.

(2) Suppose $a\in N$, an $(\al,\beta)$-trace. Then the polynomial $f$ constructed in item (1) satisfies $f(a)=b$ and for any $a'\in N$ with $a'\fac\al\ne a\fac\al$ it holds that $f(a')\not\eqc\al f(a)$. Therefore $f(N)$ is a trace and $b\in f(N)$. The second claim of item (2) follows from the observation that by Lemma~\ref{lem:minimal-sets}(5) $B$ contains an $(\al,\beta)$-trace.

(3) Let $f$ be a polynomial constructed in item (1). It satisfies the condition $f(a)=b$ and we only need to prove that $f$ is the identity mapping on $B\fac\al$. Let $a=a_0,a_1\zd a_\ell=b$ be an asm-path in $B$, then $f$ is a composition $f=f^\ell\circ\dots\circ f_1$, where each $f^i$ is one of the three types identified in Lemma~\ref{lem:central-mapping}. Note that if $a_{i-1}a_i$ is a thin semilattice or majority edge, or $a_{i-1}\eqc\al a_i$, then $f_i$ is the identity mapping on $B\fac\al$. Therefore, it suffices to assume that for every $i\in[\ell]$ the edge $a_{i-1}a_i$ is affine, $a_{i-1}\not\eqc\al a_i$, and $f^i=h_{a_{i-1}a_i}(x,a_{i-1},a_i)$. Since $B\fac\al$ is a module, $h_{a_{i-1}a_i}(x,y,z)=p_ix+q_iy+r_iz$, where $+$ is the operation of the module and $p_i,q_i,r_i$ are elements of the ring associated with the module. As $h_{a_{i-1}a_i}$ satisfies the conditions of Lemma~\ref{lem:op-s-on-affine}(2), we have $h_{a_{i-1}a_i}(c,d,d)=c$ for any $c,d\in B\fac\al$; hence, $p_i=1$. Then 
\[
f(x)=x+q_1a_0\fac\al+r_1a_1\fac\al+q_2a_1\fac\al+\dots+r_\ell a_\ell\fac\al.
\]
Since $f(a\fac\al)=a\fac\al$, it holds that $q_1a_0\fac\al+r_1a_1\fac\al+q_2a_1\fac\al+\dots+r_\ell a_\ell\fac\al=0$ and $f(x)=x$ on $B\fac\al$.
%% It suffices to consider the case when $ab$ is a thin edge. Let $f$ be the polynomial constructed in Lemma~\ref{lem:central-mapping}. The result follows by Lemma~\ref{lem:central-mapping}.
%% We argue as in the proof of Lemma~\ref{lem:central-mapping} that $f(x)\eqc\al x$ for $x\in B$. If $a\le b$ then $x\le f(x)=xb\eqc\al x$ for $B\fac\al$ is a module and contains no semilattice edge. If $ab$ is a thin majority edge then $f(x)=t_{ab}(x,b)\eqc\al x$ by Lemma~\ref{lem:op-s-on-affine}(1). If $ab$ is a thin affine edge then $f(x)=h_{ab}(x,a,b)\eqc\al h_{ab}(x,a,a)\eqc\al x$ by Lemma~\ref{lem:op-s-on-affine}(2) as $a\eqc\al b$. 
%% the proof of item (1). It suffices to notice that if $a\eqc\al b$ then $f(x)\eqc\al x$ for $x\in B$.
\end{proof}

Let $\rel$ be a subdirect product of $\vc\zA n\in\cV$, $\beta_j\in\Con(\zA_j)$, and $B_j$ a $\beta_j$-block, $j\in[n]$. Let also $I\sse[n]$, and 
$\al,\beta\in\Con(\pr_I\rel)$ be such that $\al\prec\beta\le\ov\beta_I$. We call an idempotent unary polynomial 
$f$ of $\rel$ \emph{$\al\beta$-collapsing for $\ov\beta,\ov B$} if 
\begin{itemize}
\item[(a)]
$f$ is $\ov B$-preserving;
\item[(b)]
$f(\beta\red{\ov B'_I})\not\sse\al$; %% contains an $(\al,\beta)$-minimal set, in particular, $f(\beta)\not\sse\al$; 
\item[(c)]
$f(\dl\red{B'_j})\sse\gm\red{B'_j}$, where $B'_j=\pr_j(\rel\cap\ov B)$, for every 
$\gm,\dl\in\Con(\zA_j)$, $j\in[n]$, with $\gm\prec\dl\le\beta_j$, and such 
that $(\al,\beta)$ can be separated from $(\gm,\dl)$ or $(\gm,\dl)$
can be separated from $(\al,\beta)$.
\end{itemize}

%% \begin{theorem}\label{the:collapsing}
%% Let $\rel$ be a subdirect product of $\vc\zA n\in\cV$ and choose 
%% $\beta_j\in\Con(\zA_j)$ and a $\beta_j$-block $B_j$ for each $j\in[n]$; let $\rel$ be
%% chained with respect to $\ov\beta,\ov B$ and $\relov=\rel\cap\ov B$, $B'_j=\pr_j\relov$, $j\in[n]$. Let also $i\in[n]$, and 
%% $\al\in\Con(\zA_i)$ be such that $\al\prec\beta_i$. 
%% If $B'_i\fac\al$ contains a 
%% nontrivial as-component, then there exists an $\al\beta_i$-collapsing 
%% polynomial $f$ for $\ov\beta,\ov B$. Moreover, $f$ can be chosen to
%% satisfy any one of the following conditions:
%% \begin{itemize}
%% \item[(d)] 
%% for any $(\al,\beta_i)$-subtrace $\{a,b\}\sse\amax(\pr_i\relov)$ with $b\in\as(a)$ and any $j\in[n]$, the polynomial $f$ can be chosen such that $a\fac\al,b\fac\al\in f(\zA_i)\fac\al$ and $f(B_j)\cap\umax(B'_j)\ne\eps$;
%
%% \item[(e)] 
%% if $\typ(\al,\beta_i)=\two$, for any $\ba\in\umax(\relov)$ the
%% polynomial $f$ can be chosen such that $f(\ba)=\ba$;
%
%% \item[(f)] 
%% if $\typ(\al,\beta_i)=\two$, a tuple $\ba\in\relov$ is such that $\ba\in\umax(\rel'')$, where 
%% $\rel''=\{\bb\in\rel\mid \bb[i]\eqc\al\ba[i]\}$ and 
%% $\{a,b\}\sse\amax(\pr_i\relov)$ is an $(\al,\beta_i)$-subtrace such 
%% that $\ba[i]=a$ and $b\in\as(a)$, then the polynomial $f$ can be chosen 
%% such that $f(\ba)=\ba$ and $b'\in f(\zA_i)$ for some $b'\eqc\al b$. %% ;\\[1mm]
%% \end{itemize}
%% \end{theorem}

Sometimes we will use the concept of collapsing polynomials in a slightly more general sense. If $I\sse[n]$, the relation $\rel$ can be viewed as a subdirect product of $\pr_I\rel$ and $\zA_i$, $i\in[n]-I$. This means that instead of $\al\prec\beta\le\beta_i$ for an individual factor $\zA_i$ we can use $\al\prec\beta\le\ov\beta_I$ in $\Con(\pr_I\rel)$.

\begin{theorem}\label{the:collapsing}
Let $\rel$ be a subdirect product of $\vc\zA n\in\cV$ and choose 
$\beta_j\in\Con(\zA_j)$ and a $\beta_j$-block $B_j$ for each $j\in[n]$; let $\rel$ be
chained with respect to $\ov\beta,\ov B$ and $\relov=\rel\cap\ov B$, $B'_j=\pr_j\relov$, $j\in[n]$. Let also $I\sse[n]$, and 
$\al\in\Con(\pr_I\rel)$ be such that $\al\prec\ov\beta_I$. 
If $\pr_I\rel'\fac\al$ contains a 
nontrivial as-component, then there exists an $\al\ov\beta_I$-collapsing 
polynomial $f$ for $\ov\beta,\ov B$. Moreover, $f$ can be chosen to
satisfy any one of the following conditions:
\begin{itemize}
\item[(d)] 
for any $\ba,\bb\sse\amax(\pr_I\relov)$ with $\bb\in\as(\ba)$ and $\ba\not\eqc\al\bb$, and any $J\sse[n]$, the polynomial $f$ can be chosen such that $\ba\fac\al,\bb\fac\al\in f(\pr_I\rel)\fac\al$ and $f(\pr_J\rel)\cap\umax(\pr_J\rel')\ne\eps$;
\item[(e)] 
if $\typ(\al,\ov\beta_I)=\two$, for any $\ba\in\umax(\relov)$ the
polynomial $f$ can be chosen such that $f(\ba)=\ba$;
\item[(f)] 
if $\typ(\al,\ov\beta_I)=\two$, a tuple $\ba\in\relov$ is such that $\ba\in\umax(\rel'')$, where 
$\rel''=\{\bb\in\rel'\mid \pr_I\bb\eqc\al\pr_I\ba\}$ and $(\ba_0,\bb_0)\in\ov\beta_I$ such 
%% $\{\ba_0,\bb_0\}\sse\amax(\pr_I\relov)$ is an $(\al,\ov\beta_I)$-subtrace such 
that $\pr_I\ba=\ba_0$, $\bb_0\in\as(\ba_0)$, and $(\ba_0,\bb_0)\not\in\al$, then the polynomial $f$ can be chosen 
such that $f(\ba)=\ba$ and $f(\bb_0)\eqc\al \bb_0$. 
\end{itemize}
\end{theorem}

\begin{proof}
First, we prove that an $\al\beta_I$-collapsing polynomial exists.
By Lemma~\ref{lem:chaining-inher}(1) we may assume $|I|=1$. Since we also consider the action of polynomials on $\zA_i$, $i\in I$, we apply Lemma~\ref{lem:separation-separation}(2) to create `copies' of those coordinates. Suppose $I=\{1\}$, let $C$ be a nontrivial 
as-component of $B'_1\fac\al$. Take $a,b\in B'_1$ such that $a\fac\al,b\fac\al\in C$ and $a\fac\al\ne b\fac\al$. Let $f$ be a $\ov B$-preserving polynomial of $\rel$ such that $f(a\fac\al)=a\fac\al, f(b\fac\al)=b\fac\al$. Such a polynomial exists, as the identity mapping satisfies these conditions. Let $M(f)$ denote the set of triples $(j,\gm,\dl)$ such that $j\in[n]$, $\gm,\dl\in\Con(\zA_j)$, $\gm\prec\dl\le\beta_j$, and $f(\dl\red{B'_j})\sse\gm$. Choose $f$ for which $M(f)$ is maximal (under inclusion). Note that $f$ can be replaced with its idempotent power, so it can be assumed to be idempotent. We show that $f$ is $\al\beta_1$-collapsing.  

Suppose there are $j\in[n]$ and $\gm,\dl\in\Con(\zA_j)$ such that 
$\gm\prec\dl\le\beta_j$, $(\al,\beta_1),(\gm,\dl)$ can be separated, and $(j,\gm,\dl)\not\in M(f)$. Then, since $\rel$ is chained, there is a unary $\ov B$-preserving polynomial $f_{j\gm\dl}$ of $\rel$ such that $f_{j\gm\dl}(a\fac\al)=a\fac\al$, $f_{j\gm\dl}(b\fac\al)=b\fac\al$ (this implies $f_{j\gm\dl}(\beta_1)\not\sse\al$) and $f_{j\gm\dl}(\dl\red{B'_j})\sse\gm\red{B'_j}$. Indeed, if $(\al,\beta_1)$ can be separated from $(\gm,\dl)$, by Lemma~\ref{lem:good-polys}(3) $(a\fac\al, b\fac\al)\in T_\al(a\fac\al,b\fac\al,\lb \cU(\gm,\dl,\ov B))$, and $f_{j\gm\dl}$ can be chosen to be the polynomial witnessing this. If $(\gm,\dl)$ can be separated from $(\al,\beta_1)$, then $f_{j\gm\dl}$ exists by Theorem~\ref{the:relative-symmetry}.
Let $g$ be the idempotent power of $f_{j\gm\dl}\circ f$. We have $g(a\fac\al)=a\fac\al, g(b\fac\al)=b\fac\al$, $M(f)\sse M(g)$, but $g(\dl\red{B'_j})\sse\gm$, a contradiction with the choice of  $f$. 

Next we prove that any one of conditions (d)--(f) can be satisfied. 
For condition (d), that for any $a,b\in\amax(\pr_1\relov)$ with $b\in\as(a)$ the polynomial $f$ can be chosen such that $a\fac\al,b\fac\al\in f(\zA_1)\fac\al$ follows from what is proved above, since if $a,b\in\amax(\pr_1\relov)$ and 
$b\in\as(a)$, then $a\fac\al,b\fac\al$ are members of a nontrivial 
as-component of $B'_1\fac\al$ and the result follows from the chained condition and Lemma~\ref{lem:good-polys}(3). It remains to show that, for any $J\sse[n]$, $f$ can also be chosen to satisfy $f(\pr_J\rel)\cap\umax(\pr_J\relov)\ne\eps$. 

Again, by Lemmas~\ref{lem:separation-separation},~\ref{lem:chaining-inher}(1) we may assume $I\cap J=\eps$ and $|J|=1$. Suppose $J=\{2\}$ and take any $c\in f(B'_2)$. Let also $C=\as(a)\fac\al$. Similar to the proof of Theorem~\ref{the:relative-symmetry} by $\relo^*_c\sse(\zA_1\fac\al)^2\tm\zA_2\tm\rel$ we denote the relation generated by 
$\{(a\fac\al,b\fac\al,c,\ba)\}\ \cup\{(x\fac\al,x\fac\al,y,\bz)\mid \bz\in\rel,\bz[1]=x,\bz[2]=y\}$,
where $\ba$ is an arbitrary element from $\relov$. Let 
$\relo_c=\pr_{123}(\relo^*_c\cap(B'_1\fac\al\tm B'_1\fac\al\tm B'_2\tm\ov B))$.
Observe that $\relo_c$ is exactly the set of triples $(g(a)\fac\al,g(b)\fac\al,g(c))$ 
for  $\ov B$-preserving unary polynomials $g$ of $\rel$. Also, the polynomial $g$ in such a tuple can be assumed to satisfy conditions (a)--(c) of the definition of $\al\beta_1$-collapsing (it may not be idempotent), because $(g(a)\fac\al,g(b)\fac\al,g(c))=(g(f(a))\fac\al,g(f(b))\fac\al,g(f(c)))$. By Lemma~\ref{lem:factor-as}(1) and Lemma~\ref{lem:good-polys}(3), $C^2\sse\pr_{12}\relo_c$ and therefore $(a\fac\al,b\fac\al)$ is an as-maximal element of $\pr_{12}\relo_c$ by Lemma~\ref{lem:as-product}. By the Maximality Lemma~\ref{lem:to-max}(5) there is $d\in\umax(\pr_3\relo_c)=\umax(B'_2)$ such that $(a\fac\al,b\fac\al,d)\in\relo_c$. This means that for any $c\in B'_2$ 
%% for any polynomial $g$ satisfying (a)--(c) and such that $g(a\fac\al)=a\fac\al, g(b\fac\al)=b\fac\al$, and $c\in g(B'_2)$, 
there is a  polynomial $g_c$ satisfying (a)--(c) with $g_c(a\fac\al)=a\fac\al, g_c(b\fac\al)=b\fac\al$, and $d=g_c(c)\in\umax(B'_2)$. Let $(a\fac\al,b\fac\al,c_0)\in\relo_c$.
%%be such that there is a polynomial $h$ satisfying (a)--(c) with $h(a\fac\al)=a\fac\al,h(b\fac\al)=b\fac\al$, and $c_0\in h(B'_2)$. 
Construct a sequence $c_0,c_1,\dots$ by setting $c_{\ell+1}=g_{c_\ell}(c_\ell)$. Clearly, all the $c_\ell$, $\ell>0$, are u-maximal in $B'_2$. There are $r<s$ such that $c_r=c_s$. Then for the polynomial $h=g_{c_{s-1}}\circ\dots\circ g_{c_r}$ it holds that $h(c_r)=c_r$. The idempotent power of $h$ satisfies the desired conditions.

Now, suppose that $\typ(\al,\beta_1)=\two$. We will use 
Lemma~\ref{lem:maximal-minimal}, so we need to identify some 
congruences of $\rel$ related to $\al\beta_1$-collapsing polynomials. Consider the 
congruences $\al'=\al\tm\beta_2\tms\beta_n$ and 
$\ov\beta=\beta_1\tm\beta_2\tms\beta_n$, and a maximal congruence
$\al^*$ of $\rel$, $\al'\le\al^*\le\ov\beta$, such that $\al^*\red{\relov}=\al'\red{\relov}$.
Let $b^1_1,b^2_1\in B'_1$ be such that $b^1_1\fac\al\ne b^2_1\fac\al$, and $b^1_1,b^2_1$ belong to an as-component of
$B'_1$. Such elements exist by Lemma~\ref{lem:factor-as}(2).
By the Maximality Lemma~\ref{lem:to-max}(3,5) these elements can be extended to $\bb_1,\bb_2\in\amax(\relov)$ such that $\bb_1[1]=b^1_1,\bb_2[1]=b^2_1$ and $\bb_1,\bb_2$ belong to the same as-component of $\relov$. 
Let $\beta^*=\Cgg\rel{\al^*\cup\{(\bb_1,\bb_2)\}}$. We claim that $\al^*\prec\beta^*$ and $\beta^*\red{\relov}=\ov\beta\red{\relov}$. Indeed, let $\al^*\prec\gm\le\beta^*$. Then there are $\bc_1,\bc_2\in\relov$, $(\bc_1,\bc_2)\in\gm$, such that $\bc_1[1]\not\eqc\al\bc_2[1]$. As $\typ(\al,\beta_1)=\two$, $B'_1\fac\al$ is a module, in particular, $B'_1\fac\al$ is an as-component. Since $\rel$ is chained, for any $c\fac\al,d\fac\al\in B'_1\fac\al$ there is a $\ov B$-preserving polynomial $g$ such that $g(\bc_1[1]\fac\al)=c\fac\al, g(\bc_2[1]\fac\al)=d\fac\al$, and so $\bc\eqc\gm\bd$ for any $\bc,\bd\in\relov$ such that $\bc[1]=c,\bd[1]=d$. Thus, $\beta^*\le\gm$ and $\gm\red\relov=\ov\beta\red\relov$.
%% Since $\rel$ is chained with respect to $\ov\beta,\ov B$, applying (Q1) for $I=[n]$, 
%% by Lemma~\ref{lem:good-polys}(3) for any $\bc_1,\bc_2\in\relov$ with $\bc_1[1]\not\eqc\al\bc_2[1]$,
%% it holds that $\beta^*\le\Cgg\rel{\al^*\cup\{\bc_1,\bc_2\}}$. Therefore

Observe that $\typ(\al^*,\beta^*)=\two$. Indeed, let $\al^\dg=\al\tm\zo_2\tms\zo_n$, $\beta^\dg=\beta_1\tm\zo_2\tms\zo_n$ be congruences of $\rel$. Then $\al^\dg\prec\beta^\dg$ and $\typ(\al^\dg,\beta^\dg)=\typ(\al,\beta_1)=\two$ by the Third Isomorphism Theorem. Thus, $\beta^\dg\fac{\al^\dg}$ is Abelian. Since $\ov\beta\le\beta^\dg$ and $\ov\beta\meet\al^\dg=\al'$, we get that $\ov\beta\fac{\al'}$ is Abelian. Then $\typ(\al^*,\beta^*)=\two$ by \cite[Theorem~7.2]{Hobby88:structure}. 

Let $a,b\in\amax(B'_1)$ be such that $a\eqc\al\ba[1]$ and $a\fac\al\ne b\fac\al$; note that such an $a$ exists because $B'_1\fac\al$ is a module and hence $\ba[1]\fac\al\in\amax(B'_1\fac\al)$. Let $f$ be the idempotent polynomial found in part (d) such that $f(\beta^*)\not\sse\al^*$ and $f(a\fac\al)=a\fac\al, f(b\fac\al)=b\fac\al$. 
%% for the set $f(\rel)$ contains an $(\al^*,\beta^*)$-minimal set. 
%% As $\relov$ is a $\beta^*$-block and $\relov\fac{\al^*}\cong B'_1\fac\al$ is a module, 
As $\ba$ is u-maximal in $\rel'$, $f(\ba)\sqq^{asm}_{\relov}\ba$. Also, as $f(a\fac\al)=a\fac\al$, $f(\ba)\eqc{\al'}\ba$, and so $f(\ba)\eqc{\al^*}\ba$. Therefore, by Lemma~\ref{lem:maximal-minimal}(3) there exists a polynomial $h$ such that $hf(\ba)=\ba$ and $h$ restricted to $\relov\fac{\al^*}=\relov\fac{\al'}$ is an identity mapping. The polynomial $h$ is $\ov B$-preserving because $h(\relov)\cap\relov\ne\eps$. Also, as $h$ is injective on $\relov\fac{\al^*}$, for the idempotent power $g$ of $hf$ it holds that $g(\ba)=\ba$ and $g(\beta_1\red{B'_1})\not\sse\al$.
%% 
%% Therefore, by 
%% Lemma~\ref{lem:maximal-minimal}(1) for any $\ba\in\umax(\relov)$
%%  and 
%% $hf(\rel)$ contains an $(\al^*,\beta^*)$-minimal set.  
The polynomial $g$ is as required for condition (e).

To prove condition (f) first notice that $\umax(\rel'')\sse\umax(\relov)$. Indeed, as $\typ(\al,\beta_1)=\two$, every $\al'$-block overlapping with $\relov$, where $\al'$ is as defined above, contains a u-maximal element from $\relov$ and Lemma~\ref{lem:u-max-congruence} applies. Thus $\ba\in\umax(\relov)$. As $I=\{1\}$, we assume $\ov\beta_I=\beta_1$ and $\ba_0=a_0,\bb_0=b_0$. Now, choose an $\al\beta_1$-collapsing polynomial
$f$ such that $a',b'\in f(\zA_1)$ for some $a'\eqc\al a_0,b'\eqc\al b_0$.  Such a polynomial exists by part (d). 
We apply Lemma~\ref{lem:maximal-minimal}(3) where the roles of $\zA,\al,\beta,B,a,b$ are played by $\rel,\al^*,\beta^*,\relov,f(\ba),\ba$. Then, since $\relov$ is a $\beta^*$-block and $\ba$ is u-maximal in it, similar to the previous case by 
Lemma~\ref{lem:maximal-minimal}(3) there is a $\ov B$-preserving polynomial $h$ such that $h$ restricted to $\relov\fac{\al^*}=\relov\fac{\al'}$ is the identity mapping, $hf(\ba)=\ba$, $hf(a_0)\eqc\al a_0$, $hf(b_0)\eqc\al b_0$. Thus, the idempotent power of $hf$ satisfies condition~(f).
\end{proof}

%%%%%%%%%%%%%%%%%%%%%%%%%%%%%%%%%%%%
%%%%%%%%%%%%%%%%%%%%%%%%%%%%%%%%%%%%
\section{The Congruence Lemma}\label{sec:congruence}

The main result of this section is the Congruence Lemma~\ref{lem:affine-link}. We start with introducing three closure properties of algebras and their subdirect products. 

Let $\zC$ be a subalgebra of $\zA\in\cV$. A subset $B\sse\zC$ is \emph{as-closed} in $\zC$ if for 
any $a,b\in\zC$ with $a\in\umax(B)$ and $a\sqq^{as}_\zC b$, it holds that $b\in B$. Similarly, the set $B$ is \emph{s-closed} in $\zC$ if for 
any $a,b\in\zC$ with $a\in\umax(B)$ and $a\sqq b$, it holds that $b\in B$. Thus, an as-closed (s-closed) set is just a set of elements closed under thin semilattice and affine edges (respectively, thin semilattice edges). Note that the subalgebra $\zC$ is very important here, as we normally want to `contain' as-closed (s-closed) sets within some subalgebra, and thin edges do not respect subalgebras.

Let $\zA\in\cV$ and $\zB$ a subalgebra of $\zA$. By $\Cg{\zB}$ we denote the congruence of $\zA$ generated by pairs of elements from $\zB$. That is, $\Cg{\zB}$ is the smallest congruence such that $\zB$ is a subset of a $\Cg{\zB}$-block, denote it $\Bl(\zB)$. The subalgebra $\zB$ is said to be \emph{polynomially closed} in $\zA$ if $\zB$ is as-closed in $\Bl(\zB)$.
%% Let $\rel$ be a subdirect product of $\vc\zA n\in\cV$ and $\relo$ a subalgebra of $\rel$. 
%% By $\Cg{\relo}$ we denote the congruence of $\rel$ generated by pairs of elements from $\relo$. That is, $\Cg{\relo}$ is the smallest congruence such that $\relo$ is a subset of a $\Cg{\relo}$-block, denote it $\Bl(\relo)$. The subalgebra $\relo$ is said to be \emph{polynomially closed} in $\rel$ if $\relo$ is as-closed in $\Bl(\relo)$.

\begin{remark}\label{rem:as-closed}
Polynomially closed subalgebras of Mal'tsev algebras are blocks of  congruences.
In the general case the structure of polynomially closed subalgebras is more 
intricate. The intuition 
(although not entirely correct) is that if some block $B$ of a congruence 
$\beta$  
contains several as-components, a polynomially closed subalgebra contains 
some of them and has empty intersection with the rest. However, since we do not even consider non-as-maximal
elements, the actual structure is more complicated.
\end{remark}

The next lemma follows from the definitions and 
Lemma~\ref{lem:thin-properties}, and the fact that congruences are invariant under polynomials.

%% Let $\zA\in\cV$ and let $\rel$ be a subdirect product of $\vc\zA n\in\cV$.
%% \begin{itemize}
%% \item[(1)] 
%% $\rel$ is polynomially closed in $\rel$ and $\zA$ is as-closed and s-closed in $\zA$.
%% %
%% \item[(2)] 
%% For any congruence $\beta\in\Con(\rel)$ and a $\beta$-block $\relo$, the subalgebra $\relo$ is polynomially closed in $\rel$.
%% %
%% \item[(3)] 
%% Let $\relo_1,\relo_2$ be subalgebras of $\rel$,  $\relo_1,\relo_2$ polynomially closed in $\rel$, and $\umax(\relo_1)\cap\umax(\relo_2)\ne\eps$. 
%% Then $\relo_1\cap\relo_2$ is polynomially closed in $\rel$. 
%% 
%% In particular, let $\beta\in\Con(\rel)$ and $\reli$ a $\beta$-block such that $\umax(\relo_1)\cap\umax(\reli)\ne\eps$. Then $\relo_1\cap\reli$ is polynomially closed in $\rel$.
%% 
%% Let $\zB$ be a subalgebra of $\zA$, and let $\zC_1,\zC_2$ be subalgebras of $\zB$ as-closed (s-closed) in $\zB$ and such that $\umax(\zC_1)\cap\umax(\zC_2)\ne\eps$. Then $\zC_1\cap\zC_2$ is as-closed (respectively, s-closed) in~$\zB$. 
%% %
%% \item[(4)] 
%% Let $\zB$ be a subalgebra of $\zA$ and a subalgebra $\zC\sse\zB$ as-closed (s-closed) in $\zB$. Then for any $\beta\in\Con(\zA)$, the algebra  $\zC\fac\beta$ is as-closed (respectively, s-closed) in $\zB\fac\beta$.
%% %

\begin{lemma}\label{lem:poly-closed}
Let $\zA\in\cV$.
\begin{itemize}
\item[(1)] 
$\zA$ is polynomially closed, as-closed, and s-closed in $\zA$.
\item[(2)] 
For any congruence $\beta\in\Con(\zA)$ and a $\beta$-block $\zB$, the subalgebra $\zB$ is polynomially closed in $\zA$.
\item[(3)] 
Let $\zB_1,\zB_2$ be subalgebras of $\zA$,  $\zB_1,\zB_2$ polynomially closed in $\zA$, and $\umax(\zB_1)\cap\umax(\zB_2)\ne\eps$. 
Then $\zB_1\cap\zB_2$ is polynomially closed in $\zA$. 

In particular, let $\beta\in\Con(\zA)$ and $T$ a $\beta$-block such that $\umax(\zB_1)\cap\umax(T)\ne\eps$. Then $\zB_1\cap T$ is polynomially closed in $\zA$.

Let $\zB$ be a subalgebra of $\zA$, and let $\zC_1,\zC_2$ be subalgebras of $\zB$ as-closed (s-closed) in $\zB$ and such that $\umax(\zC_1)\cap\umax(\zC_2)\ne\eps$. Then $\zC_1\cap\zC_2$ is as-closed (respectively, s-closed) in~$\zB$. 
\item[(4)] 
Let $\zB$ be a subalgebra of $\zA$ and a subalgebra $\zC\sse\zB$ as-closed (s-closed) in $\zB$. Then for any $\beta\in\Con(\zA)$, the algebra  $\zC\fac\beta$ is as-closed (respectively, s-closed) in $\zB\fac\beta$.
\item[(5)] 
Let $\vc\zA m\in\cV$, $\rel_i$, $i\in[k]$, be a subdirect product of some of the $\zA_j$'s,  and $\relo_i$ polynomially closed in $\rel_i$, $i\in[k]$. Let $\Phi$ be a multi-sorted conjunctive formula in variables $\vc xn$, where the sort of each $x_j$ is some $\zA_{s(j)}$, and predicate symbols $\vc Pk$, where each $P_i$ has the same arity and coordinate sorts as $\rel_i$ and $\relo_i$. Let $\rel=\Phi(\vc\rel k)$ be the subuniverse of $\zA_{s(1)}\tms\zA_{s(n)}$ defined by $\Phi$, when the $P_i$'s are interpreted by $\rel_i$'s, and let $\relo=\Phi(\vc\relo k)$ be obtained the same way using the $\relo_i$'s.
%% 
%% $\rel,\relo$ be conjunctive-defined through $\vc\rel k$ and $\vc\relo k$, respectively, 
%% by the same conjunctive formula $\Phi$; that is, 
%% $\rel=\Phi(\vc\rel k)$ and $\relo=\Phi(\vc\relo k)$. In other words $\rel$ consists of all tuples that are satisfying assignments of $\Phi$ which uses the $\rel_i$'s as atoms, while $\relo$ is obtained in the same way only replacing the $\rel_i$'s with the $\relo_i$'s.  
Suppose that for every atom $\relo_i(x_{j_1}\zd x_{j_\ell})$ it holds that  $\umax(\pr_{\{x_{j_1}\zd x_{j_\ell}\}}\relo)\sse\umax(\relo_i)$. Then $\relo$ is polynomially closed in $\rel$.
\item[(6)] 
Let $\rel_i$ be a subdirect product of some algebras from $\cV$, $\relo_i$ subalgebra of $\rel_i$, and $\reli_i$ a subalgebra of $\relo_i$ as-closed (s-closed) in $\relo_i$, $i\in[k]$. Let
$\rel,\relo$, and $\reli$ be pp-defined through $\vc\rel k$, $\vc\relo k$, and $\vc\reli k$, respectively, by the same pp-formula $\exists\ov x\Phi$; that is, 
$\rel=\exists\ov x\Phi(\vc\rel k)$, $\relo=\exists\ov x\Phi(\vc\relo k)$, and $\reli=\exists\ov x\Phi(\vc\reli k)$. Let also $\relo'=\Phi(\vc\relo k)$ and $\reli'=\Phi(\vc\reli k)$ and suppose that for every atom $\reli_i(x_{j_1}\zd x_{j_\ell})$ it holds that  $\umax(\pr_{\{x_{j_1}\zd x_{j_\ell}\}}\reli')\sse\umax(\reli_i)$. Then $\reli$ is as-closed (respectively, s-closed) in~$\relo$.
\end{itemize}
\end{lemma}

\begin{proof}
(1) is straightforward. (2) holds because if $\zB$ is a congruence block of $\beta$, then clearly $\Bl(\zB)=\zB$.

For (3) it suffices to observe that, as $\umax(\zB_1)\cap\umax(\zB_2)\ne\eps$, by Lemma~\ref{lem:u-max-congruence} it holds that $\umax(\zB_1\cap\zB_2)\sse\umax(\zB_1)\cap\umax(\zB_2)$. 
We have $\Cg{\zB_1\cap\zB_2}\le\Cg{\zB_1}\meet\Cg{\zB_2}$ and $\Bl(\zB_1\cap\zB_2)\sse\Bl(\zB_1)\cap\Bl(\zB_2)$. Therefore for any $a\in\umax(\zB_1\cap\zB_2)$ and $b\in\Bl(\zB_1\cap\zB_2)$  such that $a\sqq^\as b$ in $\Bl(\zB_1\cap\zB_2)$, it holds that $b\in\zB_1$ and $b\in\zB_2$.

The statement about as-closed (s-closed) subalgebras is straightforward because by Lemma~\ref{lem:u-max-congruence} $\umax(\zC_1\cap\zC_2)\sse\umax(\zC_1)\cap\umax(\zC_2)$.

Part (4) follows from Lemmas~\ref{lem:thin-properties} and~\ref{lem:u-max-congruence}.

(5) Suppose that $\Phi$ involves $n$ variables $\vc xn$. Take $\ba\in\umax(\relo)$ and $\bb\in\Bl(\relo)$ such that $\ba\sqq^{as}\bb$ in $\Bl(\relo)$. Consider some atoms $\rel_i$, $\relo_i$ of $\Phi$; let $X$ be the set of variables involved in $\rel_i,\relo_i$, and set $\ba_i=\pr_X\ba,\bb_i=\pr_X\bb$. By the Maximality Lemma~\ref{lem:to-max}(4,6) $\ba_i\in\umax(\relo_i)$ and $\ba_i\sqq^{as}_{\rel_i}\bb_i$. 
As is easily seen, $\pr_X\Bl(\relo)\sse\Bl(\relo_i)$, and since $\relo_i$ is polynomially closed, $\bb_i\in\relo_i$. As this holds for every $i\in[k]$, $\bb\in\relo$.

(6) Suppose again that $\Phi$ involves $n$ variables $\vc xn$ and  $x_{m+1}\zd x_n$ are quantified. Take $\ba\in\umax(\reli)$, and $\bb\in\relo$ such that $\ba\sqq^{as}_\relo\bb$.  By the assumptions of the lemma and the Maximality Lemma~\ref{lem:to-max}(3,5) there are $\ba'\in\umax(\reli')$ and $\bb'\in\relo'$ such that $\ba=\pr_{[m]}\ba', \bb=\pr_{[m]}\bb'$, and $\ba'\sqq^{as}_{\relo'}\bb'$. Consider an atom $\relo_i$ (respectively, $\reli_i$) of $\Phi$; let $X$ be the set of variables involved in $\relo_i,\reli_i$, and set $\ba_i=\pr_X\ba',\bb_i=\pr_X\bb'$. By the Maximality Lemma~\ref{lem:to-max}(4,6) $\ba_i\in\umax(\reli_i)$ and $\ba_i\sqq^{as}_{\relo_i}\bb_i$. Since $\reli_i$ is as-closed, $\bb_i\in\reli_i$. As this holds for every $i\in[k]$, $\bb'\in\reli'$, and so $\bb\in\reli$. For s-closedness the proof is identical replacing $\sqq^{as}$ with $\le$.
\end{proof}

To explain what Lemma~\ref{lem:affine-link} (the Congruence Lemma) 
amounts to saying consider this: let 
$\relo\sse\zA'\tm\zB'$ be a subdirect product and let the link congruence of 
$\zA'$ be the equality relation. Then, clearly, $\relo$ is the graph of 
a mapping $\sg:\zB'\to\zA'$, and the kernel of this mapping is the link
congruence $\eta$ of $\zB'$ with respect to $\relo$. Suppose now that $\relo$
is a subalgebra of $\rel$, a subdirect product of $\zA\tm\zB$ such that 
$\zA'$ is a subalgebra of $\zA$ and $\zB'$ is a subalgebra of $\zB$. Then
the restriction of the link congruence of $\zA$ with respect to $\rel$ to $\zA'$
does not have to be the equality relation, and similarly the restriction of
the link congruence of $\zB$ to $\zB'$ does not have to be $\eta$, even if $\relo=\rel\cap(\zA'\tm\zB')$. Most 
importantly, the restriction of $\Cgg\zB\eta$, the congruence of $\zB$ generated by 
$\eta$, to $\zB'$ does not have to be $\eta$. The Congruence 
Lemma~\ref{lem:affine-link} shows, however, that this is exactly what 
happens when $\rel,\relo$ and $\zA',\zB'$ satisfy some additional conditions, 
such as being chained and polynomially closed.

\begin{lemma}[The Congruence Lemma]\label{lem:affine-link}
Let $\rel$ be a subdirect product of $\vc\zA n\in\cV$, $\beta_i$ a congruence of $\zA_i$ and let $B_i$ be 
a $\beta_i$-block, $i\in[n]$. Also, let $\rel$ be chained 
with respect to $\ov\beta,\ov B$ and $\relov=\rel\cap\ov B$, $B'_i=\pr_i\relov$. 
Let $I_0\sse[n]$ (without loss of generality assume say, $I_0=[\ell]$) and $\al\in\Con(\pr_{I_0}\rel)$ be such that $\al\prec\ov\beta_{I_0}=\prod_{i\in I_0}\beta_i$, let $I\sse[n]-I_0$, $I'=I\cup I_0$, $J=[n]-I'$,
and let $\relo$ be a subalgebra of $\relov$ 
polynomially closed in $\rel$ and such that $E_1=\pr_{I_0}\relo$ contains an as-component $C$ of $\pr_{I_0}\relov$ and there is $\ba\in\relo\cap\amax(\relov)\ne\eps$ (or $\ba\in\relo\cap\max(\relov)\ne\eps$) with $\pr_{I_0}\ba\in C$.
We consider $\rel,\relov$, and $\relo$ as subdirect products of $\pr_{I_0}\rel\tm\pr_I\rel\tm\pr_J\rel$, $\pr_{I_0}\relov\tm\pr_I\relov\tm\pr_J\relov$, and $\pr_{I_0}\relo\tm\pr_I\relo\tm\pr_J\relo$, respectively. Let $\relov\fac\al=\relov\fac{\al\tm\zz_{\ell+1}\tms\zz_n}$, $\relo\fac\al=\relo\fac{\al\tm\zz_{\ell+1}\tms\zz_n}$, and let $\lnk'_1,\lnk'_2$ and $\lnk^Q_1,\lnk^Q_2$ be the link congruences of $\pr_{I'}\relov\fac\al$ and $\pr_{I'}\relo\fac\al$, respectively, and $E_2=\pr_I\relo$. Let also $E^C_2=\relo[C\fac\al]=\{\pr_I\ba\mid\ba\in\relo, \pr_{I_0}\ba\fac\al\in C\fac\al\}$ and $E''_2=\umax(E_2)\cap E^C_2$.
Then either 
\begin{itemize}
\item[(1)] 
$C\fac\al\tm\umax(E^C_2)\sse\pr_{I'}\relo\fac\al$, or 
\item[(2)] 
there is $\eta\le\ov\beta_I$, maximal (under inclusion) congruence of $\pr_I\rel$ with $\eta\red{E''_2}\sse\lnk^\relo_2$, and $\gm=\Cgg{\pr_I\rel}{E''_2}\vee\eta$, such that
$\eta\prec\gm\le\beta_I$ and the intervals $(\al,\ov\beta_{I_0})$ and $(\eta,\gm)$ cannot be separated in $\rel$. 
\end{itemize}
Moreover, in case (2) $\pr_{I'}\relo\fac\al$ is the graph of a mapping 
$\vf: E_2\to E_1\fac\al$ such that the kernel $\lnk^Q_2$ of $\vf$ is the restriction of $\eta$ on $E_2$, and for any $\ov B$-preserving polynomial $f$ such that $f(\ov\beta_{I_0}\red C)\not\sse\al$ it holds that $f(\gm\red{\pr_I\rel'})\not\sse\eta$.\footnote{This property is somewhat stronger than non-separability. The non-separability of $(\al,\ov\beta_{I_0})$ and $(\eta,\gm)$ only implies that $f(\gm)\not\sse\eta$ without any restrictions on $f(\gm\red{\ov B_I})$.}
\end{lemma}

\begin{observation}\label{obs:congruence-lemma}
(a) The statement of case (1) can be made stronger. Since $\{c\fac\al\}\tm\umax(E^C_2)\sse\pr_{I'}\relo\fac\al$ for any $c\in C$ and by idempotency, $\{c\fac\al\}\tm\Sg{\umax(E^C_2)}\sse\pr_{I'}\relo\fac\al$ and so $C\fac\al\tm\Sg{\umax(E^C_2)}\sse\pr_{I'}\relo\fac\al$.\\[3mm]
(b) Let $\rel'^C=\{\ba\mid \ba\in\relov, \pr_{I_0}\ba\in C\}$. If $\ba\in\rel'^C\cap\amax(\relov)$ (or $\ba\in\rel'^C\cap\max(\relov)$) then $\ba\in\amax(\relov^C)$ (respectively, $\ba\in\max(\relov^C)$). In particular, by the Maximality Lemma~\ref{lem:to-max}(6) $\pr_I\ba\in\amax(\pr_I\relov^C)$ (respectively, $\pr_I\ba\in\max(\pr_I\relov^C)$). 
Indeed, let $\ba\in\rel'^C\cap\amax(\relov)$ and $\bb\in\amax(\relov^C)$ (or $\ba\in\rel'^C\cap\max(\relov)$, $\bb\in\max(\relov^C)$) with $\ba\sqq^{as}_{\rel'^C}\bb$ (respectively, $\ba\sqq_{\rel'^C}\bb$). Since $\ba\in\amax(\relov)$ ($\ba\in\max(\relov)$), we also have $\bb\sqq^{as}_{\relov}\ba$ (respectively, $\bb\sqq_{\relov}\ba$). For an as-path (s-path) $\bb=\bb_1\zd\bb_\ell=\ba$ the sequence $\pr_{I_0}\bb_1\zd\pr_{I_0}\bb_\ell$ is an as-path (s-path) in $\pr_{I_0}\relov$. As $C$ is an as-component of $\pr_{I_0}\relov$, $\pr_{I_0}\bb_i\in C$, $i\in[\ell]$, witnessing that $\bb\sqq^{as}_{\rel'^C}\ba$ (respectively, $\bb\sqq_{\rel'^C}\ba$). 
%% The argument for $\max(\relov)$ is essentially the same.
%% (b) Let $\rel'^C=\{\ba\mid \ba\in\relov, \pr_{I_0}\ba\in C\}$ and $\pr_I\rel'^C\tm C\sse\pr_{I'}\relov$. If $\ba\in\rel'^C\cap\amax(\relov)$ (or $\ba\in\rel'^C\cap\max(\relov)$) then $\ba\in\amax(\reli)$ (respectively, $\ba\in\max(\reli)$), $\reli=\Sg{\umax(\rel'^C)}$. In particular, by the Maximality Lemma~\ref{lem:to-max}(6) $\pr_I\ba\in\amax(\pr_I\reli)$ (respectively, $\pr_I\ba\in\max(\pr_I\reli)$). 
%% 
%% Indeed, note that $\pr_I(\reli)\tm C\sse\pr_{I'}\relov^C$. Let $\ba\in\rel'^C\cap\amax(\relov)$ and $\bb\in\amax(\reli)$ (or $\bb\in\max(\reli)$) with $\ba\sqq^{as}_{\rel'^C}\bb$ (respectively, $\ba\sqq_{}\bb$). Since $\ba\in\amax(\relov)$, we also have $\bb\sqq^{as}_{\relov}\ba$. For an as-path $\bb=\bb_1\zd\bb_\ell=\ba$ the sequence $\pr_{I_0}\bb_1\zd\pr_{I_0}\bb_\ell$ is an as-path in $\pr_{I_0}\relov$. As $C$ is an as-component of $\pr_{I_0}\relov$, $\pr_0\bb_i\in C$, $i\in[\ell]$, witnessing that $\bb\sqq^{as}_{\rel'^C}\ba$. The argument for $\max(\relov)$ is essentially the same.
\end{observation}

We prove the lemma in three steps. First, we assume that $\relo=\relov$ (Lemma~\ref{lem:star-prime}), then we assume that $\relo$ is a congruence block of $\rel$ (Lemma~\ref{lem:etaQ}), and finally we will use polynomial closedness to prove Lemma~\ref{lem:affine-link} itself. In all of these intermediate lemmas we use the notation from Lemma~\ref{lem:affine-link} and assume that the conditions of Lemma~\ref{lem:affine-link} hold. Throughout the proof we will assume that\\[2mm]
(a) $|I_0|=|I|=|J|=1$, say, $I_0=\{1\}$, $I=\{2\}$, $J=\{3\}$, and so $\rel$ is viewed as a ternary relation, a subdirect product of $\zA_1,\zA_2,\zA_3$ ($\relov$ and $\relo$ will also be assumed to be ternary); and\\[1mm]
(b) $\al=\zz_1$ everywhere except the first part of the proof of Lemma~\ref{lem:affine-link} itself.\\[2mm]
The first assumption does not restrict generality, by Lemma~\ref{lem:chaining-inher}(1) and because it does not affect the congruences of $\pr_I\rel$ and $\Cg{\relo}$, relations $\rel$ and $\relo$ viewed as ternary relations satisfy the conditions of the lemma, and the conclusion of the lemma is true for $\rel,\relo$ regardless of a partitioning of the subdirect product. The second assumption requires caution, because while the chaining condition still holds by Lemma~\ref{lem:chaining-inher}(2), $\relo\fac\al$ is not necessarily polynomially closed. However, in the proof we will use the polynomial closedness of $\relo$ only once, and can assume $\al=\zz_1$ elsewhere.  
%% So, from now on $n=3$, $I=\{2\}$, $J=\{3\}$, and $\al=\zz_1$ unless stated otherwise. 

\begin{lemma}\label{lem:star-prime}
If $\relo=\relov$, then the conclusion of Lemma~\ref{lem:affine-link} holds.
\end{lemma}

\begin{proof}
Let $\rel''=\pr_{12}\relov$. If $\lnk'_1\ne\zz_{B'_1}$, then we can choose $(a,u),(b,u)\in\rel''$ with $a\ne b$. Then, since $\rel$ is chained, by Lemma~\ref{lem:good-polys}(3) for any $c,d\in C$ we have $(c,d)\in T_{\zz_1}(a,b,\cU_{\ov B})$. If $f\in\cU_{\ov B}$ witnesses this, then $(c,f(u)),(d,f(u))\in\rel''$ implying $(c,d)\in\lnk'_1$ and so $C^2\sse\lnk'_1$.
%% Since $\rel$ is chained, if $\lnk'_1$ is not the equality relation or $|C|=1$,  by Lemma~\ref{lem:good-polys}(3) 
%% $C^2\sse\lnk'_1$. 
In this case $C\tm\umax(\rel''[C])\sse\rel''$ by Corollary~\ref{cor:subalgebra-path} and Proposition~\ref{pro:umax-rectangular} 
and we obtain option (1) of Lemma~\ref{lem:affine-link} for $\relov$. So, suppose that 
$\lnk'_1$ is the equality relation and $|C|>1$. Then $\rel''$ is the graph
of a mapping $\vf: B'_2\to B'_1$, note that $\lnk'_2$ is the kernel
of $\vf$. 

\medskip

{\sc Claim 1.}
Let $\th\in\Con(\zA_2)$, $\th\le\beta_2$, be such that $\th\red{B'_2}\not\sse\lnk'_2$. Then there is a $\th$-block $D$ such that 
%%% $D'=D\cap\umax(B^C_2)\ne\eps$, 
for $D'=\umax(D\cap B'_2)\cap B^C_2$ it holds that $D'\sse\umax(B'_2)$ and $\th\red{D'}\not\sse\lnk'_2$, where $B^C_2=\rel''[C]$ (note that in the assumptions of Lemma~\ref{lem:star-prime} $B^C_2$ plays the role of $E^C_2$).

\medskip

Suppose that $(e,e')\in\th'-\lnk'_2$, $\th'=\th\red{B'_2}$. Then $(a,e),(b,e')\in\rel''$ for some $a,b\in B'_1$ and $a\ne b$. Since $\rel$ is chained  with respect to $\ov\beta,\ov B$, and by Lemma~\ref{lem:good-polys}(3) for any $a',b'\in C$ there is a $\ov B$-preserving polynomial $f$ such that $f(a)=a',f(b)=b'$. On the other hand, $f(e)\eqc{\th'} f(e')$. 

Let $\rel^\dg=\rel''\fac{\zz_1\tm\th'}$. 
%%be the algebra generated by $(\rel''\cap(C\tm B'_2))\fac{\zz_1\tm\th}$ in $B_1\tm B_2\fac\th$. 
By what is proved $C$ belongs to a block $E$ of the link congruence of $\rel^\dg$. As is easily seen, $C$ is an as-component in $E$. Therefore, by Proposition~\ref{pro:umax-rectangular} applied to $\rel^\dg$ there exists a $\th$-block $D$ such that $D\cap B'_2$ is u-maximal in $B'_2\fac{\th'}$, and such that for any $a'\in C$ there is a $c'\in D\cap B'_2$ with $(a',c')\in\rel''$. By the Maximality Lemma~\ref{lem:to-max}(5) applied to $\rel''\cap(B'_1\tm (D\cap B'_2))$ the element $c'$ can be chosen to be u-maximal (even as-maximal) in $D\cap B'_2$. Also, as $a'\in C$, $c'\in B^C_2$.

Thus, for any $a,b\in C$ there are $e,e'\in B'_2$ and $e_1,e'_1\in \umax(D\cap B'_2)\cap B^C_2$ such that $(e,e')\in\th'$, $(a,e),(a,e_1),(b,e'),(b,e'_1)\in\rel''$. Since $(e,e_1),(e',e'_1)\in\lnk'_2$, we get $(e_1,e'_1)\in \th'-\lnk'_2$. Finally, as $D\cap B'_2$ is u-maximal in $B'_2\fac{\th'}$, by Lemma~\ref{lem:factor-as}(2) $D\cap B'_2$ contains elements u-maximal in $B'_2$.  By Lemma~\ref{lem:u-max-congruence} $\umax(D\cap B'_2)\sse\umax(B'_2)$, and so $D'\sse\umax(B'_2)$. Claim~1 is proved.

\medskip

Pick an arbitrary maximal (under inclusion) congruence $\eta\le\beta_2$ of $\zA_2$ with $\eta\red{B''_2}\sse\lnk'_2$, $B''_2=\umax(B'_2)\cap B^C_2$. Let $\gm$ denote the congruence constructed in Lemma~\ref{lem:affine-link}, that is, $\gm=\Cgg{\zA_2}{B''_2}\join\eta$. 
We prove that $\lnk'_2=\eta\red{B'_2}$, that $\eta\prec\gm$, and that $(\zz_1,\beta_1)$, $(\eta,\gm)$ cannot be separated. 
This clearly implies option (2) of Lemma~\ref{lem:affine-link} for $\relov$. 

\medskip

{\sc Claim 2.}
$\lnk'_2=\eta\red{B'_2}$

\medskip

Note first that by Claim~1 $\eta\red{B'_2}\sse\lnk'_2$, because $\eta\red{B'_2}\not\sse\lnk'_2$ implies $\eta\red{B''_2}\not\sse\lnk'_2$. Suppose that $\lnk'_2\ne\eta\red{B'_2}$. 
We show that in this case there is a congruence $\th$ of $\zA_2$ that is strictly greater than $\eta$ and still satisfies the condition $\th\red{B''_2}\sse\lnk'_2\red{B''_2}$.  
There are $c,d\in B'_2$, $(c,d)\not\in\eta$ and $a\in B'_1$ such that $(a,c),(a,d)\in\rel''$ (as $(c,d)\in\lnk'_2$ and $\lnk'_2$ is the kernel of $\vf$).  
Let $\th=\Cgg{\zA_2}{\{(c,d)\}}\join\eta$ and $\th'=\th\red{B'_2}$. We claim that $\th\red{B''_2}\sse\lnk'_2\red{B''_2}$, and, as $\th$ is strictly greater than $\eta$, it contradicts the choice of $\eta$. Firstly, notice that if $e,e'$ are u-maximal in the same $\th'$-block then $(e,e')\in\lnk'_2$. Indeed, since $\rel$ is chained, for
any $\th$-block $D\sse B_2$ and any $e,e'\in\umax(D\cap B'_2)$ there are
$e=e_1\zd e_{k+1}=e'\in B'_2$, $\vc{e'}k\in B'_2$, $e_i\eqc\eta e'_i$ for $i\in[k]$, and $\ov B$-preserving polynomials 
$\vc fk$ such that $f_i(\{c,d\})=\{e'_i,e_{i+1}\}$. As 
$(a,c),(a,d)\in\rel''$, for every $i\in[k]$, $(e_i,e_{i+1})\in\lnk'_2$,
and therefore $(e,e')\in\lnk'_2$. The result now follows from Claim~1, because if there exists $(e,e')\in\th'-\lnk'_2$ then by Claim~1 we may assume $e,e'\in\umax(D\cap B'_2)$ for some $\th$-block $D$, which is impossible as is proved above. Thus, $\th'\sse\lnk'_2$.
%% Suppose now that $(e,e')\in\th\red{B'_2}-\lnk'_2$. Then $(a,e),(b,e')\in\rel''$ for some $a,b\in B'_1$ and $a\ne b$. Since $\rel$ is chained  with respect to $\ov\beta,\ov B$, and by Lemma~\ref{lem:good-polys}(3) for any $a',b'\in C$ there is a $\ov B$-preserving polynomial $f$ such that $f(a)=a',f(b)=b'$. On the other hand, $f(e)\eqc\th f(e')$. Let $\rel^\dg$ be the algebra generated by $(\rel''\cap(C\tm B'_2))\fac{\zz_1\tm\th}$ in $B_1\tm B_2\fac\th$. By what is proved $\rel^\dg$ is linked and by Proposition~\ref{pro:umax-rectangular} applied to $\rel^\dg$ there exists a $\th$-block $D$ such that for any $a'\in C$ there is a $c'\in D$ with $(a',c')\in\rel''$. By the Maximality Lemma~\ref{lem:to-max}(3) $c'$ can be chosen to be u-maximal in $D\cap B'_2$. Together with what is proved above this contradicts the assumption that $\lnk'_1$ is the equality relation.

\medskip

{\sc Claim 3.} 
For all $\th\in\Con(\zA_2)$ with $\eta\prec\th\le\beta_2$ the intervals 
$(\zz_1,\beta_1),(\eta,\th)$ cannot be separated, and for any $\ov B$-preserving polynomial $f$ such that $f(\beta_1\red C)\not\sse\al$ it also holds that $f(\th\red{B'_2})\not\sse\eta$.
 
\medskip

Suppose that for some $\th\in\Con(\zA_2)$ with $\eta\prec\th\le\beta_2$ 
Claim~3 is not true.
%% the intervals $(\zz_1,\beta_1),(\eta,\th)$ can be separated.

Let $c,d\in B''_2$ with $(c,d)\in\th'-\eta$, $\th'=\th\red{B'_2}$, and $a,b\in B'_1$ such that $(a,c),(b,d)\in\rel''$. We can assume that $a,b\in C$. Suppose there is a $\ov B$-preserving polynomial $g$ such that $g(\beta_1\red C)\not\sse\al$ but $g(\th\red{B'_2})\sse\eta$, say, $a''=g(a')\ne g(b')=b''$ for some $a',b'\in C$. Then by chaining and Lemma~\ref{lem:good-polys}(3) there are $\ov B$-preserving polynomials $g',g''$ with $g'(a)=a',g'(b)=b'$ and $g''(a'')=a,g''(b'')=b$. The polynomial $f=g''\circ g\circ g'$ satisfies the condition $f(a)=a,f(b)=b$ and $f(\th\red{B'_2})\sse\eta$. If either $(\zz_1,\beta_1)$ can be separated
from $(\eta,\th)$, or $(\eta,\th)$ can be separated from $(\zz_1,\beta_1)$, then, since $\rel$ is chained, a polynomial $f$ with these properties can be found by Theorem~\ref{the:collapsing}(d). 
Apply $f$ to the pairs $(a,c),(b,d)$ to obtain $(a,f(c)),(b,f(d))\in\rel''$. Since $f(\th\red{B'_2})\sse\eta\red{B'_2}\sse\lnk'_2$, we obtain $a\eqc{\lnk'_1}b$,
a contradiction with the assumption that $(c,d)\not\in\lnk'_2$.
Claim~3 is proved.

\medskip

Now we use $\lnk'_2=\eta\red{B'_2}$ to show that $\eta\prec\gm$. Clearly, $\eta\le\gm$ and, by the choice of $\eta,\gm$, $\eta\ne\gm$. Again, by the choice of $\eta$ for any $\th\in\Con(\zA_2)$ with $\eta<\th\le\beta_2$ there is a pair $(c,d)\in\th\red{B''_2}-\eta$. Let $(a,c),(b,d)\in\rel''$ for some $a,b\in B'_1$. Note that $a\ne b$, because $(c,d)\not\in\lnk'_2$. We argue similarly to the proof of Claim~1. Since $\rel$ is chained with respect to $\ov\beta,\ov B$ and by Lemma~\ref{lem:good-polys}(3), for any $a',b'\in C$ there is a $\ov B$-preserving polynomial $f$ such that $f(a)=a'$, $f(b)=b'$. Applying $f$ to the pairs $(a,c),(b,d)$ we obtain $(a,f(c)),(b,f(d))\in\rel''$ and $f(c)\eqc\th f(d)$.  Let $\rel^\dg=\rel''\fac{\zz_1\tm\th'}$, $\th'=\th\red{B'_2}$. By what is proved $C$ belongs to a block of the link congruence of $B'_1$ with respect to $\rel^\dg$.  Therefore, $B^C_2\fac{\th'}$ is also a subset of a block of the link congruence of $B'_2\fac{\th'}$ with respect to $\rel^\dg$. This means that $B^C_2$ is a subset of a $\th'\join\lnk'_2$-block. Since $\eta\red{B'_2}=\lnk'_2$ by Claim~2 and $\eta<\th$, we obtain that $B^C_2$, and therefore $B''_2$ is a subset of a $\th$-block. Thus, $\gm\le\th$. 
%% by Proposition~\ref{pro:umax-rectangular} applied to $\rel^\dg$ there exists a $\th$-block $D$ such that $D\cap B'_2$ is u-maximal in $B'_2\fac{\th'}$, and such that for any $a'\in C$ there is a $c'\in D\cap B'_2$ with $(a',c')\in\rel''$. By the Maximality Lemma~\ref{lem:to-max}(5) applied to $\rel''\cap(B'_1\tm (D\cap B'_2))$ the element $c'$ can be chosen to be u-maximal (even as-maximal) in $D\cap B'_2$. Also, as $a'\in C$, $c'\in B^C_2$.
%% be the algebra generated by $\rel''\cap(C\tm B'_2))\fac{\zz_1\tm\th}$ in $B_1\tm B_2\fac\th$. By what is proved $\rel^\dg$ is linked and by Proposition~\ref{pro:umax-rectangular} $\umax(\rel''[C])$ is a subset of a $\lnk'_2\vee\th\red{B'_2}$-block. Since $\lnk'_2\red{\umax(B'_2)}=\eta\red{\umax(B'_2)}\sse\th$, it follows that $\umax(\rel''[C])$ is contained in a $\th$-block, and therefore $\gm\sse\th$.
\end{proof}

Next, we assume that $\relo$ is a congruence block. Note that in this case polynomial closedness is trivial and the remaining conditions of Lemma~\ref{lem:affine-link} are assumed to be true.

\begin{lemma}\label{lem:etaQ}
Let in the notation of Lemma~\ref{lem:affine-link} $\relo\sse\relov$ be a block of $\xi'=\xi\red{\relov}$ with $\xi\in\Con(\rel)$ and $\xi\le\ov\beta$, and such that there is $\ba\in\relo\cap\amax(\relov)$ (or $\ba\in\relo\cap\max(\relov)$) and $\ba[1]\in C$ (in particular $\relo\in\amax(\relov\fac{\xi'})\sse\umax(\relov\fac{\xi'})$ or $\relo\in\max(\relov\fac{\xi'})\sse\umax(\relov\fac{\xi'})$, respectively). Then the conclusion of Lemma~\ref{lem:affine-link} holds.
\end{lemma}

\begin{proof}
In the proof of this lemma we use all the conventions outlined before Lemma~\ref{lem:star-prime} including $\al=\zz_1$. Let $B^C_2=\rel''[C]$ and $B''_2=\umax(B'_2)\cap B^C_2$. Also, recall that $E^C_2=\relo[C]$ and $E''_2=\umax(E_2)\cap E^C_2$. As is easily seen, $\umax(E_2)\sse\umax(B'_2)$. Indeed, since $\relo$ is a $\xi'$-block u-maximal in $\relov\fac{\xi'}$, by Lemma~\ref{lem:factor-as}(2) $\relo$ contains elements from $\umax(\relov)$. By Lemma~\ref{lem:u-max-congruence} $\umax(\relo)\sse\umax(\relov)$; therefore, by the Maximality Lemma~\ref{lem:to-max}(6) $\umax(E_2)\sse\umax(B'_2)$. Also, $E^C_2\sse B^C_2$ implying $E''_2\sse B''_2$. We also observe that $\pr_1(\relo''\cap(C\tm E''_2))=C$, where $\relo''=\pr_{12}\relo$, by the Maximality Lemma~\ref{lem:to-max}(5,6). 
%% Since $C$ is an as-component of $E_1$, by the Maximality Lemma~\ref{lem:to-max}(5) for any $a\in C$ there is $c\in\umax(E_2)$ with $(a,c)\in\relo''$. Also, $c\in E^C_2$ implying $c\in E''_2$.

By Lemma~\ref{lem:star-prime} Lemma~\ref{lem:affine-link} holds for $\relov$ in place of $\relo$. Suppose first that option (2) of the lemma holds for $\relov$. Then $\rel''=\pr_{12}\relov$, and therefore $\relo''$, as well, are the graphs of mappings, and we only need to verify that there are congruences $\eta,\gm$ satisfying the desired conditions. Let $\eta'$ be a maximal congruence such that $\eta'\red{B''_2}\sse\lnk'_2$, and $\eta'\le\beta_2$ satisfying the conditions of item (2) of Lemma~\ref{lem:affine-link} for $\rel'$. Also let $\gm'=\Cgg{\zA_2}{B''_2}\join\eta'$. As $E''_2\sse B''_2$, and $\lnk^\relo_2=\lnk'_2\cap E^2_2$, we have $\eta'\red{E''_2}\sse\lnk^\relo_2$. Therefore for $\Lambda=\{\th\in\Con(\zA_2)\mid \th\red{E''_2}\sse\lnk^\relo_2,\th\le\beta_2\}$, $\Lambda'=\{\th\in\Con(\zA_2)\mid \th\red{B''_2}\sse\lnk'_2,\th\le\beta_2\}$, we have $\Lambda'\sse\Lambda$. Hence, a congruence $\eta$ maximal with condition $\eta\red{E''_2}\sse\lnk^\relo_2$ and $\eta\le\beta_2$ can be chosen such that $\eta'\le\eta$. Observe that in fact $\eta<\beta_2$. Indeed, $\eta\red{E''_2}\sse\lnk^\relo_2$. As was observed in the beginning of the proof, $\pr_1(\relo''\cap(C\tm E''_2))=C$, $|C|>1$; and, as $\lnk^\relo_1\sse\lnk'_1$ and $\lnk'_1\red C$ is the equality relation, $E''_2$ is not contained in a $\eta$-block. Let also $\gm=\Cgg{\zA_2}{E''_2}\join\eta$. Since option (2) holds for $\relov$, $\eta'\prec\gm'$. Let $\th\in\Con(\zA_2)$ be such that $\eta\prec\th\le\gm$. As $\eta'\prec\gm'$, it also holds that $\gm'=\Cgg{\zA_2}{(c,d)}\join\eta'$ for any $(c,d)\in(E''_2)^2\cap(\th-\eta)$ (recall that $E''_2\sse B''_2$ and therefore $(c,d)\in\gm'$). Therefore $\gm'\le\th$. By the definition of $\gm'$ it holds that $B''_2$ is in a single $\th$-block, hence, $\gm\le\th$, implying $\eta\prec\gm$. Also, $\gm'\meet\eta=\eta'$, because $\eta'\le\gm'\meet\eta$ and $\eta'\prec\gm'$, and $\gm'\join\eta=\gm$, because $E''_2\sse B''_2$ and $\gm'\le\th$ for any $\th$ with $\eta\prec\th\le\gm$, as observed above. Therefore, the intervals $(\eta',\gm')$ and $(\eta,\gm)$ are perspective. By Lemma~\ref{lem:e-related} these intervals cannot be separated. Since $(\zz_1,\beta_1)$ and $(\eta',\gm')$ cannot be separated, it implies that $(\zz_1,\beta_1)$ and $(\eta,\gm)$ cannot be separated as well. 

To prove the last claim of Lemma~\ref{lem:affine-link} let $f$ be a $\ov B$-preserving polynomial of $\rel$ such that $f(\beta_1\red C)\not\sse\zz_1$, but $f(\gm\red{B'_2})\sse\eta$. Then, as $\gm=\gm'\join\eta$, we have $f((\gm'\join\eta)\red{B'_2})\sse\eta$, implying $f(\gm'\red{B'_2})\sse\eta$. Also, $f(\gm')\sse\gm'$ and $\eta'=\gm'\meet\eta$, hence $f(\gm'\red{B'_2})\sse\eta'$, a contradiction with the assumption that option (2) of the Lemma holds for $\rel'$. 

So, suppose that $C\tm \umax(\rel''[C])\sse\rel''$ (which means in particular that $C$ is a subset of a $\lnk'_1$-block) and pick an arbitrary maximal (under inclusion) congruence $\eta$ of $\zA_2$ with $\eta\red{E''_2}\sse\lnk^\relo_2$. If $\eta=\beta_2$, we are done. Indeed, in this case $\lnk^Q_2\red{E''_2}=(E''_2)^2$, and, as was observed in the beginning of the proof $\pr_1(\relo''\cap(C\tm E''_2))=C$, implying that $C$ is contained in a single $\lnk^\relo_1$-block. Since $C$ is an as-component of $E_1$ and therefore of the $\lnk^\relo_1$-block it belongs to, we have $C\tm\umax(E^C_2)\sse\relo''$ by Proposition~\ref{pro:umax-rectangular}. Thus, assume $\eta\ne\beta_2$ and $C$ is not a subset of a $\lnk^Q_1$-block.  

\medskip

{\sc Claim 1.}
For any $\th\in\Con(\zA_2)$, $\eta\prec\th\le\beta_2$ the intervals $(\zz_1,\beta_1)$ and $(\eta,\th)$ cannot be separated, and for any $\ov B$-preserving polynomial $f$ such that $f(\beta_1\red C)\not\sse\zz_1$ it also holds that $f(\th\red{B'_2})\not\sse\eta$.

\medskip

%% Let $\th\in\Con(\zA_2)$, $\eta\prec\th\le\beta_2$. 
Suppose that for some $\th\in\Con(\zA_2)$ with $\eta\prec\th\le\beta_2$ 
Claim~1 is not true.
%% $(\zz_1,\beta_1)$ and $(\eta,\th)$ can be separated. 
We prove that in this case $\eta\red{E''_2}\not\sse\lnk^\relo_2$, a contradiction with the choice of $\eta$.
%% $C\tm\umax(E^C_2)\sse\relo''$, a contradiction with our assumption that $C$ is not a subset of a $\lnk^\relo_1$-block.

By the choice of $\eta$ we have $\th\red{E''_2}\not\sse\lnk^\relo_2$. Therefore, there are $c,d\in E''_2$ with $(c,d)\in\th-\lnk^\relo_2$. Let $a,b\in E_1$ be such that $a,b\in C$ and $(a,c,e),(b,d,e')\in\relo$ for some $e,e'\in B'_3$. First, we show that there are tuples $(a,c_0,e_0),(b,d_0,e'_0)\in\rel'$ with $c_0\eqc\eta d_0$ if any of the statements of Claim~1 is false.
%% By Corollary~\ref{cor:max-min-set} $\{a,b\}$ is a $(\zz_1,\beta_1)$-subtrace. 
As $a,b\in\amax(E_1)$, if $(\zz_1,\beta_1)$ can be separated from $(\eta,\th)$ or vice versa, by Theorem~\ref{the:collapsing}(d), there is a $\ov B$-preserving idempotent polynomial $g$ of $\rel$ with $g(a)=a$, $g(b)=b$, $g(c)\eqc\eta g(d)$. The tuples $(a,g(c),g(e)),(b,g(d),g(e'))$ are as required. 
If $(\zz_1,\beta_1)$ and $(\eta,\th)$ cannot be separated, but there is a $\ov B$-preserving polynomial $f$ such that $f(\beta_1\red C)\not\sse\zz_1$ but $f(\th\red{B'_2})\sse\eta$, we obtain $f(a'')=a'\ne b'=f(b'')$ for some $a'',b''\in C$. By chaining and Lemma~\ref{lem:good-polys}(3) there are $\ov B$-preserving polynomials $f'$ and $f''$ with $f'(a)=a'', f'(b)=b''$ and  $f''(a')=a,f''(b')=b$. Then we set $g=f''\circ f\circ f'$; since $f'(c),f'(d)\in B'_2$, it holds that $g(c)\eqc\eta g(d)$, and again the tuples $(a,g(c),g(e)),(b,g(d),g(e'))$ are as required. 

Although this looks like an indication that $\eta\red{B'_2}\not\sse\lnk^\relo_2$, the tuples $(a,c_0,e_0),(b,d_0,e'_0)$ do not necessarily belong to $\relo$ and $c_0,d_0$ do not necessarily belong to $E''_2$; we have to make further steps.

As $\rel$ is chained, for any $a',b'\in C$ there is a $\ov B$-preserving polynomial $h$ with $h(a)=a', h(b)=b'$. Note also $(h(a),h(c_0),h(e_0))\eqc\xi(h(b),h(d_0),h(e'_0))$. Consider the subalgebra $\rela$ of the product $\relov\fac\xi\tm B'_2\fac\eta\tm B'_1$ given by 
\[
\rela=\{((x,y,z)\fac\xi,y\fac\eta,x)\mid (x,y,z)\in\relov\},
\]
as a subdirect product of $\pr_{12}\rela$ and $B'_1$. By what is shown above, for any $a',b'\in C$ there is $(\relp,c')\in\pr_{12}\rela$ such that $(\relp,c',a'),(\relp,c',b')\in\rela$. Therefore $C$ is a subset of a block of the link congruence of $\rela$, and by Proposition~\ref{pro:umax-rectangular} for any $(\relp,c')\in\umax(\rela^{-1}[C])$ ($(\relp,c')\in\amax(\rela^{-1}[C])$, $(\relp,c')\in\max(\rela^{-1}[C])$) it holds that $\{(\relp,c')\}\tm C\sse\rela$. 
Recall that there is $\ba\in\relo\cap\amax(\relov)$ (or $\ba\in\relo\cap\max(\relov)$) such that $\ba[1]\in C$ and by Observation~\ref{obs:congruence-lemma}(b) this tuple belongs to $\amax(\relov^C)$ (or $\max(\relov^C)$), $\rel'^C=\{\bb\mid \bb\in\relov,\bb[1]\in C\}$. The tuple $\ba$ is represented by $\ba'=(\relo,\ba[2]\fac\eta,\ba[1])$ in $\rela$, and by Lemmas~\ref{lem:to-max}(6) and~\ref{lem:factor-as}(1) $\ba'$ is as-maximal (or maximal) in $\rela^C$, $\rela^C=\{\bb\mid\bb\in\rela, \bb[3]\in C\}$, which is the set of representations of tuples from $\rel'^C$. Therefore $(\relo,\ba[2]\fac\eta)\in\amax(\rela^{-1}[C])$ (or $(\relo,\ba[2]\fac\eta)\in\max(\rela^{-1}[C])$). By Proposition~\ref{pro:umax-rectangular} this means that there is a $\eta$-block $D\in\amax(B'_2\fac\eta)$ (or $D\in\max(B'_2\fac\eta)$) such that for any $a'\in C$ there are $c'\in D$ and $e_1\in B'_3$ such that $(a',c',e_1)\in\relo$. Observe that $D$ can also be chosen such that $D\in\amax(E_2\fac\eta)$ (or $D\in\max(E_2\fac\eta)$). Indeed, $D\in E_2\fac\eta$. Let $D^\dg\in\amax(E_2\fac\eta)$ (or $D^\dg\in\max(E_2\fac\eta)$) be such that $D\sqq^{as}D^\dg$ (or $D\sqq D^\dg$) in $E_2\fac\eta$. Then $(\relo,D^\dg)\in\amax(\rela^{-1}[C])$ (or $(\relo,D^\dg)\in\max(\rela^{-1}[C])$) and we can apply Proposition~\ref{pro:umax-rectangular}.

It remains to show that for any $a'\in C$ the element $c'$ above can be chosen from $E''_2$. Indeed, this would mean that $\eta\red{E''_2}\not\sse\lnk^\relo_2$, a contradiction with the choice of $\eta$. Obviously, $c'\in E^C_2$, it remains to show it can be chosen from $\umax(E_2)$. Consider the algebra $\relo^D=\relo''\cap(E_1\tm D)$ and let $D'=\pr_2\relo^D$. By what is proved above $C\sse\pr_1\relo^D$. Also, as $D\in\amax(E_2\fac\eta)$ (or $D\in\max(E_2\fac\eta)$), $D'$ contains elements u-maximal in $E_2$. Hence $\umax(D')\sse\umax(E_2)$ by Lemma~\ref{lem:u-max-congruence}. As $C\sse\pr_1\relo^D$, every $a'\in C$ is as-maximal in $\pr_1\relo^D$. By the Maximality Lemma~\ref{lem:to-max}(5) there is $c'\in\umax(D')$ with $(a',c')\in\relo''$, as required. Claim~1 is proved.

\medskip

Since $C\tm\umax(\rel''[C])\sse\rel''$ and $(\zz_1,\beta_1)$, $(\eta,\th)$ cannot be separated, by Lemmas~\ref{lem:34-links},\ref{lem:type-equal} $\typ(\zz_1,\beta_1)=\typ(\eta,\th)=\two$, for any $\eta\prec\th\le\beta_2$. Therefore, the $\beta_1$-block $B_1$ is a module, and so is $B'_1$. Hence, $B'_1=C=E_1$. Also, this implies that $E^C_2\fac{\lnk^\relo_2}=E_2\fac{\lnk^\relo_2}$ and it is a module, as $E_1\fac{\lnk^\relo_1}\cong E_2\fac{\lnk^\relo_2}$, and every $\th$-block is a module modulo $\eta$. 

\medskip

{\sc Claim 2.}
$\lnk^Q_2=\eta\red{E_2}$.

\medskip

Suppose $\lnk^\relo_2\ne\eta\red{E_2}$, then either $\lnk^\relo_2<\eta\red{E_2}$ or there is $(c_1,d_1)\in\lnk^\relo_2-\eta$. In the former case there is a $\eta\red{E_2}$-block $F$ that contains more than one $\lnk^\relo_2$-blocks. Let $c,d\in F$ be such that $(a,c),(b,d)\in\pr_{12}\relo$ and $(a,b)\not\in\lnk^\relo_1$. Since $E_1=C$ and $E_2\fac{\lnk^\relo_2}$ are modules, every $\lnk^\relo_2$-block contains a u-maximal element of $E_2=E^C_2$. Indeed, every such block $F'$ is as-maximal in $E_2\fac{\lnk^\relo_2}$ by Lemma~\ref{lem:as-type-2}, therefore, it is u-maximal in $E_2\fac{\lnk^\relo_2}$ by Corollary~\ref{cor:amax-umax}, and so $F'\cap\umax(E_2)\ne\eps$ by Lemma~\ref{lem:factor-as}(2). Hence, $c,d$ can be chosen to be u-maximal, a contradiction with the condition $\eta\red{E''_2}\sse\lnk^\relo_2$. 

In the latter case the elements $c_1,d_1$ can be chosen such that there is $a_0\in E_1$ and $e_1,e'_1\in B'_3$ with $(a_0,c_1,e_1),(a_0,d_1,e'_1)\in\relo$. Let  $\vr=\Cgg{\zA_2}{(c_1,d_1)}\join\eta$. By the choice of $\eta$, $\vr\red{E''_2}\not\sse\lnk^\relo_2$. Let $D$ be a $\vr$-block that intersects more than one $\lnk^\relo_2$-blocks on $E''_2$ and $D'=D\cap B'_2$. Note that $D'\fac\eta$ contains elements u-maximal in $B'_2\fac\eta$, because $D'\fac\eta$ contains elements u-maximal in $E''_2\fac\eta$ by the choice of $D$, $E''_2$ contains u-maximal elements from $B'_2$ by the conditions on $\relo$, and $\umax(E''_2)\sse\umax(B'_2)$ by Lemma~\ref{lem:u-max-congruence}, which implies $\umax(E''_2\fac\eta)\sse\umax(B'_2\fac\eta)$ by Lemma~\ref{lem:factor-as}(1).
Since $\rel$ is chained, we can apply Lemma~\ref{lem:good-polys}(4) with the roles of the parameters $I,\al,\beta,\ba,\bb,E,E'$ played by $\{2\},\eta,\vr,c_1\fac\eta,d_1\fac\eta, D\fac\eta$, and $D'\fac\eta$ respectively. Then, for any $\ov c_2,\ov d_2\in\umax(D'\fac\eta)$ there are $\ov c_2=\ov c^*_1\zd\ov c^*_k=\ov d_2$ in $D'\fac\eta$ and $\ov B$-preserving polynomials $\vc f{k-1}$ with $\{f_i(c_1),f_i(d_1)\}=\{\ov c^*_i,\ov c^*_{i+1}\}$. Note that the triples $(f_i(a_0),f_i(c_1),f_i(e_1)),(f_i(a_0),f_i(d_1),f_i(e'_1))$ are related in $\xi$. 
%% Fix $\ov c_2\in\umax(D\fac\eta)\sse\umax(B'_2\fac\eta)$.

Let $\tau=\xi\meet(\zz_1\tm\eta\tm\zo_3)$. We consider $\rel'\fac\tau$ as a subdirect product $\rel^\dg$ of $\rel'\fac\xi\tm B'_1\tm B'_2\fac\eta$. Let $\lnk^\eta_3$ be the link congruence of the third coordinate of $\rel^\dg$. By what is proved above $\umax(D'\fac\eta)$ is contained in a $\lnk^\eta_3$-block. By Proposition~\ref{pro:umax-rectangular} for every as-component $C^\dg$ of $D'\fac\eta$ there are $b_0\in B'_1$ and a $\xi\red{\relov}$-block $\reli$ such that for any $\eta$-block $D_1\in C^\dg$ there is $d\in D_1$ and $e'_2\in B'_3$ with $(b_0,d,e'_2)\in\reli$. Pick an as-component $C^\dg\sse\umax(D'\fac\eta)\sse\umax(B'_2\fac\eta)$ of $D'\fac\eta$, a $\xi$-block $\reli$, $b_0\in B'_1$, $\ov c_2\in C^\dg$, $c_2\in\ov c_2$, and $e_2\in B'_3$ such that $(b_0,c_2,e_2)\in\reli$. 

It seems that $b_0$ is related to multiple $\eta$-blocks across $\lnk^\relo_2$ in contradiction with the choice of $\eta$. However, two obstacles remain. First, while $D'$ intersects multiple $\lnk^\relo_2$-blocks, the connection between $C^\dg$ and $E_2$ is unknown, yielding no contradiction. It may even be the case that $|C^\dg|=1$. Second, $\reli$ is not necessarily equal to $\relo$, again avoiding a contradiction. 

To overcome these obstacles, choose $\th\in\Con(\zA_2)$ such that $\eta\prec\th\le\vr$. By Claim~1 $(\zz_1,\beta_1)$ and $(\eta,\th)$ cannot be separated and by Lemma~\ref{lem:type-equal} $\typ(\eta,\th)=\two$. Moreover, since $B'_1\tm\umax(B'_2)\sse\rel''$, by Lemma~\ref{lem:delta-alignment}, $B'_2$ is contained in a single  $(\eta:\th)$-block, as $B'_1\tm\umax(B'_2)\sse\rel''$ means that the coordinates 1 and 2 cannot be $\beta_1\gm$-aligned in $\rel$ for any $\eta\le\gm\le\beta_2$ such that $\gm$ is nontrivial on $B'_2$. By the choice of $\eta$ there is a $\th$-block $F$ and $c_3,d_3\in F'\cap\umax(E_2)$, where $F'=F\cap B'_2$, with $(c_3,d_3)\not\in\lnk^\relo_2$, that is, $|(F'\cap\umax(E_2))\fac\eta|>1$. 

There exist $b\in B'_1$ and $e_3\in B'_3$ such that $(b,c_3,e_3)\in\umax(\relo)\sse\umax(\relov)$. Take an asm-path $(b_0,c_2,e_2)=\ba_1\zd\ba_\ell=(b,c_3,e_3)$, where $\ba_i=(b_i,c^\dg_i,e^\dg_i)$ in $\relov$. For $i\in[\ell-1]$ let $f_i$ be a polynomial of $\rel$ given in Lemma~\ref{lem:central-mapping} and defined as follows: If $\ba_i\ba_{i+1}$ is semilattice (majority, affine), then
\[
f_i(x)=x\cdot\ba_{i+1},\quad f_i(x)=t_{\ba_i\ba_{i+1}}(x,\ba_{i+1}),\quad f_i(x)=h_{\ba_i\ba_{i+1}}(x,\ba_i,\ba_{i+1}),
\]
respectively, where $t_{\ba_i\ba_{i+1}},h_{\ba_i\ba_{i+1}}$ are polynomials satisfying the conclusion of Lemma~\ref{lem:op-s-on-affine}. As is easily seen, $f_i(\ba_i)=\ba_{i+1}$, $i\in[\ell-1]$. We will apply Lemma~\ref{lem:central-mapping} to congruences $\eta\prec\th$ and $\th$-blocks $c^\dg_i\fac\th$. Observe that either $c^\dg_i=c^\dg_{i+1}$ or $c^\dg_i c^\dg_{i+1}$ is a thin edge of the same type as $\ba_i\ba_{i+1}$. If $\ba_i\ba_{i+1}$ is majority or affine then $t_{\ba_i\ba_{i+1}},h_{\ba_i\ba_{i+1}}$ satisfy the conditions of Lemma~\ref{lem:op-s-on-affine}(1,2) for the pair $(c^\dg_i,c^\dg_{i+1})$. Note also that by Lemma~\ref{lem:thin-properties} if $c^\dg_i\fac\th\ne c^\dg_{i+1}\fac\th$ then $c^\dg_i\fac\th c^\dg_{i+1}\fac\th$ is a thin edge of the same type as $c^\dg_i c^\dg_{i+1}$ and $\ba_i\ba_{i+1}$. Thus we can apply Lemma~\ref{lem:central-mapping} to $(c^\dg_i,c^\dg_{i+1})$ and $\eta\prec\th$ using $t_{\ba_i\ba_{i+1}}$ for $t_{c^\dg_i c^\dg_{i+1}}$ ($h_{\ba_i\ba_{i+1}}$ for $h_{c^\dg_i c^\dg_{i+1}}$) in the majority (affine) case, and the polynomials produced are the second components of $\vc f\ell$.
Let $G$ be the $\th$-block containing $c_2$ and $G'\sse D'$, $G'=G\cap B'_2$; note that by the observation before Claim~2 $G'\fac\eta$ is a module, and therefore $G'\fac\eta\sse C^\dg$. Composing all the polynomials $f_i$ by Lemma~\ref{lem:central-mapping} we obtain a polynomial $f$ such that $f(b_0,c_2,e_2)=(b,c_3,e_3)$ and, as $B'_2$ is contained in a $(\eta:\th)$-block, $f$ maps $G'\fac\eta$ injectively into $F'\fac\eta$. Since $\{b_0\}\tm G'\fac\eta\sse(\pr_{12}\reli)\fac{\zz_1\tm\eta}$, we obtain $\{b\}\tm f(G'\fac\eta)\sse \relo''\fac{\zz_1\tm\eta}$. It remains to show that $f(G'\fac\eta)$ is sufficiently large.

As $\ov c_2\in\umax(B'_2\fac\eta)$ (recall that $D'$ contains u-maximal elements of $B'_2$ and therefore every its u-maximal element belongs to $\umax(B'_2)$), there is $c'_2\in G'\cap\umax(B'_2)$.
Using an asm-path from $c_3$ to $c'_2$ by Lemma~\ref{lem:central-mapping} we can show that there is also an injective mapping from $F'\fac\eta$ to $G'\fac\eta$. Therefore, $|G'\fac\eta|=|F'\fac\eta|$, and so $f(G'\fac\eta)=F'\fac\eta$. As we proved above, it means that for any $\eta$-block $F_1$ with $F_1\cap B'_2\sse F'$ there is $d\in F_1\cap B'_2$ with $(b,d)\in\relo''$, a contradiction with the assumption that $F'$ contains elements from several $\lnk^\relo_2$-blocks.  Claim~2 is proved.

\medskip

It remains to prove that $\eta\prec\Cgg{\zA_2}{E''_2}\join\eta$ and that $\relo''$ is the graph of a mapping $E_2\to B'_1$. Let $\eta^*=\Cgg{\zA_2}{\eta\red{E_2}}$. Replacing $\rel$ with $\rel\fac{\zz_1\tm\eta^*\tm\zz_3}$ we may assume that $\eta\red{E_2}$ is the equality relation. Then making use of Claim~2 assume that $\lnk^\relo_2=\zz_2\red{E_2}$. 

{\sc Claim 3.} $\lnk^\relo_1$ is either the equality relation or the full relation.

\medskip

By what has been proved so far $B'_1=E_1$ is a module.
Suppose that $\lnk^\relo_1\ne\zz\red{E_1}$. Then there are $a,b\in E_1$, $a\ne b$, and $c\in E_2$, $e_1,e_2\in B'_3$ such that $(a,c,e_1),(b,c,e_2)\in\relo$. Moreover, we can assume that $\ba=(a,c,e_1)\in\umax(\relov)$. Indeed, $\lnk^\relo_2=\zz_2\red{E_2}$ implies that $\relo''$ is a module, and therefore $\relo''=\amax(\relo'')$. Then by the Maximality Lemma~\ref{lem:to-max}(5) $(a,c),(b,c)$ can be extended to tuples from $\amax(\relo)\sse\umax(\relov)$. Since $\rel$ is chained, for any $b'\in E_1=B'_1$ there exists a $\ov B$-preserving polynomial $f$ such that $f(a)=a$ and $f(b)=b'$. Then $((a,d,e'_1),(b',d,e'_2))\in\xi$, where $\cll ad{e'_1}=f\cll ac{e_1}, \cll{b'}d{e'_2}=f\cll bc{e_2}$. 
%% Since $a,b'\in\umax(B'_1)$, $d$ can be chosen from $\umax(B'_2)$. Indeed, let $d'\in B'_2$ be a maximal element with $d\sqq_{B'_2}d'$ and $d=d_1\le\dots\le d_\ell=d'$ an s-path. This path can be extended to an s-path $\bd_1\le\dots\le\bd_\ell$ in $\relov$, $\bd_i=(u_i,d_i,v_i)$, $\bd_1=(a,d,e'_1)$. We then replace the polynomial $f$ with
%% \[
%% f'(x)=(\dots((f(x)\cdot\bd_1)\cdot\bd_2)\dots\bd_{\ell-1})\cdot\bd_\ell.
%% \]
%% Since $B'_1$ is a module by Proposition~\ref{pro:very-good-operation} $f(x)=x$ for any $x\in B'_1$. Also, $f(d)=d'$. Thus $f'\cll ac{e_1}=\cll a{d'}{e''_1}, f'\cll bc{e_2}=\cll{b'}{d'}{e''_2}$.
We will argue that $(a,b')\in\lnk^\relo_1$, implying that $\lnk^\relo_1$ is the full relation. 

We again apply the method used previously. Consider the subalgebra of $\rel'\fac\xi\tm B'_2\fac\eta\tm B'_1$
\[
\rela=\{((x,y,z)\fac\xi,y\fac\eta,x)\mid (x,y,z)\in\relov\}
\]
as a subdirect product of $\pr_{12}\rela$ and $B'_1$. By what is proved above for any $b'\in B'_1$ there is $(P,\ov c)\in\pr_{12}\rela$ such that $(P,\ov c,a),(P,\ov c,b')\in\rela$. Therefore $\rela$ is linked, and by Proposition~\ref{pro:umax-rectangular} $\umax(\pr_{12}\rela)\tm B'_1\sse\rela$ (recall that $B'_1$ is a module and therefore an as-component). Similar to the proof of Claim~1 recall that there is $\ba\in\relo\cap\amax(\relov)$ (or $\ba\in\relo\cap\max(\relov)$). The condition $\ba[1]\in C$ is trivial as $B'_1=C$.  The tuple $\ba$ is represented by $\ba'=(\relo,\ba[2]\fac\eta,\ba[1])$ in $\rela$, and by Lemmas~\ref{lem:to-max}(6) and~\ref{lem:factor-as}(1) $\ba'$ is as-maximal (or maximal) in $\rela$. Therefore $(\relo,\ba[2]\fac\eta)\in\umax(\pr_{12}\rela)$, and so $\{(\relo,\ba[2]\fac\eta)\}\tm B'_1\sse \rela$. This means that for any $b'\in B'_1$ there is $c'\in\ba[2]\fac\eta$ such that $(b',c',e')\in\relo$ for some $e'\in B'_3$. It remains to observe that, as $\lnk^\relo_2=\zz_2\red{E_2}$, and $c'\in E_2$, there is $c'\in E_2$ such that $(b',c')\in\relo''$ for all $b'\in B'_1$. Claim~3 is proved. 
%% 
%% Let $\bb=(a,d,e'_1)$. As $\ba\in\umax(\relov)$ and $B'_1\tm\umax(B'_2)\sse\rel''$, there is an asm-path $\bb=\ba_1\zd\ba_\ell=\ba$ in $\relov$, where $\ba_i=(a,c_i,e^*_i)$. We now use polynomials similar to those constructed in Lemma~\ref{lem:central-mapping}. For $i\in[\ell-1]$ let $g_i$ be a polynomial of $\rel$ defined as follows: If $\ba_i\ba_{i+1}$ is semilattice (majority, affine), then
%% \[
%% g_i(x)=x\cdot\ba_{i+1},\quad g_i(x)=t_{\ba_i\ba_{i+1}}(x,\ba_{i+1}),\quad g_i(x)=h_{\ba_i\ba_{i+1}}(x,\ba_i,\ba_{i+1}),
%% \]
%% respectively. As is easily seen, $g_i(\ba_i)=\ba_{i+1}$. Also, if $(b',d,e'_2)=\ba'_1\zd\ba'_\ell$, where $\ba'_{i+1}=g_i(\ba'_i)$ for $i\in[\ell-1]$, then $(\ba_i,\ba'_i)\in\xi$ for all $i\in[\ell]$, because $(\ba_1,\ba'_1)\in\xi$ by construction. Moreover, $\ba'_i=(b',c_i,e'^*_i)$ for some $e'^*_i\in B'_3$. Therefore $(b',c)\in\relo''$, implying $(a,b')\in\lnk^\relo_1$. Claim~3 is proved.

\medskip

Recall that we assume that $\lnk^\relo_1$ is not the full relation, because otherwise $C$ is a subset of an $\lnk^\relo_1$-block, and we are done as observed in the beginning of the proof of Lemma~\ref{lem:etaQ}. Also we assume $\eta\red{E_2}=\lnk^\relo_2$ to be the equality relation. This however does not imply that $\eta$ itself or even its restriction to $B'_2$ are the equality relations. Since $(\zz_1,\beta_1)$ and $(\eta,\th)$ cannot be separated for any $\th\in\Con(\zA_2)$ with $\eta\prec\th\le\beta_2$, by Claim~2 it remains to prove that $\eta\prec\gm$, where $\gm=\Cgg{\zA_2}{E''_2}\join\eta$ is the congruence from Lemma~\ref{lem:affine-link}. By the choice of $\eta$ as a maximal congruence with $\eta\red{E''_2}\sse\lnk^\relo_2$ for any $\eta<\th$ there is $(c,d)\in (E''_2)^2$ such that $(c,d)\in\th\red{E''_2}-\eta\red{E''_2}$. Thus, it suffices to prove that for any $c,d,d'\in E_2$ it holds that $(c,d')\in\th=\Cgg{\zA_2}{(c,d)}\join\eta$. 
%% By Claim~2 $E_2\fac\eta$ is a module. 
%% Take any $c,d,d'\in E^C_2=E_2$, $c\ne d$. We will show that $(c,d')\in\th=\Cgg{\zA_2}{(c,d)}\join\eta$, which, as $c,d,d'$ are arbitrary, proves that $\eta\prec\gm$. 
Since $E_2$ is a module, $\umax(E_2)=E_2$. Let $a,b,b'\in E_1$ and $e_1,e_2\in B'_3$ be such that $(a,c,e_1),(b,d,e_2)\in\relo$. 
%% As $c\in\umax(E_2)\sse\umax(B^C_2)$ and $C\tm\umax(B'_2)\sse\rel''$, the triple $(a,c,e_1)$ can be chosen to be u-maximal in $\relov$.

Since $c\ne d$, it also holds that $a\ne b$. As $\rel$ is chained, there is a $\ov B$-preserving polynomial $f$ with $f(a)=a,f(b)=b'$. Note that $(a,f(c),f(e_1))\eqc\xi(b',f(d),f(e_2))$ and $f(c)\eqc\th f(d)$. Consider the subalgebra of $\rel'\fac\xi\tm B'_2\fac\th\tm B'_1$ given by
\[
\rela=\{((x,y,z)\fac\xi,y\fac\th,x)\mid (x,y,z)\in\relov\}
\]
as a subdirect product of $\pr_{12}\rela$ and $B'_1$. By what is proved above for any $b'\in B'_1$ there is $(P,\ov c)\in\pr_{12}\rela$ such that $(P,\ov c,a),(P,\ov c,b')\in\rela$. Therefore $\rela$ is linked, and by Proposition~\ref{pro:umax-rectangular} $\umax(\pr_{12}\rela)\tm B'_1\sse\rela$. As in the end of the proof of Claim~3 $(\relo,\ov c)\in\umax(\pr_{12}\rela)$ for some $\ov c\in B'_2\fac\th$, and so $\{(\relo,\ov c)\}\tm B'_1\sse \rela$. This means that for any $b'\in B'_1$ there is $c'\in\ov c$ such that $(b',c',e')\in\relo$ for some $e'\in B'_3$. Since $\lnk^\relo_1$ and $\lnk^\relo_2$ are equality relations on $E_1,E_2$ respectively, $E_2\sse \ov c$, implying $\gm\le\th$.
%% 
%% Let $f(c)=c'', f(d)=d'', f(e_1)=e'_1, f(e_2)=e'_2, f(e_3)=e'_3$. By the choice of $(a,c,e_1)$ we have $(a,c'',e'_1)\sqq^{asm}_{\relov}(a,c,e_1)$. As in the proof of Claim~3 using again the polynomials introduced in Lemma~\ref{lem:central-mapping} there is a $\ov B$-preserving polynomial $g$ such that $g(a)=a, g(c'')=c, g(e'_1)=e_1$, and $g(b')=b'$. Since $(a,c,e_1)\eqc\xi(g(b'),g(d''),g(e'_2))$, it holds that $(b',g(d''))\in\relo''$ and so $g(d'')= d'$. Therefore $c=gf(c)\eqc\th gf(d)= d'$.
\end{proof}

\begin{proof}[Proof of Lemma~\ref{lem:affine-link}]
Recall that unlike the two previous lemmas we cannot assume $\al=\zz_1$ at the moment. Let again $\rel''=\pr_{I'}\relov$ and $\relo''=\pr_{I'}\relo$. Note that if $|C\fac\al|=1$, the lemma is trivially true, because $E^C_2=\relo''[C\fac\al]$, so (1) holds. 
%% Similar to Lemma~\ref{lem:type23} it happens when $\typ(\al,\beta_1)\in\{\four,\five\}$. Indeed, 

By Lemma~\ref{lem:etaQ} the statement of Lemma~\ref{lem:affine-link} holds for $\rela=\Bl(\relo)$. Suppose first that option (1) of the lemma holds. Then, as $\relo$ is polynomially closed, for any $(a,c,e)\in\umax(\relo)$, $a\in C$, and any $(b,d,e')\in\rela$ with $(a,c,e)\sqq^{as}(b,d,e')$ in $\rela$, we have $(b,d,e')\in\relo$. To see that option (1) of the lemma holds for $\relo$, as well, pick $(a,c)\in\umax(\relo'')$ with $a\in C$ and $(b,d)\in\amax(\rela'')$, where $\rela''=\pr_{12}\rela$, such that $(a,c)\sqq^{as}_{\rela''}(b,d)$, clearly, $b\in C$. Then by polynomial closedness $(b,d)\in\relo''$. The same is true for the entire as-component $D$ of $\rela''$ containing $(b,d)$, it is contained in $\relo''$. Note also that $\pr_2 D\sse\amax(F^C_2)$, where $F^C_2=\rela''[C\fac\al]$. Since $\rela$ satisfies option (1), $C\fac\al\tm\umax(F^C_2)\sse\rela\fac\al$, which implies that $C\fac\al$ is contained in a block of the link congruence for $\rela''$. Thus, by Proposition~\ref{pro:umax-rectangular} $C\fac\al\tm\amax(F^C_2)\sse\rela''\fac\al$. 

Next, $C\fac\al\tm\{d\}\sse D\fac\al$. Indeed, let $a'\in C$; then $b\sqq^{as}_C a'$. Choose an as-path $b=b_0,b_1\zd b_\ell=a'$ in $C$. Then $(b\fac\al,d)=(b_0\fac\al,d)\zd(b_\ell\fac\al,d)$ is an as-path in $C\fac\al\tm\{d\}\sse\rela''\fac\al$. By Lemma~\ref{lem:factor-as}(3) there is an as-path $(b,d),(b'_1,d)\zd(b'_\ell,d)$ in $\rela''$ with $b'_i\in b_i\fac\al$ for $i\in[\ell]$. This path is in $D$ proving that $(a',d)=(b_\ell,d)\in D\fac\al$. Thus, $C\fac\al\tm\{d\}\sse D\fac\al\sse\relo''\fac\al$. Therefore, $C\fac\al$ is contained in a single block of the link congruence $\lnk^\relo_1$, and then by Proposition~\ref{pro:linkage-rectangularity}(1) $C\fac\al\tm\umax(E^C_2)\sse\relo''\fac\al$. 

We use polynomial closedness for one more property of $\relo$ (not assuming option~(1) of Lemma~\ref{lem:affine-link}): $\relo\cap\umax(\rela)\ne\eps$ and therefore $\umax(\relo)\sse\umax(\rela)$ by Lemma~\ref{lem:u-max-congruence}. Indeed, using an argument similar to the one above it is easy to see that there is an as-maximal element of $\rela$ that belongs to $\relo$.

%% pick an arbitrary maximal (under inclusion) congruence $\eta$ of $\pr_I\rel$ with $\eta\red{\umax(E^C_2)}\sse\lnk^\relo_2$
Suppose now that option (2) of the lemma holds for $\rela$. Since we are not using polynomial closedness anymore, we again assume that $\al=\zz_1$. The proof in this case is almost verbatim the argument in the beginning of the proof of Lemma~\ref{lem:etaQ}. Observe that we may assume that $(E''_2)^2\not\sse\lnk^\relo_2$, $E''_2=\umax(E_2)\cap E^C_2$. Indeed, otherwise $C$ would be a subset of a $\lnk^\relo_1$-block implying option (1) for $\relo$ and therefore for $\rela$. 

The algebra $\rela''$, and therefore $\relo''$ is the graph of a mapping, and we only need to verify that there are congruences $\eta,\gm$ that satisfy the desired conditions. Let $F_i=\pr_i\rela$, $F^C_2=\rela''[C]$, $F''_2=\umax(F_2)\cap F^C_2$, and let $\lnk^\rela_i$ be the link congruence of $F_i$ with respect to $\rela''$, $i=1,2$. Let $\eta^\rela$ be a maximal congruence such that $\eta^\rela\red{F''_2}\sse\lnk^\rela_2$ and $\eta^\rela\le\beta_2$, and such that it satisfies the conditions of item (2) of Lemma~\ref{lem:affine-link} for $\rela$. Let $\gm^\rela=\Cgg{\zA_2}{F''_2}\join\eta^\rela$. As $\umax(\relo)\sse\umax(\rela)$, by the Maximality Lemma~\ref{lem:to-max}(6) $\umax(E_2)\sse\umax(F_2)$. Combining this with $E^C_2\sse F^C_2$ we obtain $E''_2\sse F''_2$. Also, since $\lnk^\relo_2=\lnk^\rela_2\cap E^2_2$, we get $\eta^\rela\red{E''_2}\sse\lnk^\relo_2$. Therefore $\eta\in\Con(\zA_2)$ maximal with the conditions $\eta\red{E''_2}\sse\lnk^\relo_2$ and $\eta\le\beta_2$ can be chosen such that $\eta^\rela\le\eta$. 

As option (2) holds for $\rela$, $\eta^\rela\prec\gm^\rela$. Let $\th\in\Con(\zA_2)$ be such that $\eta\prec\th\le\gm$ (recall that $\gm=\Cgg{\zA_2}{E''_2}\join\eta$). As $\eta^\rela\prec\gm^\rela$, it also holds that $\gm^\rela=\Cgg{\zA_2}{(c,d)}\join\eta^\rela$ for any $(c,d)\in(E''_2)^2\cap(\th-\eta)$ (recall that $E''_2\sse F''_2$ and therefore $(c,d)\in\gm^\rela$). Thus, $\gm^\rela\le\th$.  By the definition of $\gm^\rela$ it holds that $F''_2$ is in a single $\th$-block, hence, $\gm\le\th$, implying $\eta\prec\gm$. Also, $\gm^\rela\meet\eta=\eta^\rela$, because $\eta^\rela\le\gm^\rela\meet\eta$ and $\eta^\rela\prec\gm^\rela$, and $\gm^\rela\join\eta=\gm$, because $E''_2\sse F''_2$. Therefore, the intervals $(\eta^\rela,\gm^\rela)$ and $(\eta,\gm)$ are perspective. By Lemma~\ref{lem:e-related} these intervals cannot be separated. Since $(\al,\beta_1)$ and $(\eta^\rela,\gm^\rela)$ cannot be separated, it implies that $(\al,\beta_1)$ and $(\eta,\gm)$ cannot be separated as well.

To prove the last claim of Lemma~\ref{lem:affine-link} let $f$ be a $\ov B$-preserving polynomial of $\rel$ such that $f(\beta_1\red C)\not\sse\zz_1$, but $f(\gm\red{B'_2})\sse\eta$. Then, as $\gm=\gm^\rela\join\eta$, we have $f((\gm^\rela\join\eta)\red{B'_2})\sse\eta$, implying $f(\gm^\rela\red{B'_2})\sse\eta$. Also, $f(\gm^\rela)\sse\gm^\rela$ and $\eta^\rela=\gm^\rela\meet\eta$, hence $f(\gm^\rela\red{B'_2})\sse\eta^\rela$, a contradiction with the assumption that option (2) of the Lemma holds for $\rela$. 
\end{proof}

%%%%%%%%%%%%%%%%%%%%%%%%%%%%%%%%%%%%
%%%%%%%%%%%%%%%%%%%%%%%%%%%%%%%%%%%%
\section{Chaining and maximality}\label{sec:chaining}

In this section we show that the chaining condition holds in a
fairly broad range of circumstances. In particular, it is preserved 
under certain transformations of the relation.

\begin{lemma}\label{lem:full-chaining}
Let $\rel$ be a subdirect product of $\vc\zA n\in\cV$, $\beta_i=\zo_i$, 
$B_i=\zA_i$. Then $\rel$ is chained with respect to $\ov\beta, \ov B$.
\end{lemma}

\begin{proof}
Let $I,J\sse[n]$. By Lemma~\ref{lem:separation-separation}(2) 
we may assume that $I\cap J=\eps$.
Let $\al\le\beta$, $\al,\beta\in\Con(\pr_I\rel)$, and 
$\gm\prec\dl$, $\gm,\dl\in\Con(\pr_J\rel)$, be such that if $\al\prec\beta$,
then $(\al,\beta)$ can be separated from $(\gm,\dl)$. Consider first 
condition (Q1). Since any polynomial of $\rel$ is $\ov B$-preserving,
(Q1) follows from the definitions.

For condition (Q2) let $(\ba,\bb)\in\beta-\al$, and let $f$ be a unary idempotent polynomial 
such that $f(\pr_I\rel)$ is an $(\al,\beta)$-minimal set and 
$f(\ba)\not\eqc\al f(\bb)$. By Lemma~\ref{lem:min-set-separation}(1)
$f$ can be chosen such that $f(\dl)\sse\gm$. Since 
$\beta=\Cgg{\pr_I\rel}{\{(f(\ba),f(\bb))\}}\join\al$, for any $(\bc,\bd)\in\beta$
there are polynomials $\vc fk$ such that  for some $\vc \bc k$, $\vc{\bc'}k$ with $\{f_i(f(\ba)),f_i(f(\bb))\}=\{\bc_i,\bc'_i\}$ for $i\in[k]$, $\bc'_i\eqc\al \bc_{i+1}$ for $i\in[k-1]$, $f_1(f(\ba))=\bc$, and $\bc'_k\eqc\al \bd$. Then the polynomials $f_i\circ f$, $i\in[k]$, witness that
$(\bc,\bd)\in\Cgg{\pr_I\rel,\al,\cU(\gm,\dl,\ov B)}{\{(\ba,\bb)\}}$, and, as $(\bc,\bd)$ is arbitrary, that $\beta=\Cgg{\pr_I\rel,\al,\cU(\gm,\dl,\ov B)}{\{(\ba,\bb)\}}$.
\end{proof}

\begin{theorem}\label{the:S7}
Let $\rel$ be a subdirect product of $\vc\zA n\in\cV$, $\beta_i\in\Con(\zA_i)$
and $B_i$ a $\beta_i$-block, $i\in[n]$, such that $\rel$ is chained
with respect to $\ov\beta,\ov B$. Let $\relov=\rel\cap\ov B$
and $B'_i=\pr_i\relov$. Fix $i\in[n]$, $\beta'_i\prec\beta_i$,
and let $D_i$ be a $\beta'_i$-block that is as-maximal in 
$B'_i\fac{\beta'_i}$. Let also $\beta'_j=\beta_j$ 
and $D_j=B_j$ for $j\ne i$. Then
$\rel$ is chained with respect to $\ov\beta',\ov D$.
\end{theorem}

\begin{proof}
Let $\rel_I=\pr_I\rel, \rel'_I=\pr_I\rel'$, $\rel''=\rel\cap\ov D$, $\rel''_I=\pr_I\rel''$, and $D'_j=\pr_j\rel''$, $j\in[n]$. Take $I$, $J$ 
from the definition of chaining. Let $I=[\ell]$, for condition (Q2), we assume $J=\{n-m\zd n\}$. 
Let $\al,\beta\in\Con(\rel_I)$, $\gm,\dl\in\Con(\pr_J\rel)$ be such that 
$\al\le\beta\le\ov\beta'_I$, $\gm\prec\dl\le\ov\beta'_J$ (we assume 
$\al\prec\beta$ when considering condition (Q2)). Clearly, we may 
assume $\al=\ov\zz_I$, $\gm=\ov\zz_J$, and $\beta'_i=\zz_i$. Note 
that by Lemma~\ref{lem:separation-separation}(2) 
replacing $\rel$
with the $n+1$-ary relation $\{(\ba,\ba[i])\mid \ba\in\rel\}$ we 
may assume that $i\not\in I\cup J$. Without loss of generality assume 
$i=\ell+1$. By the assumption $\beta'_{\ell+1}=\zz_{\ell+1}$, the classes of $\beta'_{\ell+1}$ are
just elements of $\zA_{\ell+1}$, so let $D'_{\ell+1}$ be denoted by $c$. Let $C$ be the 
as-component of $B'_{\ell+1}$ containing $c$. 

Let us also observe that, since $c$ is as-maximal in $B'_{\ell+1}$, the set $\rel''_I$ also contains as-maximal elements of $\rel'_I$. Therefore by Lemma~\ref{lem:u-max-congruence} $\umax(\rel''_I)\sse\umax(\rel'_I)$.

To prove the theorem let $\cU^*\in\{\cU_{\ov D},\cU(\gm,\dl,\ov D)\}$. We assume that $\beta$ is a principal congruence of $\rel_I$ such that $\ov\zz_I<\beta\le\ov\beta_I=\ov\beta'_I$. Let $\ba^0,\bb^0\in\rel''_I$ be any elements such that $\Cgg{\rel_I}{\{(\ba^0,\bb^0)\}}=\beta$. We need to prove that for any $\beta$-block $B$ such that $B\cap\umax(\rel''_I)\ne\eps$, it holds that $(\umax(B\cap\rel''_I))^2\sse\Cgg{\rel_I,\ov\zz_I,\cU^*}{\ba^0,\bb^0}$.
We use the following approach. If $\cU^*$ were one of $\cU_{\ov B},\cU(\gm,\dl,\ov B)$ by the assumption 
that $\rel$ is chained with respect to $\ov\beta,\ov B$, we would have $(\umax(B\cap\rel'_I))^2\sse\Cgg{\rel_I,\ov\zz_I,\cU^*}{\ba^0,\bb^0}$. We will start with polynomials from $\cU_{\ov B},\cU(\gm,\dl,\ov B)$ that connect elements from $\umax(B\cap\rel''_I)$ as in the condition of chaining and transform them to polynomials from $\cU^*$ in several steps. Let $\relo'=\pr_{[\ell+1]}\relov$, $B^C_I=\relo'^{-1}[C]$, and $\relo^C=\relo'\cap(B^C_I\tm C)$.
We divide the proof into two cases, depending on whether or not 
$\relo^C$ is linked. 

First, we consider 
the case when $\relo^C$ is not linked, this case is relatively easy. Let $\lnk^C_1,\lnk^C_2$ be defined for $\relo^C$ in the same way as link 
congruences are defined for a subdirect product. Note that $\relo^C$ is not necessarily a subalgebra, and $\lnk^C_1,\lnk^C_2$ are not necessarily congruences.

\medskip

{\sc Claim 1.}
Let $\relo^C$ be not linked and let $\lnk^C_1,\lnk^C_2$ be as above.
Then 
$\lnk^C_2=\zz_{\ell+1}\red C$ and $\beta\red{B^C_I}\sse\lnk^C_1$. 
%% or $(\beta\cap\lnk^C_1)\red{B^C_I}=\th\red{B^C_I}$ for 
%% some $\th\in\Con(\rel_I)$, $\ov\zz_I\le\th<\beta$.

\medskip

Firstly, note that as $(\ba^0,c),(\bb^0,c)\in\relo^C$, $\beta\red{B^C_I}\sse\lnk^C_1$. The relation $\rel'$ is an (improper) subalgebra of $\relov$ and is 
polynomially closed in $\rel$ by Lemma~\ref{lem:poly-closed}(2), because $\rel'$ is a $\ov\beta$-block.
%% of the congruence $\ov\beta\fac\eta_I$, where  
%% and the second statement of Lemma~55(3). %%\ref{lem:poly-closed}(3). 
By the Congruence Lemma~\ref{lem:affine-link} for $I_0=\{\ell+1\}$, $\al=\zz_{\ell+1}$, and $I=[\ell]$,
if $\relo^C$ is not linked then 
$\relo^C$ is the graph of a mapping $\vf:B^C_I\to C$. This means 
$\lnk^C_2=\zz_{\ell+1}\red C$. 
%% and $\lnk^C_1$ is the restriction of a congruence 
%% $\eta$ of $\rel_I$ defined in the Congruence Lemma~\ref{lem:affine-link} onto $B^C_I$. If $\beta\le\eta$, we obtain the 
%% first option of the conclusion of the claim, otherwise $\th=\beta\meet\eta$ 
%% and we have the second option.

\medskip

Note that as $\beta\red{B^C_I}\le\lnk^C_1$, any 
$\ov B$-preserving polynomial
that maps a pair of $\beta$-related elements from $\rel''_I$ on a pair
from $\rel''_I$ is also $\ov D$-preserving, because 
$\lnk^C_2=\zz_{\ell+1}\red C$. The result follows, as $\umax(\rel''_I)\sse\umax(\rel'_I)$ as is observed above and $B\cap\rel''_I= B\cap\rel'_I$. 
%% If $(\beta\meet\lnk_1)\red{B^C_I}=\th\red{B^C_I}$ for some $\th<\beta$, 
%% there is nothing to prove, because no pair $(\ba,\bb)\in\beta\cap (\rel''_I)^2$
%% generates $\beta$. 
Therefore we may assume $\relo^C$ is linked.

We start by choosing a $\beta$-block containing elements u-maximal in $\rel''_I$ and studying some of its properties.  Let $E$ be a $\beta$-block such that $E\cap\umax(\rel''_I)\ne\eps$, and let $E'=E\cap \rel'_I$, $E^C=E\cap B^C_I$, and $E^*=E\cap \rel''_I$.\\[2mm]
(P1) Since $\relo^C$ is linked, $C$ is an as-component in a block of the link congruence of $B'_{\ell+1}$ with respect to $\relo'$. By Proposition~\ref{pro:umax-rectangular} $\umax(B^C_I)\tm C\sse\relo'$ and therefore $\umax(B^C_I)\tm C\sse\relo^C$.\\[2mm]
(P2) If $\ba\in\umax(E^*)$ and $\ba\sqq^{as}_{E'}\bb$, then $\bb\in\umax(B^C_I)\cap E^*$ and $\ba\sqq^{as}_{E^*}\bb$. Indeed, by the Maximality Lemma~\ref{lem:to-max}(3) $\bb\in E^C$ and $\ba\sqq^{as}_{E^C}\bb$. Choose $\bc\in\umax(B^C_I)$ with $\bb\sqq^{asm}_{B^C_I}\bc$. By the previous property $\umax(B^C_I)\sse\rel''_I$ and so $\bc\in\rel''_I$. As $\umax(\rel''_I)\cap E\ne\eps$, pick $\bd\in\umax(\rel''_I)\cap E$. Because $\rel''_I$ is an algebra, it has a unique u-maximal component and so $\bc\sqq^{asm}_{\rel''_I}\bd$. Finally, $E^*$ is also an algebra and has a unique u-maximal component, hence, $\bd\sqq^{asm}_{E^*}\ba$. This proves $\bc\sqq^{asm}_{B^C_I}\bb$ and therefore $\bb\in\umax(B^C_I)\sse\rel''_I$, so $\bb\in E\cap\rel''_I=E^*$. The same argument shows that every element on the as-path from $\ba$ to $\bb$ is in $E^*$.\\[2mm]
(P3) $\umax(E^*)\sse\umax(B^C_I)$, as it follows from (P2) assuming $\bb=\ba$.\\[2mm]
(P4) $\amax(E^*)$ is a union of as-components of $E'$. Indeed, for $\ba\in\amax(E^*)$, let $\bb\in\amax(E')$ be such that $\ba\sqq^{as}_{E'}\bb$. By (P2) $\bb\in E^*$ and $\ba\sqq^{as}_{E^*}\bb$; hence $\bb\sqq^{as}_{E^*}\ba$, proving $\ba\in\amax(E')$. Thus, every as-component of $E^*$ is a subset of an as-component of $E'$. The reverse inclusion follows from (P2).\\[2mm]
(P5) Since $E^*$ and $E'$ are both subalgebras, $E^*\sse E'$, and $E^*$ contains an element u-maximal in $E'$, we have $\umax(E^*)\sse\umax(E')$ by Lemma~\ref{lem:u-max-congruence}.

\smallskip

%% By the Maximality Lemma~\ref{lem:to-max}(4) $\amax(E^C)$ is a union of as-components of $E'$. Indeed, let $\ba_0\in E^C$ and let \ba\in\relov$ be such that $\pr_I\ba=\ba_0$ and $\ba[\ell+1]\in C$; let also $\bb_0\in E'$ with $\ba_0\sqq^{as}\bb_0$ in $E'$. Then by the Maximality Lemma~\ref{lem:to-max}(4) there is $\bb\in\relov$ such that $\pr_I\bb=\bb_0$ and $\ba\sqq^{as}\bb$ in $\relov$. In particular, $\ba[\ell+1]\sqq^{as}\bb[\ell+1]$ implying $\bb[\ell+1]\in C$. Also, by Proposition~\ref{pro:umax-rectangular} applied to the subalgebra generated by $\relo^C$, since $\relo^C$ is linked and $\umax(E^C)\sse\umax(B^C_I)$, we have $\umax(E^C)\tm C\sse\relo^C$, and therefore $\umax(E^C)=\umax(E^*)\sse\umax(E')$. The last inclusion here is because $E^C$ contains some as-maximal elements of $E'$. In particular, $\amax(E^*)$ is a union of as-components of $E'$. 

First we prove that there are polynomials with images from an 
as-component of $E^*$ for both
conditions (Q1),(Q2) for $\ov\beta',\ov D$.

\medskip

{\sc Claim 2.}
For any $\ba,\bb\in E^*$, $\ba',\bb'\in \rel''_I$ such that $\ba,\bb$ belong to the same as-component 
of $E^*$ and $\Cgg{\rel_I}{\{(\ba',\bb')\}}=\beta$, there exists a polynomial 
$f$ from $\cU_{\ov D}$ with 
$f(\{\ba',\bb'\})=\{\ba,\bb\}$. Moreover, if $\al\prec\beta$ polynomial $f$ can 
be chosen from $\cU(\gm,\dl,\ov D)$. 

\medskip

Consider the relation $\rela$, a subdirect product of 
$\pr_I\rel\tm\pr_I\rel\tm\zA_{\ell+1}\tms\zA_n$, produced from $(\ba',\bb',\ba_0)$, where 
$\ba_0$ is a fixed tuple from $\pr_{\{\ell+1\zd n\}}\rel''$, as follows: 
$$
\rela=\{(f(\ba'),f(\bb'),f(\ba_0))\mid \text{$f$ is a unary polynomial of $\rel$}\},
$$ 
and for (Q2) we also assume that polynomials in the definition of $\rela$ 
satisfy $f(\dl)\sse\gm$.
By Lemma~\ref{lem:good-polys}(1,2) 
$\rela$ is a subalgebra, and, 
in particular it contains all the tuples of the form $(\pr_I\bb,\pr_I\bb,\bb[\ell+1]\zd\bb[n])$ for 
$\bb\in\rel$. Let $\rela'=\rela\cap\ov B$, and $\rela''=\rela\cap\ov D$. 
Every tuple from $\rela'$ or from $\rela''$ corresponds to a $\ov B$- or 
$\ov D$-preserving polynomial. In particular, $\pr_{[\ell]}\rela'=\pr_{\{\ell+1\zd2\ell\}}\rela'=\rel'_I$. Also, observe that $\rela''=\{\bb\in\rela'\mid \bb[2\ell+1]=c\}$. Therefore it suffices to prove that 
$(\ba,\bb)\in\pr_{[2\ell]}\rela''$ or equivalently, that $(\ba,\bb,c)\in\relp=\pr_{[2\ell+1]}\rela'$. 

Let $F$ be the as-component of $E^*$ containing 
$\ba,\bb$; by (P4) $F$ is also an as-component of $E'$. 
As $\rel$ is chained with respect to $\ov\beta,\ov B$, by Lemma~\ref{lem:good-polys}(3) for any $\ba_1,\bb_1$ from the same as-component of $E'$ there is a $\ov B$-preserving polynomial $f$ such that $f(\ba')=\ba_1$ and $f(\bb')=\bb_1$. Therefore  
$F^2\sse\pr_{[2\ell]}\relp$. Also, $(\be,\be)\in\pr_{[2\ell]}\rela''$ and $(\be,\be,d)\in\relp$ for any $\be\in F$ and $d\in C$, since $F\tm C\sse\relo^C$ by (P1) and (P3), and because $F\sse\umax(E^*)$. 
Let $\relp^* =\relp\cap ((E^*)^2 \times B'_{\ell+1})$. We have
$F^2 \subseteq \pr_{[2\ell]} \relp^* \leq (E^*)^2$ and $F$ is an as-component
of $E^*$. Hence $F^2$ is an as-component of $\pr_{[2\ell]} P^*$ by
Lemma~\ref{lem:as-product} and Corollary~\ref{cor:subalgebra-path}. Similarly, $C \subseteq \pr_{2\ell+1}\relp^* \leq
B'_{\ell+1}$, so $C$ is an as-component of $\pr_{2\ell+1}\relp^*$ by
Corollary~\ref{cor:subalgebra-path}.

Since $\{(\be,\be)\}\times C \subseteq \relp$ for any $\be\in F$, we get
$(F^2 \times C) \cap\relp^* \neq \emptyset$ and so $\{(\be,\be)\} \times C \subseteq\relp^*$
for some $(\be,\be) \in \pr_{[2\ell]}\relp^*$.
Hence by Proposition~\ref{pro:linkage-rectangularity}(2),
$F^2 \times C \subseteq \relp^*$, in particular, $(\ba,\bb,c)\in\relp$, as required.
Claim~2 is proved.

%% Therefore considering $\relp$ as a subdirect product of $\pr_{[2\ell]}\relp$ and $B'_{\ell+1}$, the as-component $C$ belongs to a block of the link congruence. Therefore by Proposition~\ref{pro:umax-rectangular} $\umax(\relp^{-1}[C])\tm C\sse\relp$. By the Maximality Lemma~\ref{lem:to-max}(3) $F\sse\amax(\relp^{-1}[C])$, we show that $F\sse\umax(\relp^{-1}[C])$ (which is not obvious because $\relp^{-1}[C]$ is not necessarily a subalgebra). Observe that $\reli=\relp^{-1}[c]$ is a subalgebra and by what is just proved $\umax(\reli)=\umax(\relp^{-1}[C])$. Next, $F\sse\umax(E^*)\sse\umax(\pr_{[\ell]}\reli)$ by (P3). By the Maximality Lemma~\ref{lem:to-max}(5) for any $\ba_1\in F$ there is $\bb_1\in\umax(\pr_{\{\ell+1\zd 2\ell\}}\reli)$ such that $(\ba_1,\bb_1)\in\umax(\reli)$. Since $\ba'\eqc\beta\bb'$, it also holds that $\ba_1\eqc\beta\bb_1$, which means that $\bb_1\in E^*$. Thus, $\umax(\reli\cap(E^*)^2)\sse\umax(\reli)$ by Lemma~\ref{lem:u-max-congruence}; in particular $F^2\sse\umax(\reli)$. Thus, $F^2\tm C\sse\relp$, in particular, $(\ba,\bb,c)\in\relp$, as required.
%% Claim~2 is proved.
%% As $F^2\sse\pr_{[2\ell]}\relp$, every pair from $F^2$ is as-maximal in $\pr_{[2\ell]}\relp$. Moreover, by the Maximality Lemma~\ref{lem:to-max}(4)  $\relp\cap(F^2\tm C)$ is subdirect on $F^2$ and  $C$. Hence $F^2$ also belongs to a block of the link congruence. By Proposition~\ref{pro:linkage-rectangularity}(1) 

\medskip

Now we extend the result of Claim~2 to pairs from $\umax(E^*)$. We prove 
the result in two steps. First, we show that for any $\ba',\bb'\in \rel''_I$ 
such that $\Cgg{\rel_I}{\{(\ba',\bb')\}}=\beta$ and
any $\ba,\bb\in\umax(E^*)$ there is a sequence of $\ov B$-preserving polynomials 
$\vc fk$ such that $f_1(\{\ba',\bb'\})\zd$ $ f_k(\{\ba',\bb'\})\sse E^*$ form a chain 
connecting $\ba$ and $\bb$, and $f_i(c)\in C$ for $i\in[k]$. Then we prove 
that $\vc fk$ can be chosen in such a way that 
$f_1(\{\ba',\bb'\})\zd$ $ f_k(\{\ba',\bb'\})\sse E^*$ and $f_1(c)=\dots=f_k(c)=c$.
Clearly, it suffices to prove in the case when $\bb$ is as-maximal in $E^*$, since we can link $\ba$ and $\bb$ to an as-maximal element $\bb''$ and then concatenate the two chains.
We will also observe that in both cases the polynomials $\vc fk$ can be
chosen from $\cU(\gm,\dl,\ov B)$ when necessary.

By the assumption there are $\ba=\ba_1,\ba_2\zd \ba_k=\bb$, $\vc \ba k\in E'$ and
polynomials $\vc f{k-1}$ from $\cU_{\ov B}$ (or from $\cU(\gm,\dl,\ov B)$  
when $\al\prec\beta$) such that $f_i(\{\ba',\bb'\})=\{\ba_i,\ba_{i+1}\}$,
and also $f_i(c)\in B'_{\ell+1}$. We need to show that $\vc \ba{k-1}$ and $\vc f{k-1}$
can be chosen such that $f_i(c)\in C$. Choose $\ba'',\bb''\in\rel''$
such that $\pr_I\ba''=\ba,\pr_I\bb''=\bb$ and $\ba''[\ell+1]=\bb''[\ell+1]=c$. Now let 
$g_i(x)=\maj(\ba'',f_i(x),\ba'')$ and $h_i(x)=\maj(\ba'',\bb'',f_i(x))$, where $\maj$ is a quasi-majority operation, see Theorem~\ref{the:pseudo-majority}. 
By Lemma~\ref{lem:good-polys}(2) 
$g_i,h_i$ are from $\cU_{\ov B}$ and 
from $\cU(\gm,\dl,\ov B)$ whenever $\vc fk$ are from $\cU(\gm,\dl,\ov B)$. 
Also, for each of them either $\{\bb_i,\bb_{i+1}\}=g_i(\{\ba',\bb'\})$ 
($\{\bc_i,\bc_{i+1}\}=h_i(\{\ba',\bb'\})$) is from 
$T(\al,\{(\ba',\bb')\},\cU^*)$, $\cU^*\in\{\cU_{\ov B},\cU(\gm,\dl,\ov B)\}$, 
or $g_i(\ba')=g_i(\bb')$ (respectively, $h_i(\ba')=h_i(\bb')$). The polynomials
$g_i,h_i$ satisfying the first option form a sequence of 
pairs connecting $\ba$ with $\maj(\ba,\bb,\ba)$ --- by pairs
of the form $\{\bb_i,\bb_{i+1}\}$, --- and $\maj(\ba,\bb,\ba)$ with $\maj(\ba,\bb,\bb)$
--- by pairs of the form $\{\bc_i,\bc_{i+1}\}$. Also, by 
Theorem~\ref{the:pseudo-majority} 
$\maj(\ba,\bb,\bb)$ belongs to the 
as-component of $E^*$ (and therefore of $E'$) containing $\bb$.
Therefore by Claim~2 this sequence of polynomials and pairs can be 
continued to connect $\maj(\ba,\bb,\bb)$ to $\bb$. Finally, by the same theorem
$g_i(c)=\maj(c,f_i(c),c)\in C$ and $h_i(c)=\maj(c,c,f_i(c))\in C$.

For the second step we assume that $\ba$ and $\bb$ are connected with
pairs $\{\ba_i,\ba_{i+1}\}$, $i\in[k-1]$ witnessed by polynomials 
$f_i$ from $\cU_{\ov B}$ (or from $\cU(\gm,\dl,\ov B)$ when needed) 
such that $c_i=f_i(c)\in C$. We need to show that $f_i$ can be chosen 
such that $f_i(c)=c$. Suppose that $c_i\ne c$ for some $i\in[k-1]$.
Since $c_i$ and $c$ belong to the same as-component, there is an
as-path $c_i=d_1\zd d_\ell=c$ in $C$. Suppose that there is 
a sequence of pairs $\{\bb_j,\bb_{j+1}\}=g_j(\{\ba',\bb'\})$, $\bb_1=\ba$ and 
$\bb_k=\bb$, for 
some polynomials $g_j\in\cU_{\ov B}$ (or from $\cU(\gm,\dl,\ov B)$ 
when needed), $j\in[k-1]$, such that $g_j(c)=c$ whenever $f_j(c)=c$, 
and $g_i(c)=d_t$. We will show that there are also pairs $\{\bb'_j,\bb'_{j+1}\}=g'_j(\{\ba',\bb'\})$ 
for some polynomials $g'_j\in\cU_{\ov B}$ (or from $\cU(\gm,\dl,\ov B)$ 
when needed) such that $\bb'_1=\ba$ and $\bb'_k$ is in 
the as-component containing $\bb$, $g'_i(c)=d_{t+1}$
and $g'_j(c)=g_j(c)$ for $j\ne i$.

Note that, since $\ba,\bb\in E^*$ and $E^*$ is a subalgebra of $\rel_I$, the subalgebra $\Sgg{\rel_I}{\ba,\bb}$ is a subalgebra of $E^*$. As is easily seen, it suffices to find a ternary term operation $p$
such that $p(\ba,\ba,\bb)$ belongs to the as-component of $E^*$ containing $\bb$,
and $p(d_{t+1},d_t,d_t)=d_{t+1}$. Indeed, if such a term operation exists, 
then we set $g'_j(x)=p(\ba''_j,\ba''_j,g_j(x))$, where $\ba''_j\in\relov$ is a tuple such that $\pr_I\ba''_j=\ba$ and $\ba''_j[\ell+1]=g_j(c)$, for $j\in[k-1]-\{i\}$, and 
$\{\bb'_j,\bb'_{j+1}\}=g'_j(\{\ba',\bb'\})$. Such a tuple $\ba''_j$ exists because $\umax(E^*)\tm C\sse\relo^C$ by (P1),(P3). We have $g'_1(\ba')=p(\ba,\ba,g_1(\ba'))=\ba$ (or $g'_1(\bb')=p(\ba,\ba,g_1(\bb'))=\ba$ if $\ba=g_1(\bb')$)
and $g'_j(c)=p(g_j(c),g_j(c),g_j(c))=g_j(c)$. Finally, since
$g'_{k-1}(\bb')=p(\ba,\ba,\bb)$ belongs to the as-component containing $\bb$,
we can use Claim~2 as before to connect $p(\ba,\ba,\bb)$ to $\bb$. For $g'_i$ we set
$g'_i(x)=p(\ba^\dg,\ba^\ddg,g_i(x))$ where $\ba^\dg,\ba^\ddg\in\rel'$ are such that
$\pr_I\ba^\dg=\pr_I\ba^\ddg=\ba$ and $\ba^\dg[\ell+1]=d_{t+1},\ba^\ddg[\ell+1]=d_t$. Note that 
such $\ba^\dg,\ba^\ddg$ exist, because $\umax(E^C)\tm C\sse\relo'$. It follows 
from the assumption about $p$ that $g'_i$ is as required.

If $d_t\le d_{t+1}$, then $p(x,y,z)=z\cdot x$ fits the requirements. 
If $d_td_{t+1}$ is an affine edge, consider the relation 
$\relo^\dg\sse\rel_I\tm\zA_{\ell+1}$ generated by $\{(\ba,d_{t+1}),(\ba,d_t),(\bb,d_t)\}$.
Let $\relp^\dg=\Sgg{\rel_I}{\ba,\bb}$ and $\zG=\Sgg{\zA_{\ell+1}}{d_t,d_{t+1}}$; then 
$\relp^\dg\tm\{d_t\},\{\ba\}\tm\zG\sse\relo^\dg$. By 
Lemma~\ref{lem:as-rectangularity}, 
as $d_td_{t+1}$ is a thin affine 
edge, $\umax(\relp^\dg)\tm\{d_{t+1}\}\sse\relo^\dg$.
There is $\bb^\dg$ with $\bb\sqq^{as}\bb^\dg$ in $\relp^\dg$ such that $\bb^\dg\in\umax(\relp^\dg)$.
Therefore there is a term operation $p$ with $p(\ba,\ba,\bb)=\bb^\dg$  
and $p(d_{t+1},d_t,d_t)=d_{t+1}$, as required.
\end{proof}

%%%%%%%%%%%%%%%%%%%%%%%%%%%%%%%%%%
%%%%%%%%%%%%%%%%%%%%%%%%%%%%%%%%%%
\section*{Declarations}

\subsection*{Data availability}
Data sharing not applicable to this article as datasets were neither generated nor analyzed.

\subsection*{Compliance with ethical standards}
The author is a member of the Editorial Board of Algebra Universalis. Apart from this the author declares that he has no conflict of interest.

\bibliographystyle{plain}
%% \bibliography{one}

\end{document}